\documentclass[10pt,twoside,a4paper]{amsart}

\usepackage[english]{babel}
\usepackage{a4wide}
\usepackage{graphicx}
\usepackage{amsmath,amssymb,amsthm, amsgen}
\usepackage{bbm}
\usepackage[ansinew]{inputenc}
\usepackage{listings}
\usepackage{color}
\usepackage[usenames,dvipsnames]{xcolor}
\usepackage{rotating}
\usepackage{hhline}
\usepackage{multirow}
\usepackage{enumerate,enumitem}
\usepackage[bookmarks]{}
\usepackage{amsfonts}   % if you want the fonts
\usepackage{dsfont}
\usepackage{tikz}
\usetikzlibrary{patterns}
\usepackage{longtable}
\usepackage{array}
\usepackage{booktabs}

\allowdisplaybreaks

\newtheorem{thm}{Theorem}[section]

\newtheorem{lem}[thm]{Lemma}

\theoremstyle{definition}
\newtheorem{rem}[thm]{Remark}

\usepackage{fancyhdr}

\newcommand{\floor}[1]{\left\lfloor #1 \right\rfloor}
\newcommand{\ceil}[1]{\left\lceil #1 \right\rceil}
\newcommand{\norm}[2]{ \| #1 \|_{ #2 } }

\newcommand{\acc}[1]{ \{ #1 \} }

\newcommand{\bigbhaa}[1]{\big[ #1 \big]}
\newcommand{\bighaa}[1]{\big( #1 \big)}

\newcommand{\bignorm}[2]{\big\| #1 \big\|_{ #2 } }
\newcommand{\bigabs}[1]{\big| #1 \big|}
\newcommand{\bigaccv}[2]{\big\{ #1 \, \big| \, #2 \big\}}

\newcommand{\Bigbhaa}[1]{\Big[ #1 \Big]}

\newcommand{\Biggbhaa}[1]{\Bigg[ #1 \Bigg]}
\newcommand{\Bighaa}[1]{\Big( #1 \Big)}
\newcommand{\Bigaccv}[2]{\Big\{ #1 \, \Big| \, #2 \Big\}}

\newcommand{\bigghaa}[1]{\bigg( #1 \bigg)}
\newcommand{\eps}{\epsilon}

\newcommand{\lracc}[1]{\left\{ #1 \right\}}

\newcommand{\R}[1]{\mathbb{R}^{ #1 }}
\newcommand{\N}[1]{\mathbb{N}^{ #1 }}
\newcommand{\Np}[1]{\mathbb{N}_+^{ #1 }}

\newcommand{\Dom}{\operatorname{Dom}}

\newcommand{\Qn}{Q_n}

\newcommand{\signum}{\operatorname{\text{sgn}}}
\newcommand{\ord}{\operatorname{\text{ord}}}
\def\Xint#1{\mathchoice
   {\XXint\displaystyle\textstyle{#1}}%
   {\XXint\textstyle\scriptstyle{#1}}%
   {\XXint\scriptstyle\scriptscriptstyle{#1}}%
   {\XXint\scriptscriptstyle\scriptscriptstyle{#1}}%
   \!\int}
\def\XXint#1#2#3{{\setbox0=\hbox{$#1{#2#3}{\int}$}
     \vcenter{\hbox{$#2#3$}}\kern-.5\wd0}}

\def\dashint{\Xint-}

\def\Dn{\mathcal{D}_n}

\newcommand{\framed}[2]{\begin{center}\framebox{\begin{minipage}{#1} #2 \end{minipage}}\end{center}}

\newcommand\weakto{\rightharpoonup}

\newcommand{\xweakto}[1]{ \stackrel{ #1 }{\rightharpoonup} }
\newcommand{\xto}[1]{ \xrightarrow{ #1 } }

\makeatletter
\DeclareRobustCommand{\cev}[1]{%
  \mathpalette\do@cev{#1}%
}
\newcommand{\do@cev}[2]{%
  \fix@cev{#1}{+}%
  \reflectbox{$\m@th#1\vec{\reflectbox{$\fix@cev{#1}{-}\m@th#1#2\fix@cev{#1}{+}$}}$}%
  \fix@cev{#1}{-}%
}
\newcommand{\fix@cev}[2]{%
  \ifx#1\displaystyle
    \mkern#23mu
  \else
    \ifx#1\textstyle
      \mkern#23mu
    \else
      \ifx#1\scriptstyle
        \mkern#22mu
      \else
        \mkern#22mu
      \fi
    \fi
  \fi
}

\makeatother

\title{Asymptotic analysis of boundary layers in a 
repulsive particle system}

\author{Cameron L. Hall}
\address{C. L. Hall \\ Mathematical Institute \\ University of Oxford
\\ Andrew Wiles Building \\ Radcliffe Observatory Quarter
\\ Woodstock Road \\ Oxford \\ OX2 6GG \\ United Kingdom}

\author{Thomas Hudson}
\address{T. Hudson \\ \'Ecole des Ponts ParisTech, CERMICS \\
6 et 8, Avenue Blaise Pascal \\ 77455 Champs-sur-Marne \\ France}

\author{Patrick van Meurs}
\address{P. van Meurs \\ Faculty of Mathematics and Physics \\ Kanazawa University, Kakuma \\ 920-1192, Kanazawa \\ Japan}

\begin{document}
\begin{abstract}
This paper studies the boundary behaviour at mechanical equilibrium at the ends of a finite interval of a class of systems of interacting particles with monotone decreasing repulsive force. Our setting covers pile-ups of dislocations, dislocation dipoles and dislocation walls. The main challenge is to control the nonlocal nature of the pairwise particle interactions. Using matched asymptotic expansions for the particle positions and rigorous development of an appropriate energy via $\Gamma$--convergence, we obtain the equilibrium equation solved by the boundary layer correction, associate an energy with an appropriate scaling to this correction, and provide decay rates into the bulk.
\end{abstract}

\maketitle

\noindent \textbf{Keywords}: particle system, boundary layer, discrete-to-continuum asymptotics, matched asymptotic expansions, $\Gamma$-convergence. 
\\
\smallskip 
\textbf{MSC}: 
74Q05, % Homogenization in equilibrium problems 
74G10, % Analytic approximation of solutions (perturbation methods, asymptotic methods, series, etc.)
41A60. % Asymptotic approximations, asymptotic expansions (steepest descent, etc.)

\section{Introduction}

A wide variety of physical and mathematical phenomena may be modelled as a system of 
interacting identical particles. One of the simplest examples of
such an application is a collection of electrostatically--charged 
classical particles, but other examples include atoms in
a fluid \cite{LennardJones24}, dislocations in a crystalline solid 
\cite{HirthLothe}, Ginzburg--Landau vortices in a superconductor 
\cite{BethuelBrezisHelein}, spin states in an atomic lattice 
\cite{Nussinov2015}, eigenvalues of random matrices \cite{Wigner1955,Dyson1962}, or simply a collection of hard spheres
\cite{MH53}. A core challenge in studying such particle systems is 
to identify the features of the thermodynamic equilibrium in a system where the number of particles is very large. At low temperatures, this is closely related to finding the configurations with the lowest total potential energy, i.e.~the \emph{mechanical equilibria}.

Typically, low potential energy configurations in large particle 
systems exhibit \emph{crystallisation} phenomena, i.e.~particles
arrange themselves into a regular structure with a slowly--varying density. However, it is often difficult to determine a similarly
detailed description of the particle behaviour in regions where the density varies rapidly, for
example at a free surface, or close to a rigid confining
structure. Such boundary properties and other effects of finite system size are a significant theme in current 
scientific research \cite{BraidesCicalese2007,Voskoboinikov2009,Hall2010a,
ScardiaSchloemerkemperZanini2011,Ivanov2013,Wennberg2013,PetracheSerfaty2014,
Zschocke2015} since it is through boundary interactions
that a large number of physical processes take place, some important examples being contact \cite{G80}, catalysis \cite{TT14} and crystal growth \cite{RS06}.

This work contributes to this body of research by studying a simple model for a system of particles
confined to a finite interval, and obtains concrete 
mathematical results concerning the boundary behaviour at mechanical equilibrium, advancing some of the mathematical techniques currently available to study such boundary effects in the process. In the model considered, particles are assumed to
interact via a repulsive pair potential that decays as the distance between particles increases. For a sufficiently rapid decay of the repulsive interactions, we obtain an asymptotic description of the boundary behaviour by developing a matched asymptotic expansion for the particle positions at equilibrium, and an asymptotic 
representation of the minimal energy via the technique of 
$\Gamma$--convergence. A particular challenge for this task is that we do not rely on a finite interaction range, and instead include general long-range interactions between particles. 

Our study takes place in the context of a variety of recent
mathematical results aiming to better understand surface
effects in similar particle systems, notably \cite{Hall2010, GarronivanMeursPeletierScardia14ArXiv} in the setting of dislocation pile-ups. While \cite{Hall2010} focuses on the case in which the interaction potential is homogeneous, and \cite{GarronivanMeursPeletierScardia14ArXiv,Hall2011} studies boundary layers in a continuum model for the particle density, our contribution is the derivation of a \emph{discrete} description of the boundary layer for a general class of interaction potentials. Two particular examples that we have in mind are pile-ups of dislocation walls \cite{GarronivanMeursPeletierScardia14ArXiv} and pile-ups of dislocation dipoles \cite{Hall2010}.

\subsection{Setting}
\label{sec:Setting}
We suppose that $n+1$ identical particles are confined to lie in 
the interval $[0,n]$, and all pairs of particles mutually interact via a potential $V : \R{} \to [0,+\infty]$, which is a function of  the inter-particle distance.
Labelling the position of particle $i$ as $\chi(i)$, the total
potential energy of the system is thus
\begin{equation*}
  \sum_{k=1}^n \sum_{j=0}^{n-k} V \big(\chi(j+k) - \chi(j)\big).
\end{equation*}

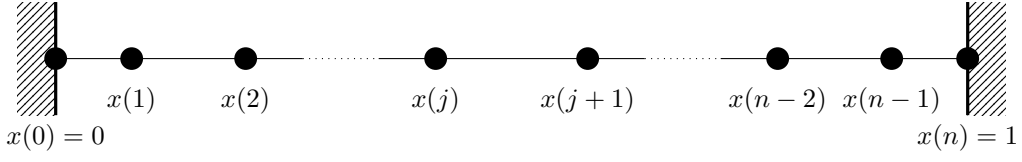
\begin{figure}[t]
\centering
\begin{tikzpicture}[scale=0.5, >= latex]
	\draw (0,-1.5) node [below] {$x(0) = 0$}; 
	\draw (2,-0.5) node [below] {$x(1)$};  
	\draw (5,-0.5) node [below] {$x(2)$};  
	\draw (10,-0.5) node [below] {$x(j)$};
	\draw (14,-0.5) node [below] {$x(j+1)$};
	\draw (19,-0.5) node [below] {$x(n-2)$};
	\draw (22,-0.5) node [below] {$x(n-1)$};
	\draw (24,-1.5) node [below] {$x(n) = 1$};
	\draw[very thick] (0,-1.5) -- (0,1.5);
	\draw (0,0) -- (6.5,0);
	\draw[dotted] (6.5,0) -- (8.5,0);
	\draw (8.5,0) -- (15.5,0);
	\draw[dotted] (15.5,0) -- (17.5,0);
	\draw (17.5,0) -- (24,0);
	\draw[very thick] (24,-1.5) -- (24,1.5);
	\fill[pattern = north east lines] (-1, -1.5) rectangle (0, 1.5);
	\fill[pattern = north east lines] (24, -1.5) rectangle (25, 1.5);	
	     
    \foreach \x in {0,2,5,10,14,19,22,24}
      {
      \fill (\x, 0) circle (0.3);
      }
\end{tikzpicture}
\caption{The setting of the particle system in rescaled 
coordinates.}\label{fig:PS:bdd}
\end{figure}

Since we subsequently wish to consider the system with $n$ large,
it is convenient to introduce rescaled coordinates 
$x(i):=\chi(i)/n$, so that $x(i)\in[0,1]$ for all $n$. Applying
this rescaling, we
are led to consider the following equivalent scenario (see also Figure \ref{fig:PS:bdd}), which is similar to that studied in \cite{VanMeursMunteanPeletier14}:
\begin{gather*} %\notag
  \Dn := \{x\in[0,1]^{n+1}\,|\, 0 =: x(0) \leq x(1) \leq \ldots 
  \leq x(n-1) \leq x(n) := 1 \}, \\
  E_n : \Dn \to [0, +\infty], 
  \qquad E_n(x) := \frac1n \sum_{k = 1}^n \sum_{j = 0}^{n-k}
   V\big(n \big[x(j+k) - x(j)\big] \big).
   %\label{eq:En_definition}
\end{gather*}
Here, $\Dn$ represents the set of all possible valid
positions, and $E_n(x)$ is the average energy per particle in the
system due to the interactions with all the other particles in the
configuration described by $x$. Our basic assumptions on the potential $V : \R{} \to [0, +\infty]$ are:
\medskip
\framed{0.9\textwidth}{\begin{description}[leftmargin=*]
  \item[(Reg)] $V:\R{}\setminus\{0\}\to(0,+\infty)$ is even 
    and $C^2$;\vspace{1mm}
  \item[(Sing)] $V(x)\to V(0)=+\infty$ as $x\downarrow 0$ or $x\uparrow0$;\vspace{1mm}
  \item[(Cvx)] for each $x\in(0,+\infty)$, there exists 
    $\lambda(x)>0$ such that $V$ is
    $\lambda(x)$--convex on $(0,x)$, i.e. 
    \begin{equation*}
      \inf_{(0,x)} V'' \geq \lambda(x) > 0; %\qquad\text{for all }s\in(0,x);
    \end{equation*}
  \item[(Dec)] $V(x),V'(x)\to0$ as $|x|\to\infty$, and
    there exists $a>1$ and constants 
    $c_\delta$ such that for any $\delta>0$,
    \begin{equation*}
      V''(x)\leq c_\delta |x|^{-a-2}
      \qquad\text{for any }x\in\R{}\setminus(-\delta,\delta).
    \end{equation*}
\end{description}}\medskip

\noindent
Figure \ref{fig:V} shows a typical graph for such a potential $V$,
and a prototypical example is $V(x)=|x|^{-a}$ with $a>1$.
We note the following immediate consequences of our basic 
assumptions.
\begin{itemize}
  \item As $V$ is non-negative, $E_n(x) \geq 0$.
  \item Together, {\bf(Cvx)} and {\bf(Dec)} demonstrate that
$V''$ is integrable on $\R{}\setminus(-\delta,\delta)$,
so by applying the Fundamental Theorem of Calculus on the
interval $(x,+\infty)$, we find that there exist constants
$c_\delta',c_\delta''>0$ and $a>1$ such that
\begin{equation}
  0 > V'(x) \geq -c'_\delta |x|^{-a-1} \quad\text{and}\quad
  0 < V(x) \leq c''_\delta |x|^{-a} \quad \text{for } x > \delta.
  \label{for:ests:V:Vprime}
\end{equation}
    Since
    $V'(x) < 0$ for $x\in(0, \infty)$, the 
    interactions are repulsive, and hence we must `confine' the 
    particles in order to ensure that they remain within a compact 
    set: here, our choice is to enforce
    confinement by fixing the outermost particles.
    
\item As observed in \cite{Geers2013, VanMeursMunteanPeletier14}, {\bf(Cvx)} implies that $E_n$ is strictly convex on
$\Dn$. Moreover, {\bf(Sing)} makes the energy infinite on the set
\begin{equation*}
  \partial \Dn := \big\{x\in\Dn\,\big|\,x(i-1) = x(i)\text{ for some }i\in\{1,\ldots,n\}\big\}.
\end{equation*} Therefore, $E_n$ has a unique minimiser, which is contained in $\Dn\setminus \partial\Dn$. Since $E_n$ is differentiable on this set, the minimiser satisfies the
following force balance:
\begin{equation}
  \frac{\partial E_n(x)}{\partial x(i)} =
  -\sum_{\substack{ k=0 \\ k \neq i} }^n 
  V'\big(n[x(k)-x(i)]\big)=0,\qquad i=1,\ldots,n-1.
  \label{eq:equilibrium}
\end{equation}
\end{itemize}

\begin{figure}[t]
\centering
\begin{tikzpicture}[scale=0.6, >= latex]
%\draw[very thin,color=gray] (-4.9,0) grid (4.9,4.9);
\draw[->] (-5.2,0) -- (5.2,0) node[right] {$r$};
\draw[->] (0,0) -- (0,5.2);
\draw[thick] (0.018,5) .. controls (0.09,1) and (1,0.176) .. (5,0.04);
\draw[thick] (-0.018,5) .. controls (-0.09,1) and (-1,0.176) .. (-5,0.04);
%\draw (0.2, 3.5) node[right] {$\sim |\log r|$};
%\draw (4.4, 0.7) node {$\sim r e^{-2r}$};
\draw (-2,5.2) node[below] {$V(r)$};
\end{tikzpicture}
\caption{Typical profile for the interaction potential $V$.}\label{fig:V}
\end{figure}
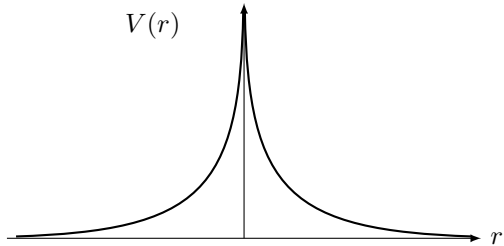

\subsection{Main results}
\label{sec:Results}
In \cite{Geers2013, VanMeursMunteanPeletier14}, it was shown that $E_n$ $\Gamma$--converges, and the limit energy has a unique minimiser corresponding to a uniform density of particles. Since the topology used to obtain these previous $\Gamma$--convergence results (i.e.~the narrow or weak topology of measures) cannot detect `microscopic' variations in particle density, here, we strengthen these convergence results by obtaining a finer characterisation of the minimiser of $E_n$. In particular, we seek
to describe the \emph{boundary layers} which appear at both ends of the bounded interval; see Figure \ref{fig:x:on:01} for a numerical illustration.
To do so, we use two different approaches: formal asymptotic analysis and $\Gamma$--development of the energy $E_n$. Our formal analysis gives us a detailed description of the equilibrium particle positions in the boundary layers when $n$ is large enough, while the $\Gamma$--development establishes a precise notion of convergence for the particle positions and boundary-layer energy as $n \to \infty$.

In \S\ref{sec:Formal1}, formal 
asymptotic analysis is used both to obtain the equations of equilibrium in the bulk, and to show
that the correct scaling for the boundary layer is in terms of the particle positions
$\chi(i) = n x(i)$, with no intermediate regime between the bulk and discrete scaling. To carry out this analysis, we require a slightly stronger condition on the differentiability of $V$, and so in this section, as well as the basic assumptions detailed in \S\ref{sec:Setting}, we require the additional regularity assumption\medskip
\begin{center}
\framed{0.7\textwidth}{\begin{description}[leftmargin=*]
\item[(Reg+)] $V \in C^3(0, \infty)$, with $| V''' (x) | \lesssim |x|^{-a-3}$ for $|x| > 1$.
\end{description}}\end{center}\medskip\noindent
In \S\ref{sec:Formal1-BL}, this permits us to obtain the result that the particle positions $\chi(i)$ in the boundary layer approximate the solution of the following infinite system of equations as $n \to \infty$:
\begin{equation} \label{intro:eqn:BL}
  \left\{ \begin{aligned}
    &0 = \sum_{\substack{k=0 \\ k \neq i}}^{\infty} V'\big[\chi(i) - \chi(k) \big], 
    &&i = 1, 2, 3, \ldots, \\
    &\chi(i) - \chi(i-1) = 1 + o(1), 
    &&\text{as } i \to \infty.
  \end{aligned} \right.
\end{equation}

In \S\ref{sec:Formal2}, we again use formal methods to obtain 
further terms in an asymptotic expansion of both the boundary layer
and the bulk behaviour in the particular case $V(x)=|x|^{-a}$, 
which is the prototypical potential satisfying the assumptions
detailed in \S\ref{sec:Setting}. Using the method of matched asymptotic expansions, we fully characterise particle locations in both the bulk and the boundary layer up to errors that are asymptotically smaller than $n^{-(a-1)}$ as $n \to \infty$. To the best of our knowledge, this characterisation of higher order terms in the asymptotic expansion for particle systems with nonlocal interactions is new.

Additionally, this analysis 
allows us to characterise the behaviour of $\chi(i)$, the solution
to \eqref{intro:eqn:BL}, as $i\to\infty$ by a more detailed description of the $o(1)$ term.
When $V(x) = |x|^{-a}$, we find that
\begin{equation} \label{intro:eqn:decay}
 \chi(i) - \chi(i-1) = 1 - \frac{i^{-(a-1)}}{\zeta(a) (a^3-a)} + o[i^{-(a-1)}], \quad \text{as } i \to \infty,
\end{equation}
where $\zeta(s)$ is the Riemann zeta function. In fact, this decay behaviour of $\chi$ remains valid for more general potential $V$ satisfying {\bf(Reg)}, {\bf(Sing)} and {\bf(Cvx)} whenever the tails of $V^{(k)}$ are asymptotically equivalent to $d^k / dx^k \, (x^{-a})$ up to sufficiently large $k$ (see \eqref{eq:GenPotential-Derivs} for details). In that case, the constant $[\zeta(a) (a^3-a)]^{-1}$ takes the more general form $[ (a-1) \sum_{j=1}^{\infty} V''(j) j^2 ]^{-1}$.

%\CAMERON{In \S\ref{sec:Formal2} we also outline how our methods may be generalised from $V(x)=|x|^{-a}$ to a broader class of potentials with appropriate decay properties. In the case where $V \in C^{\infty}$ satisfies , and where the derivatives of $V$ satisfy
%\[
% \lim_{x \to \infty} x^{a+r} V^{(r)}(x) = (-1)^{r} a \cdot (a+1) \cdots (a+r-1), 
%\]
%we find that we can again fully characterise particle locations up to errors that are asymptotically smaller than $n^{-(a-1)}$ as $n \to \infty$, and that \eqref{intro:eqn:decay} can be generalised to
%\begin{equation}\label{intro:eqn:decay2}
% \chi(i) - \chi(i-1) = 1 - \frac{i^{-(a-1)}}{Z(V) (a-1)} + o[i^{-(a-1)}], \quad \text{as } i \to \infty,
%\end{equation}
%where
%\[
% Z(V) := \sum_{k=1}^{\infty} V''(k) k^2.
%\]}

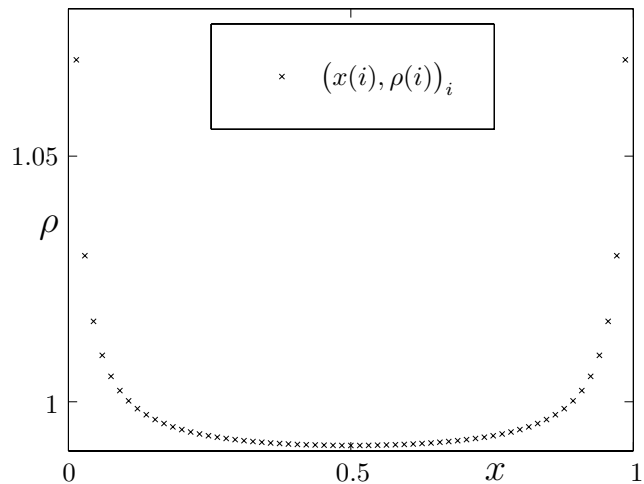
\begin{figure}[b]
\centering
\begin{tikzpicture}[scale=1.5]
\node (label) at (0,0){\includegraphics[width=3in]{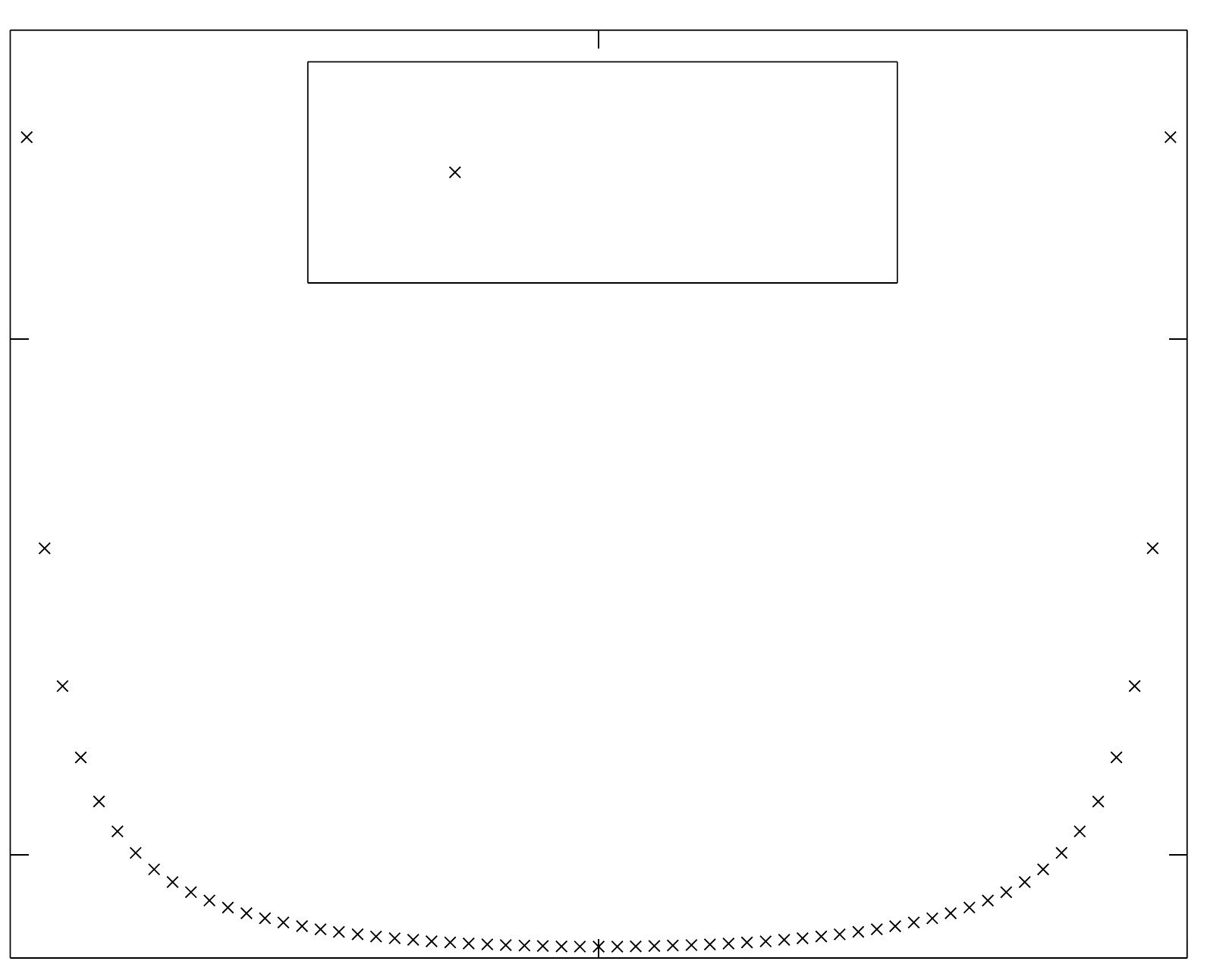}};
\draw (1.25,-2) node[below] {\LARGE $x$};
\draw (-2.5,0) node[left] {\LARGE $\rho$};
\draw (-2.49,-2) node[below] {$0$};
\draw (0,-2) node[below] {$0.5$};
\draw (2.49,-2) node[below] {$1$};
\draw (-2.5,0.63) node[left] {$1.05$};
\draw (-2.5,-1.56) node[left] {$1$};
\draw (0.3, 1.32) node {$\big( x(i), \rho(i) \big)_i$};
\end{tikzpicture}
\caption{Equilibrium configuration of $n+1 = 65$ particles $x(i)$ (i.e.~the solution to \eqref{eq:equilibrium}) for the potential $V(x) = x^{-2}$. The horizontal axis shows the domain $[0,1]$. The vertical axis measures the `discrete density' defined by $\rho(i) := 2/( n [ x(i+1) - x(i-1) ] )$.}
\label{fig:x:on:01}
\end{figure}

The formal asymptotic analysis leaves us with three questions associated with 
general case treated in \S\ref{sec:Formal1}: what is the
proper space in which to seek solutions $\chi$ to \eqref{intro:eqn:BL}; in this space, do unique solutions exist; and
if so, in what sense do solutions of \eqref{eq:equilibrium} converge to solutions of \eqref{intro:eqn:BL} as $n \to \infty$?
In \S4, we address these question by proving a $\Gamma$--convergence result. Here, the analysis rests upon a different 
strengthening of our basic assumptions, requiring
the stronger decay hypothesis\medskip
\begin{center}
\framed{0.7\textwidth}{\begin{description}[leftmargin=*]
\item[(Dec+)] $V(x),V'(x)\to0$ as $|x|\to\infty$, and
    there exists $a>\tfrac32$ and constants 
    $c_\delta$ such that for any $\delta>0$,
    \begin{equation*}
      V''(x)\leq c_\delta |x|^{-a-2}
      \qquad\text{for any }x\in\R{}\setminus(-\delta,\delta).
    \end{equation*}
\end{description}}\end{center}\medskip
Under both our basic assumptions and this extra assumption, we prove Theorem~\ref{thm:Gamma:conv}, demonstrating $\Gamma$--convergence  of the `renormalised' energy
\begin{equation} \label{intro:eq:ene:expansion}
  E_n^1(x) := n E_n(x) - \sum_{k=1}^n (n-k+1)V(k).
\end{equation}
This energy is renormalised in the sense that the subtracted
term need not be bounded as $n\to\infty$. To treat $E_n^1$ more easily and obtain a useful compactness result, we introduce a convenient change of variable, defining
\begin{equation} \label{eq:defn:eps}
  \epsilon(i) := n \big[x(i) - x(i-1)\big] - 1
  \qquad\text{ for }i=1,\ldots,n.
\end{equation}
The choice to use $\epsilon$ here as notation is due to the analogy
with the infinitesimal strain used in continuum mechanics, since
$\epsilon$ measures how far the particles deviate from being 
equispaced.

The result obtained in \S\ref{sec:Formal1}, given in \eqref{intro:eqn:BL}, suggests that $\epsilon(i) \approx \chi(i) - \chi(i-1) - 1 \to 0$ as $n \to \infty$ and $1 \ll i \ll n$. In fact, the compactness statement of Theorem~\ref{thm:Gamma:conv} states that boundedness of $E_n^1 (\epsilon)$ implies boundedness of $\epsilon$ in $\ell^2 (\N{})$, which will subsequently allow us to give a precise meaning to this statement.
%which are given by
%\begin{equation*}
%  \epsilon^{n,1/2}(i)
%  := \begin{cases}
%    \epsilon(i)
%    &\text{if } i = 1,\ldots, \lceil n/2 \rceil, \\
%    0
%    &\text{otherwise,}
%  \end{cases}
%  \quad \text{and} \quad
%  \cev\epsilon^{n,1/2}(i)
%  := \begin{cases}
%    \epsilon(n + 1 - i)
%    &\text{if } i = 1,\ldots, \lceil n/2 \rceil, \\
%    0
%    &\text{otherwise.}
%  \end{cases}
%\end{equation*}
Taking the $\Gamma$--limit as $n\to\infty$, $E_n^1$ splits into two independent, similar terms, one for each boundary layer; the term describing the left boundary layer is given by
\begin{equation*} %\label{intro:eq:Einftyl}
  \begin{gathered}
  E^{l}_\infty : \bigaccv{ \epsilon \in \ell^2 (\N{}) }{ \epsilon(i) \geq -1 \: \: \forall \, i \geq 1 } \to \overline{\R{}}, \\
  E^{l}_\infty (\epsilon) 
  := \sum_{k=1}^\infty\sum_{j=0}^\infty \phi_k \bigg(\sum_{l = j+1}^{k+j} \epsilon(l)\bigg) + (\sigma^\infty, \epsilon )_{\ell^2(\N{})},
\end{gathered}  
\end{equation*}
where $\sigma^\infty \in \ell^2 (\N{})$ is given by $\sigma^\infty(i) := \sum_{k = i+1}^\infty (k-i) \bigabs{ V' (k) }$, and may be thought of as a stress induced on the boundary layer by the presence of the bulk. In Lemma \ref{lem:ex:and:uniq:Elinf} we prove existence and uniqueness of minimisers for $E^{l}_\infty$ in $\ell^2(\N{})$. Moreover, we show that the related infinite set of Euler-Lagrange equations is equivalent to solving 
\begin{equation} \label{intro:eqn:BL:ell2}
  \left\{ \begin{aligned}
    &0 = \sum_{\substack{k=0 \\ k \neq i}}^{\infty} V'\big[\chi(i) - \chi(k) \big], \quad i = 1, 2, 3, \ldots, \\
    &\bighaa{ \chi(i) - \chi(i-1) - 1 }_{i \in \N{}} \in \bigaccv{ \epsilon \in \ell^2 (\N{}) }{ \epsilon(i) \geq -1 \: \: \forall \, i \geq 1 }.
  \end{aligned} \right.
\end{equation}
This therefore provides us with a precise characterisation of solutions to \eqref{intro:eqn:BL}.

In \S\ref{sec:numerics}, we conclude by describing a numerical method for solving \eqref{intro:eqn:BL:ell2}. The numerical scheme approximates \eqref{intro:eqn:BL:ell2} by assuming that $\chi(i) = \chi(i-1) + 1$ for all $i$ larger than a fixed index. We compare its solution to the minimiser $x^n$ of \eqref{eq:equilibrium} for various numbers of $n$ particles and for two physically relevant choices for $V$: the case of dislocation dipoles (\cite{Hall2010}, $V(x) = x^{-2}$) and dislocation walls (\cite{GarronivanMeursPeletierScardia14ArXiv}, $V$ has a logarithmic singularity at $0$ and tails which vanish exponentially fast). We observe that the convergence rate of $n x^n (i)$ to $\chi(i)$ as $n \to \infty$ is close to $O(n^{-1})$, and independent of $i$. Furthermore, we find that the boundary layer profiles are qualitatively \emph{different} for the two choices of $V$, even though they satisfy all imposed conditions, and hence resemble the graph in Figure \ref{fig:y:on:01}.

\subsection{Discussion and conclusion}

The main fruits of our analysis are \eqref{intro:eqn:BL:ell2} and
\eqref{intro:eqn:decay}. Equation \eqref{intro:eqn:BL:ell2} is significant because it provides us with a
characterisation of the boundary layer behaviour in terms of an 
infinite system of discrete equations. In particular, \eqref{intro:eqn:BL:ell2} gives a precise meaning to the idea that the particles are `equispaced' in the bulk; by showing that \eqref{intro:eqn:BL:ell2} has a unique solution for $\chi$ where $\chi(i) - \chi(i-1) - 1 \in \ell^2(\N{})$, we place asymptotic limits on the extent to which the energy-minimising particle configuration can deviate from equal spacing. Moreover, the fact that we are able to obtain \eqref{intro:eqn:BL:ell2} from an asymptotic development of the ground state energy represents a significant theoretical advance for the treatment of discrete-scale boundary layers using $\Gamma$-convergence. In previous work, such as \cite{BraidesCicalese2007,ScardiaSchloemerkemperZanini2011,
Hudson2013,GarronivanMeursPeletierScardia14ArXiv}, either only  finite interaction ranges or continuum-scale boundary layers were considered.  Figure \ref{fig:y:on:01} illustrates the discrete-scale boundary layer in the case where $V(x) = |x|^{-2}$, showing that the solution to \eqref{intro:eqn:BL:ell2} provides a good asymptotic approximation to the solution of the full problem.  
%\CAMERON{[Why are the line graphs in this figure presented as line graphs? All variables presented here are discrete, aren't they?]} \PATRICK{[This was the best visual way I found to show the convergence of the density profiles. If you have a better idea for visualizing the data, feel free to experiment with it in} \verb|/numerics/dataProc a2/DataProcCPT.m|\PATRICK{]}. \CAMERON{[Apologies. I remembered doing something `discrete' with data like this in \cite{Hall2010}, but I've looked back at the relevant figures (5.2, 5.3), and I reckon I like your version more for this. Plus I'm lazy enough not to want to do too much fiddling around in code when everything looks fine. :)  ]}

Equation \eqref{intro:eqn:decay} is significant because it gives explicit form to the tail behaviour of the boundary layer solution in the case of a homogeneous potential. While \eqref{intro:eqn:decay} is only directly relevant in the case where $V(x) = |x|^{-a}$, it hints at why the analysis in \S\ref{sec:variational} relies on the assumption that $a > \frac{3}{2}$ in {\bf(Dec+)}. In the case where $1 < a \leq \frac{3}{2}$, \eqref{intro:eqn:decay} indicates that the associated $\eps(i) := \chi(i) - \chi(i-1) - 1$ is not in $\ell^2(\N{})$, and hence a new scaling and finer analysis will be necessary to recover the correct energetic description using $\Gamma$--convergence; in fact, we show in \S\ref{ssec:a:smaller:3over2} that taking the limit functional $E^{l}_\infty$ with a potential for which $a < 3/2$ leads to an
ill--posed variational problem. The derivation of \eqref{intro:eqn:decay} in \S\ref{sec:Formal2} does however suggest some of the tools necessary to extend our analysis: in the case where $1 < a \leq \frac{3}{2}$ it appears that the contributions from singular integral terms are important. We might therefore expect that \S\ref{sec:variational} can be extended to $ a \leq \frac{3}{2}$ by seeking a bulk correction to the total energy associated with nonlocal interactions between particles.

We now discuss our assumptions and findings as well as their implications in more detail. \smallskip

\begin{figure}[t]
\centering
\begin{tikzpicture}[scale=1.5]
\node (label) at (0,0){\includegraphics[width=3in]{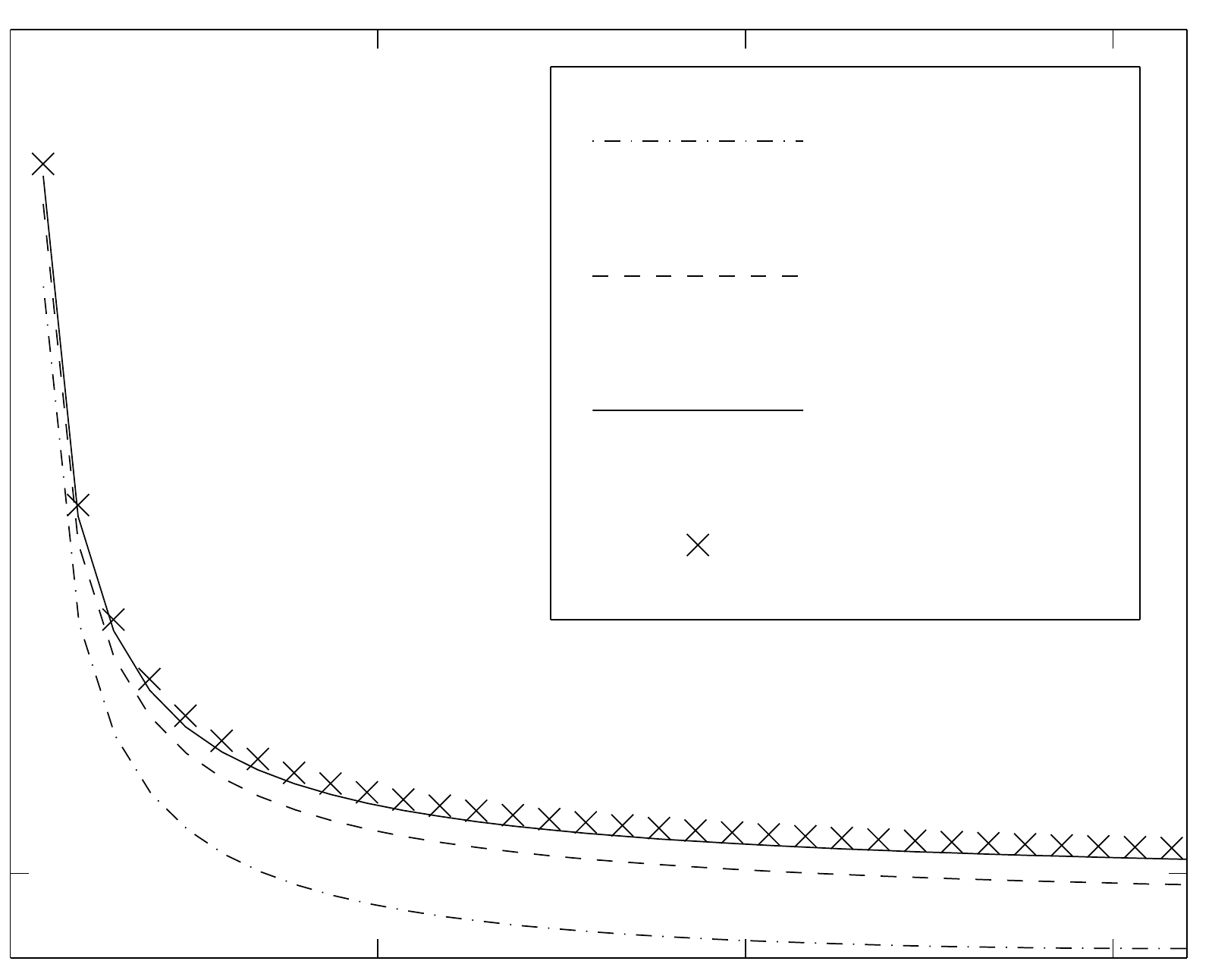}};
\draw (0,-2) node[below] {\LARGE $n x^n$};
\draw (-2.5,0) node[left] {\LARGE $\rho$};
\draw (-2.5, 1.95) node[left] {$1.1$};
\draw (-2.5,-1.6) node[left] {$1$};
\draw (-0.94,-2) node[below] {$10$};
\draw (-2.49,-2) node[below] {$0$};
\draw (0.6,-2) node[below] {$20$};
\draw (2.12,-2) node[below] {$30$};
\draw (1.1,1.5) node[right] {$n = 2^6$};
\draw (1.1,0.92) node[right] {$n = 2^8$};
\draw (1.1,0.35) node[right] {$n = 2^{10}$};
\draw (0.9,-0.25) node[right] {$\chi$ from \eqref{intro:eqn:BL:ell2}};
\end{tikzpicture}
\caption{The three line graphs depict the boundary layer of the minimiser $x^n$ of $E_n$ for $n = 2^6, 2^8, 2^{10}$ and $V(x) = x^{-2}$, in the rescaled coordinates $n x^n$. The graph of $n = 2^6$ corresponds to Figure \ref{fig:x:on:01}. The three line graphs indicate the convergence to the boundary-layer profile $\chi$ ($\times$) computed from \eqref{intro:eqn:BL:ell2}. See Figure \ref{fig:x:y:a2} for a further description of the graphs. }
\label{fig:y:on:01}
\end{figure}
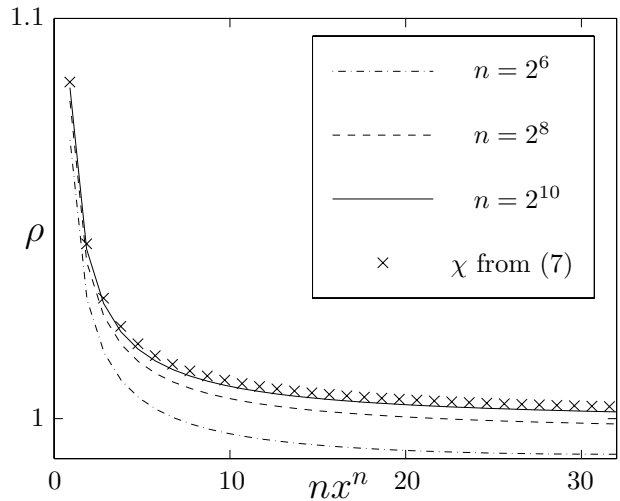

\emph{Choice of scaling}. The particular choice of scaling made here, i.e.~taking the length of the domain to depend linearly on the number of particles, is the usual choice made when considering 
the thermodynamic limit of a system with a fixed number of 
particles per unit volume \cite{BraidesDalMasoGarroni1999,BLBL07}. In other physical situations, other scalings may be more appropriate, and may lead to different energetic descriptions; see for example \cite{Geers2013,GarronivanMeursPeletierScardia14ArXiv}.

\emph{Matched asymptotic analysis versus $\Gamma$--convergence}. In keeping
with other studies, we observe that the advantage of matched asymptotic analysis is the ease and flexibility of the arguments with which we obtain equation \eqref{intro:eqn:BL}: it requires a less detailed analysis than the $\Gamma$--convergence result, relying upon on a well--chosen ansatz for the asymptotic development of the solution. On the other hand, the main advantage of the $\Gamma$--convergence statement is that it implies well--posedness of \eqref{intro:eqn:BL:ell2}, and with it the development of the ground state energy for $E_n$.
However, the analysis relies on the minimiser being close to the equispaced configuration, which makes it harder to apply e.g.~when a constant external stress is applied to the particles.

\emph{Asymptotic equilibrium problem}. The infinite sum in \eqref{intro:eqn:BL} does \emph{not} simply correspond to replacing $n$ by $\infty$ in the force balance in \eqref{eq:equilibrium}. Instead, \eqref{intro:eqn:BL} has two elements; the first equation in \eqref{intro:eqn:BL} is obtained by rescaling \eqref{eq:equilibrium} for the boundary layer case where $i = O(1)$ as $n \to \infty$, while the second equation in \eqref{intro:eqn:BL} is associated with a matching condition between the boundary layer solution and the bulk solution.

\emph{Comparison with \cite{Hall2010}}. In \cite{Hall2010}, formal asymptotic methods were used to analyse the problem of a pile-up of repulsive particles against a single fixed obstacle, driven by a constant external force. In the present work, we consider a similar problem where the particles are trapped in between two fixed obstacles. However, our formal asymptotic analysis provides two significant extensions to the results in \cite{Hall2010}. In \S\ref{sec:Formal1}, we dispense with the assumption used in \cite{Hall2010} that $V$ is $-a$-homogeneous, and obtain a leading-order asymptotic solution for a general $V$. This is novel in the formal asymptotics literature on discrete problems. Then, in \S\ref{sec:Formal2}, we reintroduce the assumption that $V$ is $-a$-homogeneous, and extend the asymptotic analysis in \cite{Hall2010} to include many higher-order corrections. Using this method, we are able to obtain a more precise description of the matching condition between the bulk and the boundary layer. %than that developed in \cite{Hall2010}.

\emph{Comparison with \cite{GarronivanMeursPeletierScardia14ArXiv}}. The setting in \cite{GarronivanMeursPeletierScardia14ArXiv} corresponds to changing the choice of scaling for the domain from $[0,n]$ made here to $[0, c_n]$ for some $1 \ll c_n \ll n$. The observed boundary layer consists of $O (n / c_n)$ particles, and is therefore expected to be described by a \emph{continuous} profile in the many--particle limit. In this paper, $c_n = n$, and while we also find that the boundary layer consists of $O (n / c_n) = O(1)$ particles, this means instead that the boundary layer profile remains \emph{discrete} in the many--particle limit. We see the effect of these different scaling regimes reflected in the assumptions on $V$; while the analysis in \cite{GarronivanMeursPeletierScardia14ArXiv} relies on less regularity and a weaker notion than convexity of $V$, we weaken the assumption of finite first moments on the tails of $V$ (which is slightly stronger than $a \geq 2$).

\emph{Other forms of confinement}. A physically--interesting extension of our scenario is to consider the system subject to a constant external stress term which pushes the particles to one of the two barriers, in place of one or both of the rigid boundary constraints that we consider. Such a constraint results in a free--boundary problem: examples of such scenarios are examined in e.g.~\cite{Hall2010, Geers2013, VanMeursMunteanPeletier14}. While we expect our asymptotic analysis to apply with some modifications along the lines of \cite{Hall2010,Hall2011}, our $\Gamma$--convergence analysis would require us to find an appropriate variable to describe the free boundary, and then to obtain \emph{a priori} estimates in this variable, similar to those given in \S\ref{ssec:key:ests}. This appears to be a significant challenge, but with appropriate intuition from formal asymptotics, may be overcome in future.

%\TOM{[Is this actually true? Wouldn't the formal analysis also require us to find a new ansatz for the free boundary and a verification that no intermediate scalings are relevant?]} \CAMERON{[Hmmm, I may have misunderstood what the message being sent here was: I thought the issue was that it's very easy to deal with fluctuations in the density when it comes to writing down the boundary layer problem using the formal asymptotic approach, whereas this is harder using the $\Gamma$-convergence approach. Yes, you need to propose an ansatz for the spatial scaling when you have a free boundary, and there will be multiple cases that need to be considered depending on whether the gaps between particles are $\ll 1$, $\ord(1)$ or $\gg 1$ in the regime you're in. With regard to the intermediate scalings, it's easy to see that they should (or shouldn't) exist, but you're right in that it can be nontrivial to show this definitively. I'm not sure how much to modify this paragraph to account for this?}

\emph{Lennard-Jones interactions}. In contrast to our assumptions of purely
repulsive interactions between particles, a system with Lennard-Jones--type
interactions is subject to both repulsive and attractive forces. To
our knowledge, there are no results yet concerning the
analysis of boundary layers in such systems without assuming a 
finite interaction range, whereas our analysis considers the 
interactions between all pairs of particles. In view of 
previous results concerning boundary layers in such systems 
\cite{BraidesCicalese2007,ScardiaSchloemerkemperZanini2011,Hudson2013}, 
it does however seem natural that with modification, similar 
techniques to those which we use in the proof of
Theorem~\ref{thm:Gamma:conv} could carry over to a Lennard-Jones 
setting including all interactions between particles. We expect that the key 
challenge here is to obtain a suitable compactness result, similar to Theorem~\ref{thm:Gamma:conv}, for `fractured' states.
\smallskip

In conclusion, our analysis enables us to give a precise characterisation of the discrete boundary layers at either end of the domain, and treats long--range particle interactions without assuming a finite interaction neighbourhood.
% where previous analyses either concentrated on the bulk behaviour of the system or concentrated on obtaining formal results.
%Since the energy associated with the location of particles in the boundary layers is small, boundary layer phenomena cannot be observed without looking to higher-order corrections to the energy. 
We have obtained these results by bringing together both formal and rigorous asymptotic methods in order to deliver a unified picture of the various scales associated with the discrete boundary layer problem, and in so doing, we succeeded in going further than prior analyses using both techniques individually. The core achievement of our work is \eqref{intro:eqn:BL:ell2}, which gives valuable insights into how systems involving finitely many particles will deviate from the predictions given by a continuum analysis of bulk behaviour. Thus, our analysis gives a firm foundation to future work on understanding surface effects at equilibrium in higher dimensional problems. 
%\CAMERON{[Was previously `which we expect to be helpful in understanding surface effects at equilibrium in higher dimensions.']}

%\PATRICK{[Tom: the conclusion was nicely written, but to me it felt more like a summary/abstract. Therefore, I tried to cut down some detailed statements that we already made elsewhere, and tried to focus on how our main result adds to the applications mentioned in the introduction. This is only done in the last line above. I failed to make our application stronger. Please see whether you can improve the statement about why our core achievement is useful in understanding current problems in surface effects, as mentioned nicely in the introduction. Of course, we can be speculative here.]}

\smallskip

The remainder of the paper is organised as follows. In \S\ref{sec:Formal1} we perform the asymptotic analysis to derive the boundary layer equation \eqref{intro:eqn:BL} from the force balance \eqref{eq:equilibrium} in the general case. In \S\ref{sec:Formal2}, we obtain higher--order corrections in the specific case where $V(x)=|x|^{-a}$. In \S\ref{sec:variational} we establish $\Gamma$--convergence of the energy difference $E_n^1$, and show how it connects the force balance \eqref{eq:equilibrium} to the description of the boundary layer in \eqref{intro:eqn:BL:ell2}. Finally, in \S\ref{sec:numerics}, we give numerical examples validating our asymptotic development for two physically relevant choices of $V$.

\section{Formal asymptotic analysis -- Leading order analysis for $a > 1$}
\label{sec:Formal1}

\subsection{Notation and preliminaries}

We begin by using classical formal asymptotic methods 
analogous to those in \cite{Hall2010} to obtain the leading-order 
asymptotic solution to the system of algebraic equations given in 
\eqref{eq:equilibrium}. The novelty of our approach here is that
it relies only upon the basic assumptions on $V$ detailed in \S\ref{sec:Setting} and {\bf(Reg+)}, and not on an explicit choice of potential as in \cite{Hall2010}.

To clarify the notation which is used
throughout this and the following section, we include 
Table~\ref{tab:Definitions} for the reader's convenience: the notation is equivalent to that used in \cite{HinchPert}. 
\begin{table}[t]
\centering
\caption{Asymptotic notation. All definitions are interpreted in the limit $n \to \infty$. }\label{tab:Definitions}
\begin{tabular}{p{5cm}m{8cm}}
\toprule
Notation & Definition\\
\midrule
$f(n)=O(g(n))$ & $\exists \, C>0$ such that $|f(n)|\leq C |g(n)|$\\
$f(n)=\ord(g(n))$ & $\exists \, C>0$ such that $\limsup |f(n)|/|g(n)| = C$\\
$f(n)=o(g(n))$, $f(n)\ll g(n)$ & $|f(n)|/|g(n)|\to 0$\\
$f(n)\sim\sum_{i=1}^\infty g_i(n)$ & $\forall \, k\in\N{}$, $f(n)-\sum_{i=1}^k g_i(n) = O(g_{k+1}(n))$\\
$f(n)\sim\sum_{i=1}^N g_i(n)$ & \parbox[t]{8cm}{$\forall \, k=1,\ldots N-1$, $f(n)-\sum_{i=1}^k g_i(n) = O(g_{k+1}(n))$,\\
$f(n)-\sum_{i=1}^N g_i(n) = o(g_N(n))$}\\
\bottomrule
\end{tabular}
\end{table}

% In addition to {\bf(Reg)} and {\bf($a$-Dec)}, in this section
% we shall make the additional assumption:
% \medskip
% \begin{center}
% \framed{0.7\textwidth}{\begin{description}[leftmargin=*]
% \item[(Reg+)] $V \in C^3(0, \infty)$, with $| V''' (x) | \lesssim |x|^{-a-3}$ for $|x| > 1$.
% \end{description}}\end{center}\medskip
% We note that $V(x)=|x|^{-a}$ with $a>1$ continues to satisfy all
% assumptions made here.

Supposing that $x(i;n)$ solves the system of equilibrium equations \eqref{eq:equilibrium} for a given
$n$, and following \cite{Hall2010}, we propose the following ansatz
for an approximate solution in the bulk:
\begin{equation}
 x(i;n) = \xi(in^{-1}; n), \label{eq:cont-ansatz}
\end{equation}
where $\xi(s;n)$ is expanded as an asymptotic power series
\begin{equation}
 \xi(s;n) \sim \xi_0(s) + n^{-b_1} \xi_1(s) + n^{-b_2} \xi_2(s) + \ldots, \label{eq:xi-expansion}
\end{equation}
with $b_i$ strictly increasing. For convenience, we simply write
$x(i)$ and $\xi(s)$ whenever possible, omitting their dependence 
upon $n$.

Following convention, we treat $\xi$ as though it were a function, even though the series definition of $\xi$ in \eqref{eq:xi-expansion} is an asymptotic series, and therefore may not converge for any fixed $s$ and $n$. Strictly speaking, equations involving $\xi$ should be interpreted as being true for fixed $s$ in the asymptotic limit as $n \to \infty$ where any instance of $\xi(s)$ is read as
\[
 \xi(s) = \sum_{k=0}^{Q} n^{-b_k} \xi_k(s) + O\big(n^{-b_{Q+1}}\big),
\]
for any choice of integer $Q$.

As we discuss in \S \ref{sec:Formal1-BLFail} and \S \ref{sec:Formal1-BL}, we encounter boundary layers when $s$ is sufficiently close to 0 or 1. As a result of these, we find that we will not be able to use the ansatz in \eqref{eq:cont-ansatz} and \eqref{eq:xi-expansion} to describe particle positions when $i$ is too close to 0 or $n$; instead, different ansatzes will be needed.
In our higher-order analysis in \S \ref{sec:Formal2}, we are careful to take account of the effects of the boundary layer from the beginning of our analysis, but in this section we begin by assuming that \eqref{eq:cont-ansatz} and \eqref{eq:xi-expansion} can be applied everywhere. While this is not strictly true (and would lead to contradictions if the analysis were extended to higher orders), identical results could be obtained by following the methods described in \S \ref{sec:Formal2-bulk}, where we use separate ansatzes for the bulk and the boundary layer from the outset.  %While this enables us to obtain a leading order solution for $\xi$, we note that it may be impossible to construct higher order corrections without separating the bulk, where the ansatz in \eqref{eq:xi-expansion} can be used, from the boundary layers, where a different ansatz is needed.

%\CAMERON{[Is this fine? I had a play with modifying this to emphasise the fact that \eqref{eq:xi-expansion} only gives an accurate assessment of the sizes of the various terms when $s = \ord(1)$ and $1 - s = \ord(1)$; otherwise the series may either fail to be asymptotic, or the functions $\xi_k$ may grow as $s \to 0$, so that the error terms are not the error terms that we would describe. But this is a fairly standard point in matched asymptotics, and it seems strange to labour it.]}

Due to the boundary conditions $x(0) = 0$ and $x(n) = 1$, we assume that $\xi(0) = 0$ and $\xi(1) = 1$. In practice, these boundary conditions will only be satisfied to leading order, so that we apply them as
  %the leading order solution $\xi_0$ is assumed to satisfy
\begin{equation}
 \label{eq:xi0-BCs}
 \xi_0(0) = 0\qquad\text{and}\qquad \xi_0(1) = 1.
\end{equation}

Additionally, we assume that $\xi(s)$ has the following smoothness and monotonicity properties:
\begin{description} 
  \item[($\xi$-Smooth)] $\xi \in C^4([0,1])$; 
  \item[($\xi$-Mon)] $\xi'$ is strictly positive, i.e.~there exists $M > 0$ such that $\xi'(s) \geq M$ for all $s \in [0,1]$.
\end{description}
%\CAMERON{Again, the presence of boundary layers means that these conditions will, in practice, be violated near $s = 0$ and $s = 1$; we describe improved versions of these conditions that can account for the boundary layer in \S \ref{sec:Formal2-bulk}. The leading order solution, $\xi_0$, must satisfy both $\xi_0 \in C^4([0,1])$ and $\xi_0' \geq M$ throughout the entire domain, and we also require both {\bf($\xi$-Smooth)} and {\bf($\xi$-Mon)} to hold when $s = \ord(1)$ and $1 - s = \ord{1}$. [Is this useful to say, or does it just mean that I'm tripping over myself?]}

The monotonicity assumption {\bf($\xi$-Mon)} implies that 
\begin{equation*}
  \chi(i+1)-\chi(i) = n[x(i+1)-x(i)] = n\int_{i/n}^{(i+1)/n} \xi'(s) \, \mathrm{d}s \geq M.
\end{equation*}
That is, we assume that at equilibrium no two particles are closer than $M/n$, 
uniformly in $n$, and thus the particle density 
is uniformly bounded. This is a
natural assumption as long as the long-range interactions between particles are not strong 
enough to make very high densities of particles favourable as 
$n\to\infty$.

Since $\xi$ is a continuous object, it is natural to cast the discrete force balance in \eqref{eq:equilibrium} in a continuous form as well. By separating the interactions with the neighbours on the left from those on the right, we rewrite \eqref{eq:equilibrium} as 
\begin{gather} \label{eq:F(s)-Defn-Formal1}
 F(s) = 0, \quad \text{for all } s = \frac1n, \ldots, \frac{n-1}n, \\
 F(s) := \sum_{k=1}^{\lfloor n- sn \rfloor} 
 V'\big(n \big[\xi(s+kn^{-1}) - \xi(s)\big] \big)
 - \sum_{k=1}^{\lfloor sn \rfloor} 
 V'\big(n \big[\xi(s) - \xi(s-kn^{-1})\big] \big). \label{eq:F(s)-Defn-Formal1:Q}
\end{gather}
In Section \ref{sec:Formal1-Bulk}, we manipulate this definition of $F(s)$ in order to obtain the leading-order dependence of $F$ on $\xi$ as $n \to \infty$. From this, we can obtain an equation for $\xi_0(s)$ and hence an asymptotic expression for the equilibrium particle locations in the bulk. 

\subsection{Asymptotic analysis using the bulk ansatz}
\label{sec:Formal1-Bulk}

In the analysis below, we show that $F(s)$ is given asymptotically by
\begin{equation}
 F(s) = n^{-1} \xi''(s) \sum_{k=1}^{\infty} \big( V''\big[\xi'(s) k] k^2 \big) + o(n^{-1}),
 \label{eq:Qasymp}
\end{equation}
so long as $s \gg n^{-\frac{1}{a}}$ and $1 - s \gg n^{-\frac{1}{a}}$. To this end, we introduce an arbitrary integer $H$ where $n^{\frac{1}{a}} \ll H\ll n$, and split the sums in \eqref{eq:F(s)-Defn-Formal1:Q} as follows:
\begin{multline} \label{eq:Qasymp:expanded}
 F(s) = \underbrace{\sum_{k=1}^{H} \Big[ 
 V'\big(n \big[\xi(s+kn^{-1}) - \xi(s)\big] \big)
 - V'\big(n \big[\xi(s) - \xi(s-kn^{-1})\big] \big) \Big]}_{=: S_1} \\
 + \underbrace{\sum_{k=1}^{\lfloor n- sn \rfloor - H} 
 V'\big(n \big[\xi(s+Hn^{-1}+kn^{-1}) - \xi(s)\big] \big)}_{=: S_2}
 \\- \underbrace{\sum_{k=1}^{\lfloor sn \rfloor - H} 
 V'\big(n \big[\xi(s)  - \xi(s- Hn^{-1}-kn^{-1})\big] \big)}_{=: S_3}.
\end{multline}
We observe from {\bf($\xi$-Mon)} that the argument of $V'$ in the sums $S_2$ and $S_3$ is bounded from below by $M(H + k)$. Then, by {\bf($a$-Dec)}, the summands in $S_2$ and $S_3$ are bounded in absolute value by $c_\delta' M^{-a-1} (H + k)^{-a-1}$. Therefore,
\begin{equation*}
  |S_2| + |S_3| 
  \leq 2 \sum_{k=1}^\infty c_\delta' M^{-a-1} (H + k)^{-a-1}
  \lesssim H^{-a} = o (n^{-1}).
\end{equation*}

To expand $S_1$ in terms of $n$, we repeatedly employ Taylor's
theorem. More precisely, using the regularity of
$V$ and $\xi$ as given by {\bf(Reg+)} and {\bf($\xi$-Smooth)}, we write
\begin{gather*}
 V'(x+\delta) = V'(x) + V''(x)\delta  + 
 \tfrac12V'''(x+\theta_\delta \delta)\delta^2, \\
 \xi(s+\delta) = \xi(s) + \xi'(s) \delta + \tfrac12 \xi''(s) 
 \delta^2 + \tfrac{1}{6}\xi'''(s+\rho_\delta \delta)\delta^3,
\end{gather*}
for some $\theta_\delta,\rho_\delta\in[0,1]$. Moreover, 
{\bf(Reg+)} implies that
\begin{equation*}
  V'''(x+\theta_\delta \delta)\delta^2 = \delta^2 \cdot O(x^{-a-3})\quad
  \text{as }x\to\infty,
\end{equation*}
as long as $\delta \ll x$.

We now apply the Taylor expansion of $\xi$ to the arguments of $V'$ in $S_1$. By the
uniform continuity of $\xi'''$, we obtain
\begin{align*}
 n \big[\xi(s \pm kn^{-1}) - \xi(s)\big] 
 &= n\big[ \pm \xi'(s) kn^{-1} + \tfrac{ 1 }{2}\xi''(s) k^2 n^{-2}+O\big(k^3n^{-3}\big)\big],\\
 &= \pm k\xi'(s) +k\Big[\tfrac{1}{2}\xi''(s)kn^{-1}
   +O\big(k^2n^{-2}\big)\Big]
\end{align*}
as long as $k\ll n$. Moreover, since $H \ll n$, we see that $k n^{-1}\leq Hn^{-1}\ll 1$ throughout the sum given in $S_1$. Applying the Taylor expansion of $V$
and using the oddness of $V'$ and evenness of $V''$, we find
that
\begin{multline} \label{for:S1:expanded}
 S_1 =
 \sum_{k=1}^{H} \bigg[
   V''\big[k\xi'(s) \big]\xi''(s)k^2 n^{-1}
   + V''\big[ k \xi'(s) \big]k \cdot O(k^2 n^{-2})
 \\+ V'''\big[k\xi'(s)+\theta_k k \cdot O(k n^{-1})\big]
   k^2 \cdot O\big(k^2n^{-2}\big)
 \\+ V'''\big[-k\xi'(s)-\theta_{-k} k \cdot O(k n^{-1})\big]k^2  \cdot
   O\big(k^2n^{-2}\big)\bigg].
\end{multline}

We now show that the sum of the last three of the four terms in the summand of \eqref{for:S1:expanded} is $o(n^{-1})$. To treat the second term, we note that 
$V''[\xi'(s) k] = O(k^{-a-2})$ as $kn^{-1} \to 0$ and
$k \to \infty$, and hence
\[
  V''\big[\xi'(s) k \big] k \cdot O(k^2n^{-2}) = O(k^{1-a}n^{-2}).
\]
We sum to find that
\begin{equation}
  \sum_{k=1}^{H} V''\big[\xi'(s) k \big] k \cdot O(k^2n^{-2})
    = \begin{cases}
    O\big(H^{2-a}n^{-2}\big) & a\in(1,2),\\
    O\big(\!\log(H) n^{-2}\big) &a=2,\\
     O(n^{-2}) & a>2.
  \end{cases}
  \label{boundingerror:asymp:temp}
\end{equation}
In all three cases, the fact that $H \gg n^{\frac{1}{a}}$ implies that the sum in \eqref{boundingerror:asymp:temp} is $o(n^{-1})$.

To bound the third and fourth term in the summand of \eqref{for:S1:expanded}, we first observe that
\begin{equation*}
  V'''\big[\pm k\xi'(s) \pm \theta_{\pm k} k \cdot O(kn^{-1})\big]
  = O(k^{-3-a})
\end{equation*}
as long as $k n^{-1}\ll 1$. Hence
\[
  V'''\big[\pm k\xi'(s) \pm \theta_{\pm k} k\cdot O(kn^{-1})\big]k^2 \cdot O(k^2 n^{-2})= 
  O(k^{1-a}n^{-2}),
\]
and an analogous argument to that used to obtain \eqref{boundingerror:asymp:temp} implies
\[
S_1 =
  n^{-1} \sum_{k=1}^{H} \Big[ 
 \xi''(s) V''\big[\xi'(s) k \big] k^2 \Big] + o(n^{-1}).
\]

Finally, to obtain \eqref{eq:Qasymp}, we use $V''\big[\xi'(s) k \big] k^2 = O(k^{-a})$ to deduce that
\[
 \sum_{k=H+1}^{\infty} \Big[ 
 \xi''(s) V''\big[\xi'(s) k \big] k^2 \Big] = O(H^{1-a}).
\]
Since $H \gg 1$, adding this term does not change \eqref{eq:Qasymp:expanded} at leading order, and we conclude that \eqref{eq:Qasymp} holds.

Next, we investigate the implications of \eqref{eq:F(s)-Defn-Formal1} and \eqref{eq:Qasymp} for our ansatz in \eqref{eq:cont-ansatz}. Since $\xi'$ is assumed to be bounded, the property {\bf(Cvx)} implies that the sum in \eqref{eq:Qasymp} is bounded from below by a positive constant. We conclude that the leading order density satisfies $\xi_0''(s) = 0$ whenever $s$ is in an appropriate range. As argued in \cite{Hall2010}, we apply the boundary conditions $\xi_0(0) = 0$ and $\xi_0(1) = 1$ to $\xi_0''(s) = 0$, because possible boundary layers can only affect higher order corrections to the particle positions in the bulk. Consequently,
\begin{equation*}
 \xi_0(s) = s.
\end{equation*} 

\subsection{Investigating a continuum rescaling for the boundary layer}
\label{sec:Formal1-BLFail}

Our derivation of \eqref{eq:Qasymp} relies on the assumption that $H \ll sn \ll n - H$ for some $H$ with $n^\frac{1}{a} \ll H$. Hence, we cannot be confident that $\xi(s) \sim s$ is a valid leading-order approximation of the particle positions when $s = O(n^{-\frac{a-1}{a}})$ or when $s = 1 - O(n^{-\frac{a-1}{a}})$. To investigate these regimes fully, a new ansatz is therefore required.

When $i$ or $n - i$ are sufficiently small, we may expect to see boundary layers where the original bulk scalings no longer apply; for an example of boundary layer analysis for a similar system, see \cite{Hall2010}. The fact that the bulk ansatz is applicable when $i \gg n^{\frac{1}{a}}$ and  $n - i \gg n^{\frac{1}{a}}$ indicates that the boundary layer width can be no greater than $\ord(n^\frac{1}{a})$.

%However, this does not imply that the boundary layer width is precisely $\ord(n^\frac{1}{a})$; as discussed in \cite{Eckhaus1979,Eckhaus1994} the width of a boundary layer is defined by finding a new distinguished limit in the variable scalings, rather than by the presence of an overlap region where both scalings are appropriate.

To find the boundary layer rescaling, we follow a similar procedure to that used for classical problems from differential equations. We propose a new scaling of the variables, and analyse it to determine whether it yields a `distinguished limit' where there is a new dominant balance between terms. Because of the symmetric geometry of the particle system, we concentrate on the boundary layer in the vicinity of $s = 0$. We begin by considering a continuum ansatz, which we assume to be valid when $i = \ord(n^\beta)$ for some $0 < \beta \leq \frac{1}{a}$, and which takes the form
\begin{equation}
 x(i) = n^{\beta-1} \hat{\xi}(in^{-\beta};n), \label{eq:x(i)-ContRescaling}
\end{equation}
where $\hat{\xi}(\hat{s};n)$ is again a continuum ansatz with the smoothness and monotonicity properties described before:
\begin{description} 
  \item[($\hat{\xi}$-Smooth)] $\hat{\xi} \in C^3[0,\infty)$;
  \item[($\hat{\xi}$-Mon)] $\hat{\xi}'$ is positive and uniformly bounded away from zero, so that $\hat{\xi}'(\hat{s}) \geq \hat{M} > 0$ for some constant $\hat{M}$.
\end{description}

The choice of scaling in \eqref{eq:x(i)-ContRescaling} is based on two observations. Firstly, the fact that we propose a new ansatz that is valid when $i = \ord(n^\beta)$ means that $\hat{\xi}$ must be a function of $\hat{s} := i n^{-\beta}$, which is $\ord(1)$ when $i = \ord(n^\beta)$. Secondly, from $\xi(s) \sim s$ as $s \rightarrow 0$, we obtain that $x(i) \sim i n^{-1}$ as $i$ decreases out of the region where the bulk ansatz is valid. By the principles that underly the method of matched asymptotic expansions, this must be identical to the behaviour of the rescaled $x(i)$ in \eqref{eq:x(i)-ContRescaling} as $i$ increases out of the region where this boundary layer ansatz is valid. By scaling $x(i)$ with $n^\beta$ in \eqref{eq:x(i)-ContRescaling}, we can satisfy this requirement by imposing the following matching condition on the leading order solution to $\hat{\xi}(\hat{s})$, based on Van Dyke's matching principle:
\begin{equation*}
 \hat{\xi}_0(\hat{s}) \sim \hat{s}, \quad \text{as }\hat{s} \rightarrow \infty.
\end{equation*} 

Now, we proceed by considering the case where $i = \ord(n^\beta)$ and so $\hat{s} = \ord(1)$, and we define $\hat{F}(\hat{s})$ in a similar manner to \eqref{eq:F(s)-Defn-Formal1} as follows:
\begin{multline}
 \hat{F}(\hat{s}) := 
 \sum_{k=1}^{\lfloor K- sn \rfloor} 
 V'\big(n^{\beta} \big[\hat{\xi}(\hat{s}+kn^{-\beta}) - \hat{\xi}(\hat{s})\big] \big)
 - \sum_{k=1}^{\lfloor sn \rfloor} 
 V'\big(n^{\beta} \big[\hat{\xi}(\hat{s}) - \hat{\xi}(\hat{s}-kn^{-\beta})\big] \big) \\
 + \sum_{k=K+1}^{n}
 V'\big(n\big[\xi(0+kn^{-1}) - n^{\beta-1}\hat{\xi}(\hat{s})\big] \big),
 \label{eq:tildeQ-Intro}
\end{multline}
where $K$ is chosen so that $n^\beta \ll K \ll n$, so that $K$ lies in the intermediate region between the two scaling regimes.

We start by showing that the third term in \eqref{eq:tildeQ-Intro} is small. Using {\bf($a$-Dec)} and {\bf($\hat{\xi}$-Mon)}, we quickly see that
\[
 V'\big(n \big[\xi(0+kn^{-1}) -  n^{\beta-1}\hat{\xi}(\hat{s})\big] \big) = O(K^{-a-1}),
\]
which when summed, gives
\[
\sum_{k=K+1}^n V'\big(n\big[\xi(0+kn^{-1}\big)-
n^{\beta-1}\hat{\xi}(\hat{s})\big]\big)=O(K^{-a})= 
o(n^{-a \beta}).
\]
Moreover, we can manipulate the first two sums in \eqref{eq:tildeQ-Intro} as we did in Section \ref{sec:Formal1-Bulk} by introducing $\hat{H}$ where $n^{\frac{\beta}{a}} \ll \hat{H}\ll n^{\beta}$, so that
\begin{multline*}
 \hat{F}(\hat{s}) = \sum_{k=1}^{\hat{H}} \Big[ 
 V'\big(n^\beta \big[\hat{\xi}(\hat{s}+kn^{-\beta}) - \hat{\xi}(\hat{s})\big] \big)
 - V'\big(n^\beta \big[\hat{\xi}(\hat{s}) - \hat{\xi}(\hat{s}-kn^{-\beta})\big] \big) \Big] \\
 + \sum_{k=1}^{\lfloor K- \hat{s}n \rfloor - \hat{H}} 
 V'\big(n^\beta \big[\hat{\xi}(\hat{s}+\hat{H}n^{-\beta}+kn^{-\beta}) - \hat{\xi}(\hat{s})\big] \big) \\
 - \sum_{k=1}^{\lfloor \hat{s}n \rfloor - \hat{H}} 
 V'\big(n^\beta \big[\hat{\xi}(\hat{s}) - \hat{\xi}(\hat{s}- \hat{H}n^{-\beta} -kn^{-\beta})\big] \big) 
 + o(n^{-\beta}).
\end{multline*}

With $\hat{F}(\hat{s})$ in this form, we can repeat the computations in Section \ref{sec:Formal1-Bulk}, ultimately obtaining the result that
\begin{equation*}
 \hat{F}(\hat{s}) = n^{-\beta} \hat{\xi}''(\hat{s}) \sum_{k=1}^{\infty} \big( V''\big[\hat{\xi}'(\hat{s}) k] k^2 \big) + o(n^{-\beta}),
\end{equation*}
which is similar to \eqref{eq:Qasymp}. Since $\hat{F}(\hat{s}) = 0$ whenever $\hat{s} = in^{-\beta}$ for $i = \ord(n^\beta)$, we again obtain that $\hat{\xi}''(\hat{s}) = 0$. Hence, the leading order behaviour of $\hat{\xi}$ in the proposed boundary layer is identical to the leading order behaviour of $\xi$. Since there is no qualitative difference between the equation to be solved in the bulk and the equation to be solved in the boundary layer, we conclude that this is not a distinguished limit of the system, and 
thus the only possible boundary layer in our system is the discrete boundary layer that could occur when $i = \ord(1)$.

\subsection{Discrete boundary layer scaling}
\label{sec:Formal1-BL}

Having established that there can be no boundary layers associated with a continuum rescaling, we propose the discrete boundary layer ansatz, $nx(i;n) = \chi(i;n)$, where $\chi(i;n)$ is expanded as an asymptotic series in powers of $n$. As before, we omit the explicit dependence on $n$ unless we wish to emphasise that $\chi(i)$ is an asymptotic series. Since $\chi(i)$ only takes integer arguments, we do not make any smoothness assumptions about $\chi(i)$. To preserve the ordering of the particles, we require that $\chi(i)$ is strictly increasing.

Using Van Dyke's matching principle, we find that the leading order behaviour of $\chi(i)$ for large $i$ must be given by
\[
  \chi_0(i) \sim i \quad \text{as } i \to \infty.
\]
More specifically, we can use Van Dyke's matching principle to match between $\xi'(s)$ and $\chi(i) - \chi(i-1)$, which gives the following, stronger condition on $\chi_0(i)$ as $i \to \infty$:
\begin{equation}
 \chi_0(i) - \chi_0(i-1) \sim 1 \quad \text{as } i \to \infty.
\label{eq:chi0-MatchingCond}
\end{equation}

Let $i = \ord(1)$ and let $K$ be chosen so that $1 \ll K \ll  n$. Similar to \eqref{eq:F(s)-Defn-Formal1} and \eqref{eq:F(s)-Defn-Formal1:Q}, we write the force balance equation for the $i$th particle as:
\begin{equation}
 0 = \sum_{\substack{k=0 \\ k \neq i}}^{K} V'\big[\chi(i) - \chi(k) \big] - \sum_{k=K+1}^{n} V'\big(n \big[\xi(kn^{-1}) - n^{-1} \chi(i) \big] \big). 
 \label{eq:BL1-ForceBalance-LeadingTemp}
\end{equation}
By the same method as discussed in Section \ref{sec:Formal1-BLFail}, we note that
\[
 \sum_{k=K+1}^{n} V'\big(n \big[\xi(kn^{-1}) - n^{-1} \chi(i) \big] \big) = O(K^{-a}) = o(1),
\]
while the first sum in \eqref{eq:BL1-ForceBalance-LeadingTemp} has no explicit dependence on $n$ except through the fact that $\chi$ is an asymptotic series. 

Since $\chi_0(i) \sim i$ as $i \to \infty$, we note that the first sum in \eqref{eq:BL1-ForceBalance-LeadingTemp} must be finite as $K \rightarrow \infty$. Hence, we can take $n$ and $K$ to $\infty$ in \eqref{eq:BL1-ForceBalance-LeadingTemp} to obtain the leading order equation
\begin{equation*}
%\label{eq:chi0-System}
 0 = \sum_{\substack{k=0 \\ k \neq i}}^{\infty} V'\big[\chi_0(i) - \chi_0(k) \big], \quad i = 1, 2, 3, \ldots,
\end{equation*}
which must be solved subject to the matching condition in \eqref{eq:chi0-MatchingCond}. Note that this system of discrete equations generalises those of \cite{Hall2010} to a wide class of potentials $V$.

\section{Formal asymptotic analysis -- Higher order analysis for $V(x) = |x|^{-a}$}
\label{sec:Formal2}

\subsection{Summary of results} %\CAMERON{To self: Come back and fix this subsection after dealing with the rest of the section.}

In this section, we use formal asymptotic analysis to determine higher order corrections to the particle positions in the specific case where the potential is given by $V(x) = |x|^{-a}$ for any $a > 1$. As in \S \ref{sec:Formal1}, we obtain our solutions by assuming a continuum ansatz for the particle positions in the bulk of the of the domain and a discrete ansatz in the boundary layers at the ends of the domain. Using the method of matched asymptotic expansions, we obtain asymptotic solutions to the particle positions up to $o(n^{-(a-1)})$ in both the bulk problem and the rescaled boundary layer problem.

To obtain equations for the particle positions in the bulk, we draw on the results from \cite{Sidi2012} to use Euler--Maclaurin summation to express the total force on any particle as the sum of a `local contribution' involving the particle density at that point, and a singular integral that represents the effect of long-range interactions between particles. At leading order, the particle density is governed by a simple differential equation as in \S \ref{sec:Formal1}. Only at higher orders does the nonlocal effect of the long-range interactions on the bulk behaviour become significant, appearing through a singular integral term. %\CAMERON{Using the method of matched asymptotic expansions, these high-order corrections to the bulk solution can be used to obtain detailed information about the decay behaviour of the leading order solution to the discrete boundary layer problem.} 

%\PATRICK{[Well, from the matching procedure at the end it appears that the singular integral term is responsible for the decay rate of $\chi_0(i) - \chi_0(i-1)$. It therefore gives a precise connection between the decay rate and the non-local interactions, their strength being related to the decay of the tails of $V$. I think this is a very nice result; not only does it strengthen the $\Gamma$-convergence result, but it also gives a nice link between the 'size' of the boundary layer and the 'strength' of the non-locality. I have added some lines in our introduction on this. For here, I would give the singular integral more credit than just saying that it adds to higher order correctors.]} \CAMERON{[Does the sentence above address this? I'm not quite sure what to say here about the link with the $\Gamma$-convergence result. One extra thing is that I think of the singular integral term as indicating the decay of $\chi_0$ but not controlling it. Ultimately, the solution for $\chi_0$ should be the same regardless of whether we have the extra info that comes from higher orders of matched asymptotics; it's just a tool that lets us investigate properties of the solution for $\chi_0$ that are less obvious from the discrete problem.]}
%\PATRICK{[Ah, now I see. It was a misunderstanding from my part; I also took the leading order of $\chi$ into account here. Sorry. I changed the sentence a bit to avoid this confusion from happening again, even if it is just me :P.]} \CAMERON{[I have also rephrased it slightly, just to improve the flow of the sentence.]}

As previously, we find that the particle positions in the boundary layers are governed by different equations from the particle positions in the bulk. In order to analyse this, we apply the method of matched asymptotic expansions, using the techniques of intermediate matching (see, for example, \cite{HinchPert}). This enables us to exploit the existence of an `intermediate scaling regime', where the bulk ansatz and the boundary layer ansatz give equivalent results, in order to determine the sizes of the asymptotic correction terms and the appropriate matching conditions that relate the bulk solution to the boundary layer solution.

We find that the particle locations in the bulk region are given by $x(i;n) \sim \xi(in^{-1};n)$, where the asymptotic expansion of $\xi(s;n)$ takes different forms depending on the value of $a$. In the case where $1 < a < 2$, we find that $\xi(s;n)$ takes the form
\begin{equation*}
 \xi(s;n) = s + \frac{s^{-(a-2)} - (1-s)^{-(a-2)} + 1 - 2s}{\zeta(a) (a-2) (a^3 - a)} \, \frac{1}{n^{a-1}} +  o\left(\frac{1}{n^{a-1}}\right);
\end{equation*} 
in the case where $a = 2$, we find that $\xi(s;n)$ takes the form
\begin{equation*}
 \xi(s;n) = s + \frac{2s - 1}{\pi^2} \, \frac{\log n}{n} + \left[ \frac{\log(1-s) - \log(s)}{\pi^2} + (s - \tfrac{1}{2}) \tilde{p} \right] \frac{1}{n} + o\left(\frac{1}{n}\right); 
\end{equation*} 
and in the case where $a > 2$, we find that
\begin{equation*}
 %\label{eq:GenFormal-Bulk}
 \xi(s;n) = s + \sum_{k=1}^{\ceil{a-2}} ( s - \tfrac12)p_k \frac1{n^{k}}
 + \left[\frac{s^{-(a-2)} - (1-s)^{-(a-2)}}{\zeta(a) (a-2)(a^3-a)} +  (s - \tfrac12)\tilde{p}\right] \frac1{n^{a-1}} + o \Bighaa{ \frac1{n^{a-1}} },
\end{equation*}
where $\zeta(s)$ is the Riemann zeta function, and where $p_k$ and $\tilde{p}$ are constants. The constants $p_k$ and $\tilde{p}$ are chosen so that the bulk solution and the boundary layer solution are equivalent in some overlap region. As such, these constants depend on the solutions of the boundary layer problems, and can be defined iteratively as we discuss below.

In the boundary layer, we find that the particle locations are given by $x(i) = n^{-1} \chi(i;n)$, where the expansion of $\chi(i;n)$ takes the form
\begin{equation*}
 \chi(i;n) = 
 \begin{cases} \displaystyle
    \sum_{j=0}^{\ceil{a-2}} n^{-j} \chi_j(i) + n^{-(a-1)} \tilde{\chi}(i) + o(n^{-(a-1)}), & a \neq 2; \\
    \chi_0(i) + n^{-1} \log n \tilde{\chi}^*(i) + n^{-1} \tilde{\chi}(i) + o(n^{-1}), & a = 2.
 \end{cases}
\end{equation*} 
The sequences $\chi_j$, $\tilde{\chi}$ and $\tilde{\chi}^*$ are found by solving infinite systems. We leave their precise description to \S \ref{sec:Formal2-BL}.
The infinite system satisfied by $\chi_j(i)$ is linear for $j \geq 1$, and nonlinear for $j = 0$. These infinite systems must be solved subject to a matching condition with the bulk. For example, $\chi_0(i)$ must satisfy the matching condition given by $ \chi_0(i) - \chi_0(i-1) \rightarrow 1$ as $i \rightarrow \infty$. In the case where $a > 2$, we then obtain $p_1$ from the limit
\begin{equation*}
 p_1 = 2\lim_{i \rightarrow \infty} \left[i - \chi_0(i) \right],
\end{equation*}
and we find that $\chi_1(i)$ must satisfy the matching condition $ \chi_1(i) - \chi_1(i-1) \rightarrow p_1$. If $a > 3$, we can then define $p_2$ by the limit
\begin{equation*}
 p_2 = 2\lim_{i \rightarrow \infty} \left[p_1 i - \chi_1(i) \right],
\end{equation*}
and we find that $\chi_2(i)$ must satisfy the matching condition $ \chi_2(i) - \chi_2(i-1) \rightarrow p_2$. This iterative process can be continued until the $O(n^{-(a-1)})$ terms---or, if $a = 2$, $O(n^{-1} \log n)$ terms---are reached. At this stage, we can again use the same process to find the matching conditions for $\tilde{\chi}$ and $\tilde{\chi}^*$. If $a = 2$, we find that $\tilde{\chi}^*$ satisfies $\tilde{\chi}^*(i) - \tilde{\chi}^*(i-1) \rightarrow \frac{2}{\pi^2}$, while for all $a > 1$ we find that $\tilde{\chi}$ satisfies $\tilde{\chi}(i) - \tilde{\chi}(i-1) \rightarrow \tilde{p}$, where $\tilde{p}$ is given by 
%[I noticed that I'd avoided mentioning the matching condition for $\tilde{\chi}^*$.]
\begin{equation*}
 \tilde{p} = 
 \begin{cases}
  -\frac{2}{\zeta(a) (a-2) (a^3 - a)}, & a \not \in \mathbb{N}; \\
  2 \lim_{i \rightarrow \infty}  \left[ i - \chi_{0}(i) \right], & a = 2; \\
  2 \lim_{i \rightarrow \infty}  \left[ p_{a-1}i - \chi_{a-1}(i) \right] -\frac{2}{\zeta(a) (a-2) (a^3 - a)}, & a = 3,\,4,\,\ldots.
 \end{cases}
\end{equation*}

One valuable feature of the matched asymptotic analysis is that it gives us more precise details of the decay rates of the discrete solutions, $\chi_k(i)$, than could be obtained for $\chi_0(i)$ using leading order analysis. In particular, we find that the asymptotic behaviour of $\chi_0(i)$ for large $i$ is given by
\begin{equation} \label{for:chi0:result}
 \chi_0(i) = 
 \begin{cases}
  i + \frac{i^{-(a-2)}}{\zeta(a) (a-2) (a^3 - a)} + o\left( i^{-(a-2)} \right), & a < 2; \\
  i - \frac{\log i}{\pi^2} - \frac{\tilde{p}}{2} + o(1), & a = 2; \\
  i - \frac{p_1}{2} + \frac{i^{-(a-2)}}{\zeta(a) (a-2) (a^3 - a)} + o\left( i^{-(a-2)} \right), & a > 2.
 \end{cases}
\end{equation}
It follows in all cases that
\begin{equation}
 \label{eq:chi0Diff-FinalResult}
 \chi_0(i) - \chi_0(i-1) = 1 - \frac{i^{-(a-1)}}{\zeta(a)(a^3 - a)} + o\big(i^{-(a-1)}\big),
 \qquad a > 1,
\end{equation}
where the scaling of the error term is implied by the postulated differentiability of $\xi$.

Similarly, we can obtain bounds on the decay rates of the higher order corrections, $\chi_j(i)$, in the case where $0 < j < a-1$. These take the form
\begin{equation*}
 \chi_j(i) = 
 \begin{cases}
    p_{j} i - \tfrac{p_{j+1}}{2} + O\big(i^{-(a-2-j)}\big), & 0 < j < a-2; \\
    p_{j} i - \left(\frac{1}{\zeta(a) (a-2) (a^3 - a)} + \frac{\tilde{p}}{2} \right) + o(1), & j = a-2; \\
    p_{j} i + O\big(i^{-(a-2-j)}\big), & a-2 < j < a-1.
 \end{cases}
\end{equation*}
%\PATRICK{[Where is $(.)_k$ defined again? Maybe it appears first in \eqref{eq:Ap-Defn}, but also there I couldn't find its definition.]} \CAMERON{[The Pochhammer symbol is defined after \eqref{eq:S1General-NoBell}, but I've edited the equation above to remove it.]}

Even higher order corrections to the solution (both in the bulk and in the boundary layer) could potentially be obtained by applying the same asymptotic techniques. However, this would involve addressing the direct influence of the boundary layer on the bulk, leading to significant mathematical complications without leading to greater insights into the behaviour of the solution. Throughout this section, we discuss how higher order corrections might be obtained, but we do not pursue any high-order analysis in detail. 

%With some modifications, the methods presented in this section (and the methods for obtaining high-order corrections) can be extended to all potentials $V(x)$ whose the tail behaviour is given by
%\[
% V(x) \sim c_0 \, x^{-a_0} + c_1 \, x^{-a_1} + c_2 \, x^{-a_2} + \ldots, \quad \text{as } x \rightarrow \infty,
%\]
%where $1 < a_0 < a_1 < a_2 < \ldots$, and the constants $\{c_k\}$ are all positive. 
%\PATRICK{[I commented out some text above, because it was confusing and inconsistent with the text below. It contained a different assumption than \eqref{eq:GenPotential-Derivs}, which I think is much weaker, but not sufficient for the arguments later on in \S 3]} \CAMERON{[That's fine. The above was the assumption I used when I was sketching out the arguments for the very high order terms as well.]}
Additionally, the results obtained in this section for $\xi$ and $\chi$ up to $\ord(n^{-(a-1)})$ can be extended to a larger class of potentials $V$ which satisfy, in addition to {\bf(Reg)}, {\bf(Sing)}, and {\bf(Cvx)}, that
\begin{equation}
 \label{eq:GenPotential-Derivs}
 V^{(r)}(x) \sim (-1)^r a \cdot (a+1) \cdots (a+r-1) x^{-a-r}, 
 \quad \text{as } x \rightarrow \infty,
 \quad \text{for } r = 0, \ldots, 2\floor{\tfrac{a+3}{2}}.
\end{equation}
%\CAMERON{[Note: This is from the fact that the derivatives up to $2\floor{\frac{a-1}{2}} + 2$ will appear in the definitions of $\mathcal{B}_p$, and we add some `extra derivatives' in order to be able to bound the correction terms using an argument analogous to \S2.2. Unless you can be very clever with making some Taylor corrections cancel in a way that seems a little dodgy, I think two `extra' derivatives  are always needed. At any rate I think two extra derivatives will always be more than enough!]}
In this general case, we find that many results hold with minor modifications. For example, we find that \eqref{eq:chi0Diff-FinalResult} generalises to
\begin{equation} \label{intro:eqn:decay2}
 \chi_0(i) - \chi_0(i-1) = 1 - \frac{i^{-(a-1)}}{Z(V)(a - 1)} + o\big(i^{-(a-1)}\big),  
\end{equation}
where $Z(V)$ is defined by
\[
 Z(V) := \sum_{k=1}^{\infty} V''(k) k^2.
\]

At the end of each subsection, we outline how the argument for $-a$-homogeneous $V$ extends to those potentials that satisfy \eqref{eq:GenPotential-Derivs}. For clarity of the arguments, however, we only present detailed results for $V(x) = |x|^{-a}$.

\subsection{Asymptotic analysis using the bulk ansatz}
\label{sec:Formal2-bulk}

We obtain asymptotic solutions for $x(i)$ in both the bulk and boundary layer regimes using the method of matched asymptotic expansions. The method that we use involves matching with an intermediate variable, and is analogous to the methods used in \cite{Hall2010,Voskoboinikov2009}. Whereas \cite{Hall2010,Voskoboinikov2009} concentrate on leading-order matching, we use the method to obtain higher order corrections. We begin by introducing a continuum bulk ansatz, $x(i;n) = \xi(in^{-1};n)$, which we assume to be valid when $i \gg 1$ and $n-i \gg 1$. At the same time, we introduce a discrete boundary layer ansatz, $x(i;n) = n^{-1} \chi(i;n)$, which we assume to be valid when $i \ll n$. Thus, both ansatzes are assumed to be valid asymptotic expansions when $1 \ll i \ll n$.

This means that we can introduce an arbitrary $K$ with $1 \ll K \ll n$ and use the boundary layer ansatz for $x(i;n)$ when $i \leq K$ or $i \geq n- K$ and use the bulk ansatz for $x(i;n)$ when $K < i < n-K$. By the principles that underly the method of matched asymptotic expansions, the precise dependence of $K$ on $n$ should not matter; the behaviour of $\xi(s;n)$ as $s \to 0$ should match with the behaviour of $\chi(i;n)$ as $i \to \infty$ so as to yield consistent asymptotic expressions for $x(i)$ when $1 \ll i \ll n$ regardless of whether $i$ is treated as being in the bulk regime or the boundary layer regime. We think of $K$ as an arbitrary intermediate point where we connect the bulk ansatz with the boundary layer ansatz.

We make the following assumptions about the behaviour of $\xi(s;n)$ and $\chi(i;n)$:
\begin{description}
 \item[(MinSpacing)] There exists $M > 0$ such that $\chi(i+1) - \chi(i) \geq M$ whenever $i \ll n$, and $\xi'(s) \geq M$ whenever $s \gg n^{-1}$ and $1 - s \gg n^{-1}$.
 \item[($\xi$-Smooth)] $\xi \in C^\infty((\eta,1-\eta))$ for any choice of $\eta$ where  $n^{-1} \ll \eta \ll 1$.
 %\item[($\xi_0$-Bdry)] $\xi_0 \in C^\infty([0,1])$, and $\xi_0(0) =0$ and $\xi_0(1) = 1$.
\end{description}
%\PATRICK{[Am I right that instead of $C^\infty$ you actually need $C^k$ where $k$ is large enough to account for all derivatives of $\xi$ and $\xi_0$ that appear in the higher order equations? If so, it might be a hassle to describe it as such, and maybe the easiest way is indeed to keep the $C^\infty$ there. Secondly, you have already shown in (19) that $\xi_0 (s) = s$. This makes item 3 redundant, and it has moreover the confusing effect that it implies that we do not know what $\xi_0$ is yet. Also, I have not found you ever referring to this assumption explicitly anyway ;). Maybe it makes the whole of \S 3 easier by using from the start that $\xi_0 (s) = s$?]}
%\CAMERON{[Yes on all counts. As you take more and more terms, you require more and more differentiability of $\xi$, so I feel like it's safest just to say that $\xi$ is completely smooth (sufficiently differentiable?). And yes, I was playing with these condition in order to work out what I `really' should have used in the previous section. The differentiability of $\xi_0$ on $(0,1)$ is inherent in {\bf($\xi$-Smooth)}, and the boundary condition (and continuity on the closed interval) can be thought of as consequences of the matching process, so getting rid of the extra condition is fine.]}
As previously, these statements must all hold true in the asymptotic limit as $n \to \infty$ where $\xi$ and $\chi$ are replaced with
\[
 \xi(s) = \sum_{k=0}^{Q} n^{-b_k} \xi_k(s) + O\big(n^{-b_{Q+1}}\big), \quad \text{and} \quad \chi(i) = \sum_{j=0}^{P} n^{-\beta_j} \chi_j(i) + O\big(n^{-\beta_{P+1}}\big),
\]
respectively for any choices of $P$ and $Q$.

We note that {\bf(MinSpacing)} implies that $\chi(i+1) - \chi(i) \geq M$ whenever $i \leq K$, and equally that
\[
 n \int_{\frac{i}{n}}^{\frac{i+1}{n}} \xi(s) \, \mathrm{d} s \geq M,
\]
whenever $i \geq K$, regardless of the choice of $K$ as long as $1 \ll K \ll n$. Hence, {\bf(MinSpacing)} implies a minimum separation between particles that holds uniformly in $n$ independently of the choice of `cutoff' between the bulk region and the boundary layer region.

We also note that replacing $\chi$ and $\xi$ with their leading order approximations in {\bf(MinSpacing)} and considering the limits as $n \to \infty$, yields the result that $\chi_0(i+1) - \chi_0(i) \geq M$ and $\xi_0'(s) \geq M$ throughout. Additionally, we observe that {\bf(MinSpacing)} places growth restrictions on higher order corrections to  $\chi$ and $\xi'$. Specifically, it means that $\chi_j(i+1) - \chi_j(i)$ cannot grow (negatively) at a rate greater than $i^{\beta_j}$ as $i \to \infty$, and that $\xi_k'(s)$ cannot grow (negatively) at a rate greater than $s^{-b_k}$ as $s \to 0$.

We further use the symmetry of the problem to assert that $\xi(1-s) = 1-\xi(s)$ and that $x(n-i) = 1 - n^{-1} \chi(n-i;n)$ when $n - i = O(1)$. We also assume that the bulk ansatz and the discrete boundary layer ansatz are the only scalings that we need to consider for the method of matched asymptotic expansions. That is, we assume that there is no distinguished intermediate scaling between $i = \ord(n)$ and $i = \ord(1)$. A justification of this assumption can be obtained by using the methods described in \S\ref{sec:Formal1-BLFail}.

Given that $V(x) = |x|^{-a}$, the force balance equation from \eqref{eq:equilibrium} yields
\begin{equation}
 an^{-a-1} \bigghaa{ \sum_{k=1}^{i} 
 \big[x(i) - x(i-k)\big]^{-a-1}
 - \sum_{k=1}^{n-i} 
 \big[x(i+k) -x(i)\big]^{-a-1} }
 =0.
 \label{eq:ForceBalance-PowerLaw}
\end{equation}

In the remainder of this section, we concentrate on analysing force balance in the bulk, where $i = \ord(n)$. As described above, we split the sums into regions where we apply the continuum ansatz for $x(i\pm k)$ and regions where we apply the discrete ansatz:
\begin{multline}
 \underbrace{an^{-a-1} \sum_{k=0}^{K-1}
 \left( 
 \big[\xi(in^{-1}) - n^{-1}\chi(k)\big]^{-a-1}
 - \big[ 1 - \xi(in^{-1}) - n^{-1}\chi(k) \big]^{-a-1}
 \right)}_{S_{0}} \\
 + an^{-a-1} \bigg( \sum_{k=1}^{i-K} 
 \big[\xi(in^{-1}) - \xi([i-k]n^{-1})\big]^{-a-1}
 - \sum_{k=1}^{n-K-i} 
 \big[\xi([i+k]n^{-1}) -\xi(in^{-1})\big]^{-a-1} \bigg)
 =0. 
 \label{eq:ForceBalance-PowerLaw-ApplyAnsatzes}
\end{multline}
Since $K \ll n$, the sum marked $S_0$ in \eqref{eq:ForceBalance-PowerLaw-ApplyAnsatzes} is $o(n^{-a})$. From previously, we recognise that the leading order terms in the bulk force balance will be $\ord(n^{-1})$; hence, it will be possible to obtain expressions for $\xi(s)$ up to $o(n^{-(a-1)})$ while entirely neglecting any contributions from $S_0$. Higher order corrections to $\xi(s)$ may be obtained by expanding the summand of $S_0$ using Taylor series, and then exploiting the properties of $\chi(i)$. While it is possible to carry out these manipulations, we do not consider these high-order corrections in detail in this paper. %[Prelude to being able to say more about even higher order terms.]
%\CAMERON{[Patrick: I think I edited this last time but may not have highlighted it. I want to emphasise that corrections to $\xi(s)$ beyond $O(n^{-(a-1)})$ are meaningful, but can't be treated without dealing with $S_0$ in more detail (which is doable, but very messy, and so we avoid for the present). Is this paragraph sufficient or should I put a bit more in at the matching stage?]} \PATRICK{[it's fine :)]}

 %From previously, we recognise that the leading order terms in the bulk force balance will be $O(n^{-1})$, and so the first sum will be a small correction term for any $a > 1$.
% \[
%  an^{-a-1}\sum_{k=1}^{i-K} 
%  \big[\xi(in^{-1}) - \xi([i-k]n^{-1})\big]^{-a-1}
%  - an^{-a-1}\sum_{k=1}^{n-K-i} 
%  \big[\xi([i+k]n^{-1}) -\xi(in^{-1})\big]^{-a-1}
%  = O(n^{-a-1+\delta}).
% \]
We can therefore follow the approach used previously and neglect $S_0$. This leads us to define the following force function, $F(s)$, noting that force balance in the bulk requires $F(s) = o(n^{-a})$ for all $s = \frac{i}{n}$ where $i = \ord(n)$ and $n-i = \ord(n)$:
\[
 F(s) := 
 an^{-a-1} \Biggbhaa{ \sum_{k=1}^{\floor{sn -K}} 
 \big[\xi(s) - \xi(s-kn^{-1})\big]^{-a-1}
 - \sum_{k=1}^{\floor{n-K-sn}} 
 \big[\xi(s+kn^{-1}) -\xi(s)\big]^{-a-1} }.
\]
We now separate $F(s)$ into three parts as previously, introducing an arbitrary integer $H$ where %\CAMERON{$H = \ord(n^\frac{a+\kappa}{a+1})$ for some $0 < \kappa < 1$}: %$H = \ord\big(n^{1-\frac{\delta^*}{a}}\big)$ for some $\delta^* \in (0,1)$, so that  
$n^{\frac{a}{a+1}} \ll H \ll n$:
\begin{multline} \label{eq:FasympPwr:expanded}
 F(s) = \underbrace{a n^{-a-1}\sum_{k=1}^{H} \Big( 
 \big[\xi(s) - \xi(s-kn^{-1})\big]^{-a-1}
 - \big[\xi(s+kn^{-1}) -\xi(s)\big]^{-a-1} \Big)}_{=: S_1} \\
 + \underbrace{a n^{-a-1} \sum_{k=H+1}^{\floor{sn} - K} 
 \big[\xi(s) - \xi(s-kn^{-1})\big]^{-a-1}}_{=: S_2} \\
 - \underbrace{a n^{-a-1} \sum_{k=H+1}^{\floor{n-sn} - K}
 \big[\xi(s+kn^{-1}) -\xi(s)\big]^{-a-1}}_{=: S_3}.
\end{multline}

We note that $H \gg n^{\frac{a}{a+1}}$ places restrictions on $K$, since we require $K \gg H$ in order for the sums $S_2$ and $S_3$ to contain large numbers of terms. This lower bound on $K$ might suggest the presence of a distinguished scaling between $i = \ord(n)$ and $i = \ord(1)$, so that there is a continuum boundary layer problem to solve between the continuum bulk problem and the discrete boundary layer. We expect that the methods in \S\ref{sec:Formal1-BLFail} could be used to show that no such continuum boundary layer problem can exist and that hence the bulk ansatz is valid for all $i \gg 1$, but we do not pursue this analysis further.

%We note that $H \gg n^{\frac{a}{a+1}}$ places restrictions on $K$, since we require $K \gg H$ in order for the sums $S_2$ and $S_3$ to contain large numbers of terms. This is not a problem, since there is always an overlap region between the bulk regime and the boundary layer regime. It will always be possible to choose $K$ and $H$ so that $n^{\frac{a}{a+1}} \ll H \ll K \ll n$, regardless of the value of $a$. The fact that we need to choose $K \gg n^{\frac{a}{a+1}}$ might seem to imply that the bulk ansatz for $x(i;n)$ is only valid when $i \gg n^{\frac{a}{a+1}}$. We expect that the methods in Section \ref{sec:Formal1-BLFail} could be used to show that the bulk ansatz is valid for all $i \gg 1$, but we do not pursue this analysis further. \PATRICK{[This paragraph feels wordy to me, i.e. the opposite of concise]}

We begin our analysis of \eqref{eq:FasympPwr:expanded} by considering $S_2$.  Using the Euler--Maclaurin summation formula with an offset from the integers (see, for example, \cite{Sidi2012}), we find that
\begin{multline}
 \label{eq:Formal2-S2}
 S_2 = an^{-a} \int_{Kn^{-1}}^{s-Hn^{-1}} \frac{\mathrm{d} u}{[\xi(s) - \xi(u) ]^{a+1}} \\
     - an^{-a-1} \left(\frac{1}{2 [\xi(s) - \xi(s-Hn^{-1})]^{a+1}} + \frac{B_1(\{sn\})}{[\xi(s) - \xi(Kn^{-1})]^{a+1}} \right) \\
     + a(a+1) n^{-a-1} \int_{Kn^{-1}}^{s-Hn^{-1}} \frac{B_1(\{sn+un\}) \xi'(u)}{[\xi(s) - \xi(u) ]^{a+2}}  \, \mathrm{d}u, 
\end{multline}
%\CAMERON{[corrected factor of $n$ inside Bernoulli polynomials.]}
where $B_1(\{\cdot\})$ is the 1-periodic extension of the first Bernoulli polynomial. Using {\bf(MinSpacing)} we observe that $\xi(s) - \xi(s-Hn^{-1}) \geq MHn^{-1}$ and that $\xi(s) - \xi(Kn^{-1}) = \ord(1)$. Using H\"older's inequality to show that the integral remainder term is asymptotically no larger than the terms on the second line of \eqref{eq:Formal2-S2}, we therefore find that
\[
 S_2 = an^{-a} \int_{Kn^{-1}}^{s-Hn^{-1}} \frac{\mathrm{d} u}{[\xi(s) - \xi(u) ]^{a+1}} + O(H^{-a-1}).
\]
%\PATRICK{[Before I thought that Euler-MacLaurin with a first order error term is too much, but by trying an alternative, I found out that the problem is not as easy as it looked to me on first sight. Hence I agree that this is the neatest way to treat it, even though we get 3 error terms, two of which are one order lower than the other.]}
An identical argument applies to $S_3$. Using the fact that $H \gg n^{\frac a{1+a}}$ and $K \ll n$, we can combine the expansions of $S_2$ and $S_3$ to show that
\begin{equation}
 \label{eq:Formal2-S2S3}
 S_2 + S_3 
 = an^{-a} \left[ \int_{0}^{s-Hn^{-1}} \frac{\mathrm{d}u}{[\xi(s) - \xi(u)]^{a+1}} 
 -  \int_{s+Hn^{-1}}^{1} \frac{\mathrm{d}u}{[\xi(u) - \xi(s)]^{a+1}} \right]
 + o(n^{-a}).
\end{equation}

% 
% Following the same argument as in Section \ref{sec:Formal1}, we find that
% \begin{equation}
% \label{eq:Formal2-S2S3}
%   |S_2| + |S_3| 
%   \leq 2 \sum_{k=1}^\infty a M^{-a-1} (H + k)^{-a-1}
%   \lesssim H^{-a} = O(n^{-a+\delta^*}).
% \end{equation}
% %Thus, $S_2$ and $S_3$ will only make a small contribution to the force balance.
% %By making the specific choice $H = \ord(n^{1-\frac{\delta}{a}}$ for some $\delta \in (0,1)$, we find that $|S_2| + |S_3| =  O(n^{-(a-\delta)})$.

Now consider $S_1$. Using Taylor's theorem, {\bf($\xi$-Smooth)} implies that $\xi(s\pm kn^{-1})$ can be approximated by the series
\[
 \xi(s\pm kn^{-1}) \sim
 \xi(s) \pm \xi'(s) kn^{-1} + \tfrac{1}{2} \xi''(s) k^2n^{-2} + \ldots,
\]
which is asymptotic for any $k \ll n$ and $s \gg n^{-1}$. Since $k \leq H \ll n$ in $S_1$ and $s = \ord(1)$ in our present analysis, we can apply this asymptotic expansion throughout. %\CAMERON{[ASIDE: I think this is the point where the `changing error' that's worrying Patrick creeps in. If we consider $s = s(n)$ so that $s \to 0$ as $n \to \infty$, but $s \gg n^{-1}$, we will find that the Taylor expansion is still asymptotic, but the size of the error terms is not what we would `normally' expect, since $\xi$ and its derivatives may not be bounded as $s \to 0$.]}
This yields
\begin{multline}
 \label{eq:Formal2-S1First}
 S_1 \sim
 a \sum_{k=1}^{H} k^{-a-1} \Big( 
 \big[\xi'(s) - \tfrac{1}{2} k n^{-1}\xi''(s) + \tfrac{1}{6} k^2 n^{-2} \xi'''(s) + \ldots \big]^{-a-1} \\
 - \big[\xi'(s) + \tfrac{1}{2} k n^{-1}\xi''(s) + \tfrac{1}{6} k^2 n^{-2} \xi'''(s) + \ldots\big]^{-a-1} \Big).
\end{multline}
Noting that $k n^{-1} \ll 1$, we can apply the binomial series and rearrange to obtain
\begin{multline}
 S_1 \sim
 \sum_{k=1}^{H} k^{-a-1} \bigg( 
 (-a)_2 \xi''(s) \big[\xi'(s)\big]^{-a-2} k n^{-1}
 + \Big[
 \tfrac{1}{12} (-a)_2 \big[\xi'(s)\big]^{-a-2} \xi''''(s) \\
 + \tfrac{1}{6}(-a)_3 \big[\xi'(s)\big]^{-a-3} \xi''(s) \xi'''(s)
 + \tfrac{1}{24} (-a)_4 \big[\xi'(s)\big]^{-a-4} \big[\xi''(s)\big]^3 
 \Big]
  k^3 n^{-3}
 + O(k^5 n^{-5}) \bigg)
 \label{eq:S1General-NoBell}
\end{multline}
where $(\cdot)_r$ is the Pochhammer symbol, defined so that $(\alpha)_r := \alpha \cdot (\alpha - 1) \cdots (\alpha - r + 1)$. Since the summand in equation \eqref{eq:S1General-NoBell} is obtained by taking compositions of functions defined as formal series, we can use the properties of partial Bell polynomials (see, for example, \cite{ComtetComb,Wang2009}) to obtain a general expression for the terms in the summand of \eqref{eq:S1General-NoBell}. Specifically, we find that
\begin{equation}
\label{eq:Formal2-S1Bell}
 S_1 \sim \sum_{k=1}^{H} \sum_{p=0}^\infty 
 \mathcal{B}_p[\xi](s)
  k^{-a+2p} n^{-(2p+1)},
\end{equation}
where
\begin{equation}
 \label{eq:Ap-Defn}
 \mathcal{B}_p[\xi](s) 
 := \left(\frac{2}{(2p+1)!} \sum_{q=1}^{2p+1}Y_{2p+1,q}\left[\frac{\xi''(s)}{2},\frac{\xi'''(s)}{3},\ldots,\frac{\xi^{(2p-q+3)}}{2p-q+3} \right] (-a)_{q+1} \big[\xi'(s)\big]^{-a-1-q}\right),
\end{equation} 
and $Y_{p,q}(t_1,t_2,\ldots,t_{p-q+1})$ is a partial Bell polynomial. These polynomials are defined by the expression
\begin{equation*}
  Y_{p,q}(t_1,t_2,\ldots,t_{p-q+1}) 
 = \sum \frac{p!}{r_1! \, r_2! \, \cdots \, r_{p-q+1}!} 
 \left(\frac{t_1}{1!} \right)^{r_1} \, \left(\frac{t_2}{2!} \right)^{r_2} \, \cdots \, \left(\frac{t_{p-q+1}}{(p-q+1)!} \right)^{r_{p-q+1}},
\end{equation*}
where the sum is taken over all integer sequences $\{r_1,\,r_2,\,\ldots,\,r_{p-q+1}\}$ where
\begin{equation*}
 %\label{BellSumCond}
 \sum_{k=1}^{p-q+1} k r_k = p, \quad \text{and} \quad \sum_{k=1}^{p-q+1} r_k = q.
\end{equation*}
As described in \cite{ComtetComb}, partial Bell polynomials have the property that
\begin{equation*}
% \label{BellPolysForFormalSeries}
 \frac{1}{q!} \left[\sum_{j=1}^\infty \frac{x^{j}}{j!} \, t_j \right]^q = \sum_{p=q}^{\infty} Y_{p,q}(t_1,t_2,\ldots,t_{p-q+1}) \, \frac{x^p}{p!}.
\end{equation*}
It is this property of partial Bell polynomials that makes it possible to obtain \eqref{eq:Ap-Defn} from \eqref{eq:Formal2-S1First}.

Since $kn^{-1} \ll 1$, it is possible to swap the order of summation in \eqref{eq:Formal2-S1Bell} while still retaining asymptoticity:
\begin{equation*}
  S_1 \sim \sum_{p=0}^\infty 
 \mathcal{B}_p[\xi](s) \, n^{-(2p+1)} \sum_{k=1}^{H}
  k^{-a+2p}.
\end{equation*}
To evaluate the sum over $k$, we note from \cite{Hall2010} that the asymptotic behaviour of the generalised harmonic numbers is given by
\begin{equation*}
%\label{eq:HarmonicNumber}
 \sum_{k=1}^{H} k^{-r} = 
 \begin{cases}
  \zeta(r) - \frac1{r-1} {H^{-(r-1)}} + O(H^{-r}), & r > 1, \\
  \log(H) + \gamma + O(H^{-1}), & r = 1, \\
  O(H^{1-r}), & r < 1,
 \end{cases}
\end{equation*}
where $\gamma$ is the Euler--Mascheroni constant. Then, from $n^{\frac{a}{a+1}} \ll H \ll n$ we obtain
\begin{equation*}
 %\label{eq:HarmonicNumber-2}
  \sum_{k=1}^{H}  k^{-a+2p} n^{-(2p+1)} = 
 \begin{cases}
  \zeta(a-2p)n^{-(2p+1)} - \frac{(Hn^{-1})^{-(a-2p-1)}}{a-2p-1}n^{-a} + o(n^{-a}), & 2p < a-1, \\
  \bighaa{ \log(H) + \gamma } n^{-a} + o(n^{-a}) & 2p = a-1, \\ 
  o(n^{-a}), & 2p > a-1, 
 \end{cases} 
\end{equation*} 
and \eqref{eq:Formal2-S1Bell} yields
\begin{multline}
\label{eq:Formal2-S1-TempForCombining}
 S_1 =\sum_{p=0}^{\ceil{\frac{a-1}{2}}-1} \left(
 \mathcal{B}_p[\xi](s) \left[\zeta(a-2p) n^{-(2p+1)} - \frac{(Hn^{-1})^{-(a-2p-1)}}{a-2p-1}n^{-a}\right]\right) \\
 + \mathbbm{1}_{a \in 2\mathbb{N}+1} \mathcal{B}_{\frac{a-1}{2}}[\xi](s) \big[ \log(H) + \gamma \big] n^{-a}
 + o(n^{-a}).
\end{multline}
%\PATRICK{[My definition of the newcommand `mathbbm' is indeed wrong. I only did it to make the code compile. Please recover your previous version of `mathbbm' (the code of which I couldn't find)]} \TOM{[This is from \texttt{\textbackslash usepackage\{bbm\}}.]}

In order to simplify this expression into a form where it can be combined with \eqref{eq:Formal2-S2S3}, it is useful to introduce finite part integration. Following \cite{Lyness1993,Monegato1998}, we define the one-sided finite part integral for 
% power functions as
% \begin{equation}
% \label{eq:FinitePartIntegral}
%  \dashint_0^y u^a \, \mathrm{d}u := 
%  \begin{cases}
%   \frac{y^{a+1}}{a+1} & a \neq -1, \\
%   \log y & a = -1.
%  \end{cases}
% \end{equation}
%This can also be extended to more general 
functions that are well-behaved apart from a possible singularity at zero, and which satisfy
\[
 \psi(u) = \sum_{j=0}^{\Upsilon-1} c_j u^{-a_j} + c_{\Upsilon} u^{-1} + O(u^{-1+\delta}), \quad \text{as }u \to 0,
\]
where $a_0 > a_1 > \ldots > a_{\Upsilon-1} > 1$ and $\delta > 0$. In this case, we define
\begin{equation}
 \label{eq:GenFinitePartIntegral} 
 \dashint_0^y \psi(u) \, \mathrm{d}u := \lim_{\eta \to 0} \left[\int_\eta^y \psi(u) \, \mathrm{d}u - \sum_{j=0}^{\Upsilon-1} \frac{c_j \eta^{-(a_j-1)}}{a_j-1} - c_\Upsilon \log \frac1\eta  \right].
\end{equation}
%Using this definition, finite part integration commutes with summation as long as there are only finitely many terms with singularities.

From \eqref{eq:Formal2-S1Bell}, we note that
\[
 a\big[\xi(s) - \xi(s-u)\big]^{-a-1}
 - a\big[\xi(s+u) -\xi(s)\big]^{-a-1} 
 = \sum_{p=0}^{\floor{\frac{a-1}{2}}} 
 \mathcal{B}_p[\xi](s)
  u^{-(a-2p)} + O\left(u^{1 + 2\floor{\frac{a}{2}} - a}\right).
\]
Then, using \eqref{eq:GenFinitePartIntegral}, we obtain
\begin{multline*}
 \dashint_{0}^{y}   
 a\big[\xi(s) - \xi(s-u)\big]^{-a-1}
 - a\big[\xi(s+u) -\xi(s)\big]^{-a-1}  \, \mathrm{d}u \\
 = -\sum_{p=0}^{\ceil{\frac{a-1}{2}}-1}
 \mathcal{B}_p[\xi](s) \frac{y^{-(a-2p-1)}}{a-2p-1}
 + \mathbbm{1}_{a \in 2\mathbb{N}+1}  \mathcal{B}_{\frac{a-1}{2}}[\xi](s) \log(y) + o(1)
\end{multline*}
as $y \rightarrow 0$. Hence, \eqref{eq:Formal2-S1-TempForCombining} becomes
\begin{multline*}
 S_1 =\sum_{p=0}^{\ceil{\frac{a-1}{2}}-1}
 \mathcal{B}_p[\xi](s) \zeta(a-2p) n^{-(2p+1)} \\
 + a n^{-a} \dashint_0^{Hn^{-1}} \big[\xi(s) - \xi(s-u)\big]^{-a-1}
 - \big[\xi(s+u) -\xi(s)\big]^{-a-1}  \, \mathrm{d}u \\
 + \mathbbm{1}_{a \in 2\mathbb{N}+1} \mathcal{B}_{\frac{a-1}{2}}[\xi](s) \big[ \log(n) + \gamma \big] n^{-a}
 + o(n^{-a}).
\end{multline*}
Combining with \eqref{eq:Formal2-S2S3}, we therefore find that
\begin{multline}
\label{eq:Formal2-ForceBalance}
 F(s) = \sum_{p=0}^{\ceil{\frac{a-1}{2}} - 1} 
 \mathcal{B}_p[\xi](s)
  \zeta(a-2p) \, n^{-(2p+1)} \\
  + \bigghaa{ \mathbbm{1}_{a \in 2\mathbb{N}+1} \mathcal{B}_{\frac{a-1}{2}}[\xi](s) \big[ \log(n) + \gamma \big]
  + a \dashint_{0}^{1} \frac{\signum(s-u)}{| \xi(s) - \xi(u)|^{a+1}} \, \mathrm{d}u } n^{-a}
  +o\big(n^{-a}\big).
\end{multline}

In the more general case where $V$ satisfies \eqref{eq:GenPotential-Derivs}, we find that much of the argument outlined in this section still holds. Since it is possible to approximate $V'(x)$ by $-ax^{-a-1}$ for large $x$, we find that $S_2+S_3$ will still be given by \eqref{eq:Formal2-S2S3}. The most significant changes required to generalise our argument involve the manipulation of $S_1$. Repeated use of Taylor series (analogous to the manipulations of $V$ in \S \ref{sec:Formal1-Bulk}) are needed to obtain a new definition for $\mathcal{B}_p$ for a general $V$; specifically, we find that $(-a)_{q+1} [\xi'(s)]^{-a-1-q}$ in \eqref{eq:Ap-Defn} should be replaced with $k^{a+1+q}V^{(q+1)}[\xi'(s) k]$. 

While it is true that 
\[
 k^{a+1+q}V^{(q+1)}[\xi'(s) k] \to (-a)_{q+1} [\xi'(s)]^{-a-1-q} \quad \text{as } k \to \infty,
\]
the fact that the modified definition of $\mathcal{B}_p$ involves $k$ creates complications for the manipulation of sums involving $k$ through the rest of the argument. Ultimately, we find that the asymptotic properties of these sums mean that the approach outlined above remains valid, and that the analogous equation to \eqref{eq:Formal2-ForceBalance} is
\begin{multline*}
%\label{eq:Formal2-ForceBalance}
 F(s) = \sum_{p=0}^{\ceil{\frac{a-1}{2}} - 1} 
 \bar{\mathcal{B}}_p[\xi](s)
  n^{-(2p+1)} \\
  + \bigghaa{ \mathbbm{1}_{a \in 2\mathbb{N}+1} \bigg[ \tilde{\mathcal{B}}_{\frac{a-1}{2}}[\xi](s) \log(n)  + \mathcal{G}[\xi] \bigg]
  + a \dashint_{0}^{1} \frac{\signum(s-u)}{| \xi(s) - \xi(u)|^{a+1}} \, \mathrm{d}u } n^{-a}
  +o\big(n^{-a}\big),
\end{multline*}
where $\bar{\mathcal{B}}_p$, $\tilde{\mathcal{B}}_{\frac{a-1}{2}}$, and $\mathcal{G}$ depend on $V$. Note that the terms which gave rise to the zeta function and Euler--Mascheroni constant in \eqref{eq:Formal2-ForceBalance} are replaced with new formulations that depend on $V$ and $\xi'$, but the overall structure of the total force from \eqref{eq:Formal2-ForceBalance} remains the same. We find that
\begin{equation}
\label{eq:barA0-Defn}
 \bar{\mathcal{B}}_0[\xi](s) = \xi''(s) \sum_{k=1}^{\infty} V'' [\xi'(s) k] k^{2}, 
\end{equation}
and that $\bar{\mathcal{B}}_p[\xi]$, $\tilde{\mathcal{B}}_{\frac{a-1}{2}}[\xi]$ and $\mathcal{G}[\xi]$ all evaluate to the zero function when $\xi$ is affine. These observations enable us to extend the results of the following section to more general potentials $V$ that satisfy \eqref{eq:GenPotential-Derivs}.

% \CAMERON{While we will not discuss the general problem in detail, it may be noted that the analysis in this section can be extended to more general functions $V(x)$ as long as $V^{(r)}(x) \sim (-a)_r x^{-a-r}$ as $x \to \infty$. The analysis of $S_2$ and $S_3$ is identical, but the expressions for $\mathcal{B}_p[\xi]$ obtained by expanding
% \[
%  V'\big[\xi'(s) - \tfrac{1}{2} k n^{-1}\xi''(s) + \tfrac{1}{6} k^2 n^{-2} \xi'''(s) + \ldots\big]
% \]
% become more complicated that the expressions given in \eqref{eq:Ap-Defn}, and the sums in \eqref{eq:HarmonicNumber-2} are no longer so simple.
% }

\subsection{Solving for higher order corrections in the bulk}
\label{sec:Formal2-Solve} 

We now return to the case where $V(x) = |x|^{-a}$ and we seek an asymptotic expansion of $\xi(s)$ that will enable \eqref{eq:ForceBalance-PowerLaw-ApplyAnsatzes} to be satisfied for integers $i$ where $i \gg K$ and $n-i \gg K$. If we restrict our analysis to corrections up to $\ord\big[n^{-(a-1)}\big]$, we find that this is equivalent to seeking $\xi(s)$ so that $F(s) = o(n^{-a})$, and hence we can make immediate use of \eqref{eq:Formal2-ForceBalance}. Thus, we begin by expanding $\xi(s)$ as an asymptotic series as follows:
\begin{equation} \label{for:xi:exp:a:higher:orders:PversionInText}
  \xi(s) = \xi_0(s) + \sum_{k=1}^{\bar{Q}} n^{-b_k} \xi_k(s) + n^{-(a-1)} \tilde \xi(s) + o(n^{-(a-1)}),
  \qquad 0 < b_1 < \ldots < b_{\bar{Q}} < a-1;
\end{equation}
where $\bar{Q}$ may perhaps be infinite or zero. %We then substitute this into $F(s) = o(n^{-a})$ where $F(s)$ is taken from  \eqref{eq:Formal2-ForceBalance}.}

% Having obtained several terms in the asymptotic expansion of $F(s)$, we can now consider what functions $\xi(s)$ will enable \eqref{eq:ForceBalance-PowerLaw-ApplyAnsatzes} to be satisfied for integers $i$ where $K < i < n - K$. We begin by expanding $\xi(s)$ as an asymptotic series as $n \rightarrow \infty$,
% \[
%   \xi(s) \sim \xi_0(s) + n^{-b_1} \xi_1(s) + n^{-b_2} \xi_2(s) + \ldots
% \]
% 
% seeking \eqref{eq:ForceBalance-PowerLaw-ApplyAnsatzes}. }
% 
% 
% Having obtained several terms in the asymptotic expansion of $F(s)$, we can now consider what functions $\xi(s)$ satisfy $F(s) = o(n^{-a})$. 
% % Since $n^{-a-1+\delta} \ll n^{-a+\delta}$, this is equivalent to solving
% % \begin{equation}
% % \label{eq:Formal2-ForceBalance}
% %  \sum_{p=0}^{\ceil{\frac{a-1}{2}} - 1} 
% %  \mathcal{B}_p[\xi](s)
% %   \zeta(a-2p) n^{-(2p+1)}
% %   = o\big(n^{-a+\delta}\big),
% % \end{equation}
% We begin by expanding $\xi(s)$ as an asymptotic series as $n \rightarrow \infty$,
% \begin{equation} \label{for:xi:exp:a:higher:orders}
%   \xi(s) \sim \xi_0(s) + n^{-b_1} \xi_1(s) + n^{-b_2} \xi_2(s) + \ldots
% \end{equation}
% and we substitute this into $F(s) = o(n^{-a})$ where $F(s)$ is taken from \eqref{eq:Formal2-ForceBalance}.

On substituting \eqref{for:xi:exp:a:higher:orders:PversionInText} into \eqref{eq:Formal2-ForceBalance}, we find that the largest nontrivial terms are recovered at $O(n^{-1})$. These yield the result that $\zeta(a)\mathcal{B}_0[\xi_0](s) = 0$. Using the definition of $\mathcal{B}_p[\xi]$ in \eqref{eq:Ap-Defn}, this becomes
\[
 \zeta(a) (-a)_2 \xi_0''(s) \big[\xi_0'(s)\big]^{-a-2} = 0,
\]
and hence $\xi_0(s)$ is affine. More specifically, we can use the leading order boundary conditions from \eqref{eq:xi0-BCs} to conclude that $\xi_0'(s) = 1$ and $\xi_0(s) = s$.

%\PATRICK{[I still got confused about the $\tilde \xi$ term. Can't we be a bit more explicit here that the $b_k$-term in the expansion enters at first in the $O(n^{-1-b_k})$-term of the expansion for $F(s)$? Since $F(s)$ is only found up to $O(n^{-a})$, it makes no sense to add terms in the expansion of $\xi$ for $b_k > a-1$. Therefore, we postulate a \emph{finite} expansion of $\xi$:  
%\begin{equation} \label{for:xi:exp:a:higher:orders:Pversion}
%  \xi(s) \sim \xi_0(s) + \sum_{k=1}^{\mathsf K} n^{-b_k} \xi_k(s) + n^{-(a-1)} \tilde \xi,
%  \qquad 0 < b_1 < \ldots < b_{\mathsf K} < a-1.
%\end{equation}
%At this point we can postulate $\mathsf K$ to be finite, which you show later in the matching procedure. I think \eqref{for:xi:exp:a:higher:orders:Pversion} is convenient for \S 3.4, in particular to use it in \eqref{eq:IntermediateAsymptotics-Prelim} such that $\tilde \xi$ appears. Also, the condition $b_k < a-1$ is vital for the matching to work, and \eqref{for:xi:exp:a:higher:orders:Pversion} includes this clearly.]} \CAMERON{[I was concerned about losing the fact that $\xi$ should be the `true' asymptotic solution, regardless of the simplifications used to obtain $F$, but I agree that an earlier statement of $\xi$ could be helpful. I've changed the introduction above---does that help?]} \PATRICK{[Ah, I see. But even the current $\xi$ is still the 'true' one right? We have just hidden part of it in an $o$-term, and that is exactly the part which we cannot characterise because of our simplifications on $F(s)$.]}

In order to characterise the next nontrivial term in the expansion of $F(s)$, we assume for the moment that $b_1 < a-1$ to avoid dealing with the singular integral term at $O(n^{-a})$. Since $\xi_0'(s)$ is constant and nonzero, it follows from \eqref{eq:Ap-Defn} that $\mathcal{B}_p[\xi_0] \equiv 0$ for all $p$. Hence, the next nontrivial terms in the expansion of \eqref{eq:Formal2-ForceBalance} appear at $O(n^{-1-b_1})$, where we find that
\[
 \zeta(a) (-a)_2 \xi_1''(s) = 0.
\]
Again, we conclude that $\xi_1(s)$ is affine and we find that $\mathcal{B}_p[\xi_0 + n^{-b_1}\xi_1] \equiv 0$ for all $p$. We cannot apply boundary conditions to $\xi_1(s)$ at this stage, since the boundary conditions on $\xi_1$ will depend on the matching between the bulk solution and the boundary layer solution. However, we can use the symmetry of the force balance problem to conclude that $\xi_1(s) = -\xi_1(1-s)$ and hence
\[
 \xi_1(s) =(s - \tfrac12)  p_1 
\]
where $p_1  = \xi_1'(s)$ is a constant to be determined from matching with the boundary layer.

As long as $b_k < a-1$ we can apply the same argument to show that $\xi_k$ is affine. We will use this freedom in the choice of $b_k$ later on to match with the boundary layer. For now, we rewrite the expansion of $\xi$ in \eqref{for:xi:exp:a:higher:orders:PversionInText} as
\begin{equation*}
%\label{eq:Formal2-xiSolnUnusualFormat}
 \xi(s) = s + \bar{p}(n) (s - \tfrac12) + n^{-(a-1)}\tilde{\xi}(s) + o(n^{-(a-1)})
\end{equation*}
where $\bar{p} := p_1 n^{-b_1} + p_2 n^{-b_2} + \ldots$, so that $\bar{p} \ll 1$. This yields
\[
  n^{-a} \left[ \zeta(a) (-a)_2 \tilde{\xi}''(s) + a \dashint_0^{1} \frac{\signum(s-u)}{|s-u|^{a+1}} \, \mathrm{d}u  \right] = o(n^{-a}).
\]
Next we solve for $\tilde \xi$. Since 
\begin{equation*}
  a \dashint_0^{1} \frac{\signum(s-u)}{|s-u|^{a+1}} \, \mathrm{d}u
  = (1-s)^{-a} - s^{-a},
\end{equation*}
we have that
\begin{equation*}
  \tilde{\xi}''(s) = \frac{s^{-a} - (1-s)^{-a}}{\zeta(a) (-a)_2}.
\end{equation*}
By using again the symmetry of the force balance (i.e.~$\tilde \xi(s) = - \tilde \xi(1-s)$), we obtain
\begin{equation}
 \tilde{\xi}(s) =
 \left\{ \begin{aligned}
    &\frac{s^{-(a-2)} - (1-s)^{-(a-2)}}{\zeta(a) (-a+2)_4} + (s - \tfrac12)\tilde{p} , && a \neq 2 \\
  &\frac{\log (1-s) - \log(s) }{\pi^2} + (s - \tfrac12)\tilde{p} , && a = 2.
  \end{aligned} \right. 
 \label{eq:Formal2-xiSolnLogs}
\end{equation}
where $\tilde{p}$ is a constant to be determined from matching with the boundary layer. %\CAMERON{In the case where $a = 2$, we will find that $\tilde{p}$ must be logarithmically large in $n$ in order to achieve good matching; this is not a contradiction with stating that $\tilde{p}$ is constant, since we only require $\tilde{p}$ to be independent of $s$.}
%\CAMERON{[The change above is fine with me.]}

% Thus, analysis of the bulk ansatz yields
% % \begin{equation}
% %  \xi'(s) = 1 + p_1 n^{-b_1} + p_2 n^{-b_2} + \ldots + O\big(n^{-(a-1)}\big),
% % \end{equation} 
% % and hence
% \begin{equation}
%  \xi(s) = s + ( p_1 s - \tfrac{p_1}{2}) n^{-b_1}  + ( p_2 s - \tfrac{p_2}{2}) n^{-b_2} + \ldots + \tilde{\xi}(s) n^{^{-(a-1)}} +   o\big(n^{-(a-1)}).
% \end{equation} 
% %Higher order terms could be obtained by using Euler--Maclaurin summation to obtain asymptotic approximations of $S_2$ and $S_3$ and by considering the influence of the boundary layer terms discarded in order to obtain \eqref{eq:FasympPwr:expanded}. However, this leads to a significant increase in complexity, and so we stop our analysis before $O\big(n^{-(a-1)}\big)$.

In the more general case where $V$ satisfies \eqref{eq:GenPotential-Derivs}, we still find that $\xi_k(s) = p_k(s - \tfrac{1}{2})$ whenever $b_k < a-1$ as a consequence of the fact that $\bar{\mathcal{B}}_p[\xi] \equiv 0$ when $\xi$ is affine. We can also evaluate $\tilde{\xi}$ by using the definition of $\bar{\mathcal{B}}_0$ given in \eqref{eq:barA0-Defn}. This yields 
\begin{equation}
 \tilde{\xi}(s) =
 \left\{ \begin{aligned}
    &\frac{s^{-(a-2)} - (1-s)^{-(a-2)}}{Z(V) (-a+2)_2} + (s - \tfrac12)\tilde{p} , && a \neq 2 \\
  &\frac{\log (1-s) - \log(s) }{Z(V)} + (s - \tfrac12)\tilde{p} , && a = 2.
  \end{aligned} \right. 
 \label{eq:Formal2-xiSolnLogs-General}
\end{equation}
where 
\[ Z(V) := \sum_{k=1}^{\infty} V''(k) k^2.\]

\subsection{Asymptotic analysis in the boundary layer}
\label{sec:Formal2-BL}

We now return to assuming $V(x) = |x|^{-a}$ and seek solutions for $\chi_j(i)$ by considering the case where $i = \ord(1)$ in \eqref{eq:ForceBalance-PowerLaw}. We recall that we introduced $K$ at the beginning of \S \ref{sec:Formal2-bulk} so that $1 \ll K \ll n$, and hence $K$ is in the intermediate region where both the boundary layer ansatz and the bulk ansatz can be used. %In order to analyse the boundary layer in detail, it is convenient to pick $K$ to be as large as possible; in this section, we therefore choose $K$ to be arbitrarily close to $n$ while maintaining the requirement that $K \ll n$.

%so that the boundary layer ansatz can be used wherever $i \leq K$, and that this implies that $K \ll n$. Since , we choose $K$ to be arbitrarily close to $n$
%$K$ with $1 \ll K \ll n$ and use the boundary layer ansatz for $x(i;n)$ when $i \leq K$

Assuming $i = \ord(1)$, we split the sums in \eqref{eq:ForceBalance-PowerLaw} to obtain
\begin{multline}
\label{eq:Formal2-BLForceBalance1}
 \underbrace{a \sum_{\substack{k=0 \\ k\neq i}}^{K} 
 \frac{\signum(i-k)}{|\chi(i) - \chi(k)|^{a+1}}}_{S_4}
 -\underbrace{an^{-a-1}\sum_{k=K+1}^{n-K-1} 
 \big[\xi(kn^{-1}) -n^{-1}\chi(i)\big]^{-a-1}}_{S_5} \\
 -\underbrace{an^{-a-1}\sum_{k=0}^{K} 
 \big[1-n^{-1}\chi(k) -n^{-1}\chi(i)\big]^{-a-1}}_{S_6} 
 =0.
\end{multline}

Since all the terms in the summand of $S_6$ are $O(1)$, we find that $S_6 = O(Kn^{-a-1}) = o(n^{-a})$. Moreover, using {\bf(MinSpacing)} we obtain that
\begin{equation*}
  S_5 
  \lesssim n^{-a-1} \sum_{k=K+1}^{n-K-1} \big[M kn^{-1}\big]^{-a-1}
  \lesssim \sum_{k=K}^{\infty} k^{-a-1}
  = O (K^{-a}),
\end{equation*}
and hence \eqref{eq:Formal2-BLForceBalance1} becomes
%By taking $K$ to be as large as possible while still retaining $K \ll n$, we therefore find that \eqref{eq:Formal2-BLForceBalance1} implies
\begin{equation}
\label{eq:Formal2-BLProb}
 a \sum_{\substack{k=0 \\ k\neq i}}^{K} 
 \signum(i-k)|\chi(i) - \chi(k)|^{-a-1} = O\big(K^{-a}\big).
\end{equation}
%for any $\delta > 0$.

% \CAMERON{Now, even without the full analysis described in \S\ref{sec:Formal2-Matching}, we can see from }

Following the methods described in \S \ref{sec:Formal1-BL}, it follows that the leading order solution in the boundary layer is a solution to the infinite system of algebraic equations
\begin{equation*}
%\label{eq:Formal2-LeadingBLProb}
  a \sum_{\substack{k=0 \\ k\neq i}}^{\infty} 
 \signum(i-k)|\chi_0(i) - \chi_0(k)|^{-a-1} = 0, \quad i = 1,\,2,\,\ldots
\end{equation*}
subject to the matching condition $\chi_0(i) - \chi_0(i-1) \rightarrow 1$ as $i \rightarrow \infty$.

To obtain higher order corrections, we begin by assuming an asymptotic power series expansion for $\chi(i)$. As in \S\ref{sec:Formal2-Solve}, we will seek solutions up to $\ord[n^{-(a-1)}]$ and thus it is convenient to introduce a power series of the form
\begin{equation}
\label{eq:Formal2-chiAsymp-betaGeneral}
  \chi(i) = \chi_0(i) + \sum_{j=1}^{\bar{P}} n^{-\beta_j} \chi_j(i) + n^{-(a-1)} \tilde \chi(i) + o[n^{-(a-1)}],
  \qquad 0 < \beta_1 < \ldots < \beta_{\bar{P}} < a-1;
\end{equation}
where $\bar{P}$ may be zero or infinite.

As we discuss in \S\ref{sec:Formal2-Matching}, asymptotic matching implies that $\chi_j(i) - \chi_j(i-1)$ must have a finite limit as $i \to \infty$ for any $\beta_j < a-1$, and an identical result holds for  $\tilde{\chi}(i) - \tilde{\chi}(i-1)$. The fact that these limits are finite enables us to make significant simplifications after we substitute \eqref{eq:Formal2-chiAsymp-betaGeneral} into \eqref{eq:Formal2-BLProb}. Applying the multinomial expansion, we see that this yields
\begin{multline}
 \label{eq:Formal2-ExpandedSums}
 a \sum_{\substack{k=0 \\ k\neq i}}^{K} 
 \signum(i-k)|\chi(i) - \chi(k)|^{-a-1}  
 \sim a \sum_{\substack{k=0 \\ k\neq i}}^{K} 
 \signum(i-k)|\chi_0(i) - \chi_0(k)|^{-a-1} \\
 - n^{-\beta_1} a(a+1) \sum_{\substack{k=0 \\ k\neq i}}^{K} |\chi_0(i) - \chi_0(k)|^{-a-2} \left[\chi_1(i) - \chi_1(k) \right] \\
 - n^{-\beta_2} a(a+1) \sum_{\substack{k=0 \\ k\neq i}}^{K} |\chi_0(i) - \chi_0(k)|^{-a-2} \left[\chi_2(i) - \chi_2(k) \right] \\
 + \frac{n^{-2\beta_1}}{2} a(a+1)(a+2) \sum_{\substack{k=0 \\ k\neq i}}^{K} |\chi_0(i) - \chi_0(k)|^{-a-3} \left[\chi_1(i) - \chi_1(k) \right]^2 + \ldots = O(K^{-a}).
\end{multline}
Since $\chi_j(i) - \chi_j(k) \sim C_j (i- k)$ for some constant $C_j$ as $k \to \infty$, it follows that 
\begin{gather*}
 \sum_{k=K+1}^{\infty} 
 \signum(i-k)|\chi_0(i) - \chi_0(k)|^{-a-1} = O(K^{-a}), \\
  \sum_{k=K+1}^{\infty} |\chi_0(i) - \chi_0(k)|^{-a-2} \left[\chi_1(i) - \chi_1(k) \right] = O(K^{-a}),
\end{gather*}
and so on. This enables us to extend the sums in \eqref{eq:Formal2-ExpandedSums} to infinity without introducing significant errors. Choosing $K$ so that $n^{1-\frac{1}{a}} \ll K \ll n$, we therefore find that
\eqref{eq:Formal2-BLProb} becomes
\begin{multline*}
 %\label{eq:Formal2-HigherBLProbs}
 a \sum_{\substack{k=0 \\ k\neq i}}^{\infty} 
 \signum(i-k)|\chi_0(i) - \chi_0(k)|^{-a-1} \\
 - n^{-\beta_1} a(a+1) \sum_{\substack{k=0 \\ k\neq i}}^{\infty} |\chi_0(i) - \chi_0(k)|^{-a-2} \left[\chi_1(i) - \chi_1(k) \right]
 + \ldots = o[n^{-(a-1)}].
\end{multline*}

Collecting $\ord(n^{-\beta_1})$ terms, we obtain the following infinite homogeneous linear system for system for $\chi_1(i)$:
\begin{equation}
 \label{eq:chi1-HomoSys}
  -a(a+1) \sum_{\substack{k=0 \\ k\neq i}}^{\infty} |\chi_0(i) - \chi_0(k)|^{-a-2} \left[\chi_1(i) - \chi_1(k) \right] = 0, \quad i = 1,\,2,\,\ldots,
\end{equation}
where $\chi_1(0) = 0$ due to the fact that $x(0) = 0$. This system must be solved subject to some matching condition that relates the behaviour of $\chi_1(i)$ as $i \to \infty$ to the behaviour of the bulk solution as $s \to 0$. Since $|\chi_0(i) - \chi_0(k)|^{-a-2} = O(k^{-a-2})$ as $k \to \infty$, we observe that the sums in \eqref{eq:chi1-HomoSys} are absolutely convergent when $\chi_1(k) = O(k^{a+1-\delta})$ for some $\delta > 0$. Since asymptotic matching gives $\chi_j(k) = O(k)$ as $k \to \infty$ for all $\beta_j < a-1$, it follows that the sums in \eqref{eq:chi1-HomoSys} are absolutely convergent for any $i$.

In \S\ref{sec:Formal2-Matching}, we show that asymptotic matching can be used to determine the exponents $\beta_j$ and $b_k$. As we will see, this analysis relies on the claim that if \eqref{eq:chi1-HomoSys} is solved subject to the particular matching condition $\chi_1(i) - \chi_1(i-1) \to 0$, then the only possible solution is the trivial solution, $\chi_1(i) \equiv 0$. To prove this claim, we observe that \eqref{eq:chi1-HomoSys} is a linear equation of the type $\mathcal A \chi_1 = 0$, where we interpret $(\chi_1(i))_{i=1}^\infty$ as a sequence and $\mathcal A$ as an infinite matrix $A$ with entries
\begin{equation*}
  A_{ij} := \left\{ \begin{aligned}
    &- |\chi_0(i) - \chi_0(j)|^{-a-2},
    && \text{if } i \neq j \\
    &\sum_{\substack{k=0 \\ k\neq i}}^{\infty} |\chi_0(i) - \chi_0(k)|^{-a-2},
    && \text{if } i = j
  \end{aligned} \right\},
  \qquad \text{for } i,j \geq 1.
\end{equation*}
Since $A$ is strictly diagonally dominant and symmetric, $\zeta \mapsto \zeta^T \mathcal A \zeta$ is a
positive, strictly convex function on the space of sequences satisfying the matching condition $\zeta(i)-\zeta(i-1)\to0$, and is thus uniquely globally minimised when $\zeta=0$. Since any $\chi_1$ satisfying $\chi_1(k) = O(k)$ and $\mathcal A \chi_1 = 0$ also satisfies
$\chi_1^T \mathcal A \chi_1 = 0$, it follows that $\chi_1 = 0$, which proves the claim.

A corollary of this claim is that any solution obtained to \eqref{eq:chi1-HomoSys} subject to the matching condition $\chi_1(i) - \chi_1(i-1) \to p$ is unique, since otherwise the difference between two such solutions would be a nonzero solution to \eqref{eq:chi1-HomoSys} that satisfies $\chi_1(i) - \chi_1(i-1) \to 0$.

From the form of \eqref{eq:Formal2-ExpandedSums}, we observe that each higher correction $\chi_j$ will satisfy an infinite linear system of the form
\begin{equation}
 \label{eq:chij-HomoSys}
  -a(a+1) \sum_{\substack{k=0 \\ k\neq i}}^{\infty} |\chi_0(i) - \chi_0(k)|^{-a-2} \left[\chi_j(i) - \chi_j(k) \right] = g_j(i), \quad i = 1,\,2,\,\ldots,
\end{equation}
where $g_j(i)$ is obtained from the $\ord(n^{-\beta_j})$ terms in the multinomial expansion of $|\chi(i) - \chi(k)|^{-a-1}$, which in turn only depend on $\chi_0, \ldots, \chi_{j-1}$. Once an appropriate matching condition is specified in the form  $\chi_j(i) - \chi_j(i-1) \to q_j$ for some constant $q_j$, we find that there will be a unique solution for $\chi_j$. Similarly, $\tilde{\chi}$ will satisfy a linear system of the form given in \eqref{eq:chij-HomoSys}, and the identical style of matching condition will be required.

We note that $g_j(i)$ will only be nonzero if $\beta_j$ can be expressed as the sum of $\beta_J$ values (possibly including repetitions) where $J < j$. For example, $g_2$ will only be nonzero if $\beta_2$ is a multiple of $\beta_1$. Since the linear system for $\chi_j$ above is identical to the linear system for $\chi_1$, we see that $\chi_j(i) - \chi_j(i-1) \not \to 0$ as $i \to \infty$ is necessary for $\chi_j$ to have a nontrivial solution unless $g_j$ is nonzero. This is an important observation for performing the matched asymptotic analysis in \S \ref{sec:Formal2-Matching}.

In the more general case where $V$ satisfies \eqref{eq:GenPotential-Derivs}, we find with very minor modifications of the analysis above that each $\chi_j$ satisfies the infinite linear system
\[
  -\sum_{\substack{k=0 \\ k\neq i}}^{\infty} V''\big[\chi_0(i) - \chi_0(k)\big] \left[\chi_j(i) - \chi_j(k) \right] = g_j(i), \quad i = 1,\,2,\,\ldots,
\]
where the functions $g_j(i)$ are obtained from Taylor series expansions of $V'(\chi(i) - \chi(k))$.

\subsection{Matching between the bulk and the boundary layer}
\label{sec:Formal2-Matching}

We established in \S\ref{sec:Formal2-Solve} that $\xi$ can be expanded as an asymptotic series of the form \eqref{for:xi:exp:a:higher:orders:PversionInText} where $\xi_k(s) = (s - \tfrac12)  p_k$ and $\tilde{\xi}$ is given in \eqref{eq:Formal2-xiSolnLogs}. Additionally, we established in \S\ref{sec:Formal2-BL} that $\chi$ can be expanded as an asymptotic series of the form  \eqref{eq:Formal2-chiAsymp-betaGeneral}, where $\chi_j$ for $j > 0$ and $\tilde{\chi}$ are all solutions to infinite linear systems subject to a condition of the form $\chi_j(i) - \chi_j(i-1) \to q_j$ (and similarly for $\tilde\chi$). However, we have not yet characterised the exponents $b_k$ and $\beta_j$ in the power series  \eqref{for:xi:exp:a:higher:orders:PversionInText} and  \eqref{eq:Formal2-chiAsymp-betaGeneral}, neither have we determined the constants $p_k$, $\tilde{p}$, $q_j$ and $\tilde{q}$. We achieve this by using the method of matched asymptotic expansions.

We perform our asymptotic matching by introducing an intermediate matching variable, $R$, where $R$ is an integer with $1 \ll R \ll n$. We assert that this $R$ lies in the `overlap region', so that both the bulk ansatz and the boundary layer ansatz yield asymptotic series solutions for $x(R;n)$ when $1 \ll R \ll n$. This involves making some assumptions about the asymptoticity of the bulk and boundary layer solutions outside the domains in which they are naturally defined. For example, we recall that we assumed that $i = \ord(n)$ in order to obtain the bulk equations described in \S\ref{sec:Formal2-bulk}. We now assert that the bulk series solution obtained in \S \ref{sec:Formal2-Solve} remains valid whenever $i \gg 1$. That is, we assert that $\xi_0(Rn^{-1}) \gg n^{-b_k} \xi_k(Rn^{-1})$ for any $k > 0$ as long as $R \gg 1$. Despite the fact that $\tilde{\xi}(s)$ becomes unbounded as $s \to 0$, we observe that this assumption is consistent with comparing $\xi_0(s) = s$ with the solution for  $\tilde{\xi}$ given in \eqref{eq:Formal2-xiSolnLogs}.

%\CAMERON{[ASIDE: This isn't the usual definition of a series being asymptotic, and so I'm coyly using the terminology `valid'. The reasons for this are a little messy and nonstandard, and they're not something that one would ever worry about in practice. Essentially, a series would be asymptotic if $n^{-b_{k-1}}\xi_{k-1}(Rn^{-1}) \gg n^{-b_k}\xi_k(Rn^{-1})$, but instead it is more appropriate to use the condition $\xi_0(Rn^{-1}) \gg n^{-b_k} \xi_k(Rn^{-1})$ for any $k > 0$ . This is because all higher order problems are linear (as they always will be for any asymptotic problem), and it therefore doesn't affect the overall solution if the high-order terms reorder themselves. The condition for the loss of asymptoticity is if a higher order solution grows to be the same size as the leading order solution.]} \PATRICK{[OK, got it]}

The matching variable, $R$, is distinct from the cut-off, $K$, used in several of the sums. We introduce the matching variable in order to analyse the relationship between the solution of the discrete boundary layer problem and the solution of the continuum bulk problem, whereas we introduce $K$ in order to account for the `bulk' and `boundary layer' contributions to the force on any individual particle. %It would be useful to introduce $R$ in any discrete boundary layer problem (including, for example, boundary layers that arise when each particle only interacts with its nearest neighbours), whereas $K$ is only needed because each particle interacts with every other particle.}

% \CAMERON{DEPRECATED PARAGRAPH (EARLIER ATTEMPT) BUT MAY BE INTERESTING: When we say that the bulk ansatz and the boundary layer ansatz can both be applied when $i = R$, we mean that both $\xi(Rn^{-1};n)$ and $n^{-1}\chi(R;n)$ can be expanded as asymptotic series when $1 \ll R \ll n$, and that these series can both be used as approximations of $x(i;n)$. However, we note that the error terms obtained by truncating these series may be different from the error terms that we state in \eqref{for:xi:exp:a:higher:orders:PversionInText} and \eqref{for:chi:exp:a:higher:orders:PversionInText}, since the error terms in these original equations are applicable only when $s = \ord(1)$ for the bulk ansatz and when $i = \ord(1)$ for the boundary layer ansatz. For example, \eqref{eq:Formal2-xiSolnLogs} gives the result that $\tilde{\xi}(s) \sim \tilde \kappa s^{-(a-2)}$ for constant $\tilde \kappa$ when $a \neq 2$ and $s \to 0$. Hence, $n^{-(a-1)} \tilde{\xi}(Rn^{-1}) = \ord(R^{-(a-2)}n^{-1})$, which is larger than the $\ord(n^{-(a-1)})$ that we might have expected from the original expansion. [Does this explain the point about the changing sizes of error terms? The origin of the changing size of the error terms (in the bulk problem) is the implicit assumption that all $\xi_k$ are bounded when we apply Taylor series, whereas $\tilde{\xi}$ and higher order terms will grow.]} 

Asymptotic matching requires that $\xi(Rn^{-1})$ and $n^{-1} \chi(R)$ should be asymptotically equivalent throughout the overlap region. That is, we require that
\begin{equation}
 \label{eq:IntermediateAsymptotics-Prelim}
 \xi_0(Rn^{-1}) + n^{-b_1} \xi_1(Rn^{-1}) + \ldots \sim n^{-1} \chi_0(R) + n^{-\beta_1-1} \chi_1(R) + \ldots,
\end{equation}
for all choices of $R$ with $1 \ll R \ll n$. Each term obtained from expanding $\xi(Rn^{-1})$ under the assumption that $Rn^{-1}$ is small should match with an equivalent term obtained from expanding $n^{-1} \chi(R)$ under the assumption that $R$ is large. In the case where logarithmic terms and related complications are absent, this can be conveniently expressed using a matching table, in which the rows represent asymptotic expansions of $n^{-b_k}\xi_k(Rn^{-1})$ for small $Rn^{-1}$ and the columns represent expansions of $n^{-\beta_j - 1} \chi_j(R)$ for large $R$. Every row and column of the matching table should be a valid asymptotic series when $1 \ll R \ll n$, and every term in the interior of the table should be asymptotically larger than the terms below and to the right.

In order to construct a plausible matching table, we begin by exploiting the information that we already have about the functions $\xi_k$ and $\tilde \xi$. Specifically, we observe from our analysis in \S \ref{sec:Formal2-Solve} that
\begin{equation}
\label{eq:GenMatching-LinOnLHS}
n^{-b_k}\xi_k(Rn^{-1}) =  - \tfrac{p_k}{2} n^{-b_k} + p_k R n^{-b_k - 1}, \text{ wherever } b_k < a-1,
\end{equation}
while the further assumption that $a \neq 2$ yields
\begin{multline}
\label{eq:GenMatching-MoreComplicatedOnLHS}
 n^{-(a-1)}\tilde{\xi}(Rn^{-1}) = \frac{1 }{\zeta(a) (-a+2)_4}R^{-(a-2)}n^{-1} 
 \\
 - \left(\frac{1}{\zeta(a) (-a+2)_4} + \frac{\tilde{p}}{2} \right) n^{-(a-1)} 
 + \left(\frac{1}{\zeta(a) (-a+1)_3} + \tilde{p}\right) Rn^{-a}
 + O(R^2 n^{-(a+1)}).
\end{multline}

Based on these results, we construct the following `matching table' where each row and column can be read as an equation: \\
{\small \begin{tabular}{ c c c c c c c c c c c c}
 $x(i;n)$ & $\sim$ &  $n^{-1}\chi_0(R)$ & $+$ &  $n^{-\beta_1 - 1} \chi_1(R)$  & $+$ & $\cdots$ & $+$ &  $n^{-a} \tilde{\chi}(R)$ & $+$ & $\cdots$ \\
 $\wr$ & &  $\wr$ & &  $\wr$ &&&& $\wr$ \\
 $\xi_0(Rn^{-1})$ & $=$ & $Rn^{-1}$ & \\
  $+$ &  & $+$ & \\
 $n^{-b_1}\xi_1(Rn^{-1})$ & $=$ & $-\tfrac{p_1}{2}n^{-b_1}$ & + & $p_1 R n^{-b_1 - 1}$ \\
  $+$ &  & & & $+$ \\
 $n^{-b_2}\xi_2(Rn^{-1})$ & $=$ & && $-\tfrac{p_2}{2}n^{-b_2}$ & + & \ldots \\  
  $+$ & & $+$ &  & $+$ & & \\
  $\vdots$ & & $\vdots$ &  & $\vdots$ & & $\ddots$ \\
  $+$ & & $+$ &  & & & \\
 $n^{-(a-1)} \tilde{\xi}(Rn^{-1})$ & $\sim$ &  $\tilde \kappa_4 R^{-(a-2)}n^{-1}$  && & $+$ & $\cdots$ & $+$ & $(\tilde \kappa_3 + \tilde{p}) R n^{-a}$ & $+$ & $\cdots$ \\
  $+$ & & $+$ &  & $+$ & & & & $+$ \\
  $\vdots$ & & $\vdots$ &  & $\vdots$ & & $\ddots$ & & $\vdots$ && $\ddots$
\end{tabular}} \\
where $\tilde \kappa_r := \tfrac{1 }{\zeta(a) (-a+2)_r}$.

The matching table illustrates the fact that each term in the expansions of $n^{-b_k}\xi_k(Rn^{-1})$ given in \eqref{eq:GenMatching-LinOnLHS} and \eqref{eq:GenMatching-MoreComplicatedOnLHS} must correspond to an equivalent term in the asymptotic expansion of one of the functions $n^{-\beta_j-1} \chi_j(R)$.  While the entries  in the matching table above are based on the expansions of $\xi_k$, the columns must also be valid series. This places significant restrictions on the choices of $b_k$ and $\beta_j$; for example, inspection of the $n^{-1} \chi_0(R)$ column of the matching table strongly suggests that $b_1 = 1$.

%This places significant restrictions on the asymptotic behaviour of the functions $\chi_j(R)$. 

More rigorously, we can determine the values of $b_k$ and $\beta_j$ without appealing directly to the matching table. Since $n^{-b_k}\xi_k(Rn^{-1})$ is given by \eqref{eq:GenMatching-LinOnLHS} when $b_k < a-1$, we find that the only terms on the left hand side of \eqref{eq:IntermediateAsymptotics-Prelim} that take the form $c_{\tau\upsilon} n^{-\tau} R^{\upsilon}$ for nonzero $c_\tau\upsilon$ are terms where $\upsilon = 1$ or $\upsilon = 0$ or
%\PATRICK{[or $\to$ and?]} \CAMERON{[`or' is correct, but should have been $\leq$. This is the case which isn't dealt with by $b_k < a-1$.]}
$\upsilon \leq \tau - a + 1$. This third possibility is associated with the case where $b_k \geq a-1$ and hence $n^{-b_k}\xi_k(Rn^{-1})$ may not be linear.

Since every term on the right hand side of \eqref{eq:IntermediateAsymptotics-Prelim} must balance with an identical term on the left hand side of \eqref{eq:IntermediateAsymptotics-Prelim}, this implies that
\begin{equation}
 \label{eq:GenMatching-RHS}
 \chi_j(R) = 
 \begin{cases}
 q_{j} R + \hat{q}_{j} + O(R^{\beta_j - (a-2)} ), & \beta_j < a-2; \\
 q_{j} R +  O(R^{\beta_j - (a-2)} ), & a-2 \leq \beta_j < a-1;
 \end{cases}
\end{equation}
where $q_j$ and $\hat{q}_j$ are constants. In order to match between equivalent terms on either side of \eqref{eq:IntermediateAsymptotics-Prelim}, we find that the values of the constants $q_j$ and $\hat{q}_j$ will be associated with values of $p_k$. Since \eqref{eq:GenMatching-RHS} is concerned with the behaviour of $\chi_j(R)$ when $R$ is large, we note that differencing \eqref{eq:GenMatching-RHS} also provides justification of the fact that $\chi_j(i) - \chi_j(i-1)$ has a finite limit as $i \to \infty$ wherever $\beta_j < a-1$. More rigorously, this result could be established by exploiting the assumed differentiability of $\xi$ and considering asymptotic matching between $\xi'(Rn^{-1})$ and $\chi(R) - \chi(R-1)$.

%\PATRICK{[How does the $O$-term match with something in the expansion of $\xi$? In fact, if you put any term in the $O$-term, I think it is possible to choose $R \gg 1$ small enough to obtain a term which does \textbf{not} appear in the LHS of \eqref{eq:IntermediateAsymptotics-Prelim}, and thus the matching becomes impossible]} \CAMERON{[It matches with the terms beyond those that we've obtained so far in our expansion of $\xi$. This is why I've been keen to emphasise the fact that the bulk asymptotic solution continues beyond the point allowed by our $F(s) = o(n^{-a})$ approximation earlier.]} \PATRICK{[they can only match up if $\xi_k$ is not bounded at $0$ and $1$, because otherwise $n^{-b_k} \xi_k (R n^{-1}) = O(n^{-b_k})$. However, I thought the whole idea behind asymptotic expansions was assuming regularity on $\xi_k$ (i.e. at least boundedness) to separate effects on different length scales. But you are saying ]} \CAMERON{[We assume regularity of $\xi_k$ in their `home territory' where $s = \ord(1)$ and $1-s = \ord(1)$, but we expect the asymptoticity of the series to break down as the boundary layer is approached, which (very frequently) means that higher order terms will blow up. In the intermediate region (where we are now), the series will still be asymptotic, but $\xi_k$ will be starting to show symptoms of the loss of asymptoticity that happens when $s = \ord(n^{-1})$. Is that making sense? It is a funny concept and I know I haven't been doing a good job with it. :/ ]} %P: I think I get it

Now, let us assume that there exists some $b_k$ in the range $0 < b_k < a-1$. By matching terms on either side of \eqref{eq:IntermediateAsymptotics-Prelim} and using \eqref{eq:GenMatching-LinOnLHS}, we find that both $b_k$ and $b_k - 1$ must be values taken by exponents $\beta_j$. Since the smallest $\beta_j$ is $\beta_0 = 0$, it follows that $b_k \geq 1$. If $a < 2$, this leads to a contradiction with the requirement that $b_k < a-1$, and we would therefore conclude that there are no exponents $b_k$ in the range $0 < b_k < a-1$.

The case where $a > 2$ is a little more complicated. In order to analyse this problem, we recall from \S\ref{sec:Formal2-BL} that if $0 < \beta_j < a-1$, then $q_j = 0$ implies either that $\chi_j(i)$ has only the trivial solution $\chi_j(i) \equiv 0$, or that $\beta_j$ can be expressed as the sum of other $\beta_{J}$ values (allowing possible repetitions), where all of these other $\beta_J$ are associated with nontrivial solutions for $\chi_J(i)$. Using this result, we show that the only possible $b_k$ with $0 < b_k < a-1$ are the integers. %[STOP HERE BECAUSE I NEED TO GO, SORRY. THIS RESULT STILL NEEDS TO BE PROVED IN SEC \ref{sec:Formal2-BL}. It's slightly nontrivial; it starts off easily. By expanding the binomial in the general equation for $\chi$ and collecting terms by order, we see that any $\beta_j$ that can't be extressed as the sum of other $\beta_J$ values must be the solution of a homogeneous nonlinear infinite system. All that remains is to show that this linear system has only one zero eigenvalue (it looks positive semi-definite, but I don't know if that helps), so taht the only way to satisfy a `zero boundary condition' is to have the zero solution.]

For the purposes of contradiction, assume that $a > 2$ and that there exists some smallest noninteger $\theta$ in the range $0 < \theta < a-1$ so that $b_k = \theta$ is associated with a nontrivial solution for $\xi_k$. As noted above, this implies that %the matching table contains terms of the form $-\tfrac{p_k}{2} n^{-b_k}$ and $p_k R n^{-b_k-1}$, and hence 
both $\theta-1$ and $\theta$ must be values taken by the exponents $\beta_j$. Now, consider the function $\chi_j(i)$ associated with $\beta_j = \theta -1$. Since this has a nontrivial solution, it follows that either $q_j \neq 0$ or that $\theta - 1$ can be expressed as the sum of $\beta_J$ values associated with nontrivial solutions for $\chi_J$. However, $q_j \neq 0$ would imply that the matching table contains a term of the form $q_j R n^{-\theta}$, which must correspond to $\theta-1$ being a value taken by one of the $b_k$; this would be a contradiction with the assumption that $\theta$ is the smallest noninteger value of $b_k$. Similarly, if $\theta - 1$ can be expressed as a sum of $\beta_J$ values, at least one of these must be noninteger, which would also lead to a contradictory noninteger value of $b_k$ less than $\theta$. %[I'm sure you guys can write this in more mathematical language, but that's essentially what the argument is.] %P: I think its fine :)

For $a \neq 2$, we therefore find that the solution in the bulk region takes the form
\begin{equation*}
%\label{eq:Formal2-xiSoln}
  \xi(s) = s + \sum_{k=1}^{\ceil{a-2}} p_k ( s - \tfrac12) n^{-k} + n^{-(a-1)} \tilde{\xi}(s) +  o\big(n^{-(a-1)}\big),
\end{equation*}
where $\tilde{\xi}(s)$ is given in \eqref{eq:Formal2-xiSolnLogs}. Using either intermediate matching (as described above) or Van Dyke's matching criterion, we can use this expression to find the asymptotic behaviour of the functions $\chi_j(i)$ as $i \to \infty$. It follows that the solution in the boundary layer region takes the form
\begin{equation}
\label{eq:Formal2-chiSoln}
  \chi(i) = \chi_0(i) + \sum_{j=1}^{\ceil{a-2}} \chi_j(i) n^{-j} + n^{-(a-1)} \tilde{\chi}(i) +  o\big(n^{-(a-1)}\big).
\end{equation}

Moreover, we can use the matching table to define the asymptotic behaviour of $\chi_j(i)$ as $i \to \infty$ in terms of the constants $p_k$ and $\tilde{p}$ from the solution in the bulk region. Specifically, we find that the asymptotic behaviour of $\chi_j(i)$ for large $i$ and $a > 2$ is given by
\begin{equation}
 \label{eq:chij-LargeiAsymp-a>2}
 \chi_j(i) = 
 \begin{cases}
    i - \tfrac{p_1}{2} + \frac{1 }{\zeta(a) (-a+2)_4}i^{-(a-2)} + o(i^{-(a-2)}), & j = 0; \\
    p_{j} i - \tfrac{p_{j+1}}{2} + O\big(i^{-(a-2-j)}\big), & 0 < j < a-2; \\
    p_{j} i - \left(\frac{1}{\zeta(a) (-a+2)_4} + \frac{\tilde{p}}{2} \right) + o(1), & j = a-2; \\
    p_{j} i + O\big(i^{-(a-2-j)}\big), & a-2 < j < a-1,
 \end{cases}
\end{equation}
The behaviour of $\tilde{\chi}(i)$ for large $i$ is given by
\begin{equation*}
 \tilde{\chi}(i) = {\textstyle \left(\frac{1}{\zeta(a) (-a+1)_3} + \tilde{p}\right)} i + o(i).
\end{equation*}

These expressions enable us to define the constants $p_j$ based on the solutions obtained for $\chi_j(i)$.
%\PATRICK{[please cite the formula from which we can (theoretically) compute $\chi_j(i)$. $*$reads further$*$. Ah, they appear in \S 3.5. Well, I think it is better to put \S 3.5 before \S 3.4. This comes at the price of not having the boundary conditions explicit (like in \S 3.3), but it makes it clear how to solve the iterative system \eqref{eq:pjExpression}. Also, you may need the equation for $\chi_j$ to motivate that $q_j$ and $\hat q_j$ cannot be $0$, which is still a missing piece in your reasoning for having $b_{k-1} = b_k - 1$]} \CAMERON{[Hmmmmm. I've been playing with this, and I'm not sure which way works better. If you expand out the sum in \S3.5 without first knowing that the first correction is $O(n^{-1})$, you have to deal with the possibility that you encounter divergent power series at low orders (by which I mean before $\ord(n^{-(a-1)})$), and I wanted to avoid talking about this possibility. But you're right in that it is important that $q_j = 0$ leads to a trivial solution for $\chi_j$ and this is a missing part of the argument above.]} \PATRICK{[OK if you want to keep the order like this. But indeed, mention this missing argument, which we then motivate later in \S 3.5]} 
For $1 \leq j < a-1$, we see that
\begin{equation*}
% \label{eq:pjExpression}
 p_j = 2\lim_{i \to \infty} \left[p_{j-1} i - \chi_{j-1}(i) \right],
\end{equation*}
where we take $p_0 = 1$. If $a$ is an integer, we also find that
\begin{equation*}
 \tilde{p} =  2\lim_{i \rightarrow \infty} \left[p_{a-1} i - \chi_{a-1}(i) \right] - \frac{2}{\zeta(a) (-a+2)_4}.
\end{equation*}
If $a$ is not an integer, we require that $\tilde{p} =  -\frac{2}{\zeta(a) (-a+2)_4}$. If this were not the case, \eqref{eq:GenMatching-MoreComplicatedOnLHS} would yield an $\ord[R^0 n^{-(a-1)}]$ term on the left hand side of \eqref{eq:IntermediateAsymptotics-Prelim} that could not be balanced by any equivalent term on the right hand side of \eqref{eq:IntermediateAsymptotics-Prelim} without contradicting the result that $\chi$ has an expansion of the form given in \eqref{eq:Formal2-chiSoln}.

In the case where $a < 2$, we recall that $\xi(s)$ must take the form
\begin{equation*}
 \xi(s) = s + n^{-(a-1)} \left[ \frac{s^{-(a-2)} - (1-s)^{-(a-2)}}{\zeta(a) (-a+2)_4} + (s - \tfrac12)\tilde{p} \right] +  o\big(n^{-(a-1)}\big).
\end{equation*} 
By the same argument as above for noninteger $a$ when $a > 2$, we find that $\tilde{p} =  -\frac{2}{\zeta(a) (-a+2)_4}$ also when $a < 2$. Hence, we find from \eqref{eq:Formal2-chiSoln} that $\chi(i) = \chi_0(i) + n^{-(a-1)}\tilde{\chi}(i)$ and that the asymptotic behaviours of these functions are given by
\begin{equation}
\label{eq:chi0-expansion-a<2}
 \chi_0(i) \sim i + {\textstyle \frac{1 }{\zeta(a) (-a+2)_4}i^{-(a-2)}} + o(i^{-(a-2)}),
\end{equation}
and
\begin{equation*}
 \tilde{\chi}(i) = {\textstyle  \frac{-a}{\zeta(a)(-a+2)_4}} i + o(i).
\end{equation*}

In the case where $a = 2$, the logarithm in \eqref{eq:Formal2-xiSolnLogs} requires careful handling, and we find that some additional terms that are logarithmically large in $n$ need to be introduced. This makes it more difficult to construct a matching table, but the arguments described above can still be used with some modifications. %With some modifications, the arguments from above can still be used to determine the sizes of the terms in the expansions of $\xi(s)$ and $\chi(i)$. 
Ultimately, we find that we can account for all logarithmic terms using the expansions
\begin{equation*}
 \xi(s) = s + n^{-1} (\log n) (s - \tfrac12)\tilde{p}^* + n^{-1} \left[ \frac{\log (1-s) - \log(s) }{\pi^2} + (s - \tfrac12)\tilde{p}\right] + o(n^{-1})
\end{equation*}
and
\begin{equation}
\label{eq:Formal:a2:logterm}
 \chi(i) = \chi_0(i) + n^{-1} (\log n) \tilde{\chi}^*(i) + n^{-1} \tilde{\chi}(i) + o(n^{-1}).
\end{equation}
Matching between the bulk and the boundary layer can then be achieved by setting $\tilde{p}^* = \frac{2}{\pi^2}$, and taking
\[
 \tilde{p} = 2 \lim_{i \rightarrow \infty}  \left[ i - \chi_{0}(i) \right].
\]
While we have concentrated on obtaining terms up to $\ord[n^{-(a-1)}]$ in our expansions of both $\xi$ and $\chi$, it may be noted that further high order terms can also be obtained using the techniques of matched asymptotic expansions. However, obtaining these high-order terms becomes much more algebraically laborious. In \S\ref{sec:Formal2-bulk}, we commented that finding higher-order corrections requires us to expand $S_0$ in \eqref{eq:ForceBalance-PowerLaw-ApplyAnsatzes} and exploiting the properties of $\chi(i)$. In the same way, obtaining higher order corrections in the boundary layer would require us to expand $S_5$ and $S_6$ in \eqref{eq:Formal2-BLForceBalance1} and exploit the properties of $\xi(s)$. Additionally, we find that the high order solutions for $\xi_k$ are no longer as simple as the expressions obtained when $b_k < a-1$, which causes the matching table to become much more complicated.

As described in this section, formal asymptotic methods can be used to elucidate the structure of the original discrete problem and determine the appropriate scalings for higher-order asymptotic analsyis. By the principles of matched asymptotic expansions, we use information about the behaviour of the bulk solution to construct the boundary layer solution and vice versa; this is where formal asymptotic analysis becomes particularly useful. For example, our higher-order analysis of $\xi$ gives us detailed information about the decay properties of $\chi_0$. Indeed, combining \eqref{eq:chij-LargeiAsymp-a>2}, \eqref{eq:chi0-expansion-a<2}, and \eqref{eq:Formal:a2:logterm}, we obtain the decay properties of $\chi_0$ as given by \eqref{for:chi0:result}, from which \eqref{intro:eqn:decay} follows. 
%\begin{equation*}
% \chi_0(i) = 
% \begin{cases}
%  i + \frac{i^{-(a-2)}}{\zeta(a) (a-2) (a^3 - a)} + o\left( i^{-(a-2)} \right), & a < 2; \\
%  i - \frac{\log i}{\pi^2} - \frac{\tilde{p}}{2} + o(1), & a = 2; \\
%  i - \frac{p_1}{2} + \frac{i^{-(a-2)}}{\zeta(a) (a-2) (a^3 - a)} + o\left( i^{-(a-2)} \right), & a > 2.
% \end{cases}
%\end{equation*}
%This then implies that
%\begin{equation*}
% \chi_0(i) - \chi_0(i-1) = 1 - \frac{i^{-(a-1)}}{\zeta(a)(a^3 - a)} + o\big(i^{-(a-1)}\big);
%\end{equation*}
%this result would be very difficult to obtain without using matched asymptotic expansions.

In the general case where $V$ satisfies \eqref{eq:GenPotential-Derivs}, the coefficients of various terms change but the structure of the asymptotic matching remains identical up to $\ord(n^{-(a-1)})$. Hence, we also find that the solution for $\tilde{\xi}$ given in \eqref{eq:Formal2-xiSolnLogs-General} can be used to obtain information about the decay behaviour of $\chi_0$ for a general $V$. From this, we find that we can generalise \eqref{intro:eqn:decay} to obtain \eqref{intro:eqn:decay2}.

\section{Asymptotic development of the ground state energy}
\label{sec:variational}

This section is devoted to the statement and proof of Theorem~\ref{thm:Gamma:conv}, which demonstrates $\Gamma$--convergence of the functional $E_n^1$ defined in \eqref{intro:eq:ene:expansion}: for an
introduction to the method of $\Gamma$--convergence, we refer
the reader to \cite{Braides02} or \cite{DalMaso93}. As stated in
\S\ref{sec:Results}, to establish these results we make the stronger decay assumption {\bf(Dec+)} in addition to the basic assumptions detailed in \S\ref{sec:Setting} throughout this section.

We begin in \S\ref{ssec:BL:ene} by reformulating the minimisation problem for \eqref{intro:eq:ene:expansion} in terms of the variable $\eps$ as introduced in \eqref{eq:defn:eps}. In \S\ref{ssec:key:ests}, we then establish key estimates used in the proof of our Theorem~\ref{thm:Gamma:conv}, which is then stated and proved in \S\ref{ssec:Gconv}. In \S\ref{ssec:minz:plus:EL:eqn} we apply Theorem \ref{thm:Gamma:conv} to prove that solutions to the force balance \eqref{eq:equilibrium} converge to solutions to the boundary layer equation \eqref{intro:eqn:BL:ell2}. In \S\ref{ssec:a:smaller:3over2} we then argue that {\bf(Dec+)} is a natural condition for the methods we use here, and that additional ideas are required to obtain a result assuming only {\bf(Dec)}, or a yet weaker decay hypothesis. Some technical computations required for the proof of Theorem~\ref{thm:Gamma:conv} are left until \S\ref{app:sigman:computs}.

\subsection{Reformulation}
\label{ssec:BL:ene}
Let $\bar{x}(i):=\frac{i}{n}$ be the equispaced configuration. 
We reinterpret \eqref{eq:defn:eps} as
\begin{equation*} %\label{for:defn:eps:alt}
  \epsilon(i) := n \bighaa{ \big[x(i) - x(i-1)\big] - \big[\bar x(i) - \bar x(i-1)\big]}
  \qquad\text{ for }i=1,\ldots,n.
\end{equation*}
Figure \ref{fig:Tom:coos} illustrates the definition of 
$\epsilon(i)$ as the difference between the blown--up
perturbations of the positions $x(i)$ relative to the reference 
equispaced configuration $\bar{x}(i)$ for $i = 0, \ldots, n$. We interpret $\epsilon(i)$ as a strain variable, since it
expresses the local change in distance between particles away from the equispaced configuration. Since $x(0)=0$, the inverse transformation is given by
$x(i) = \frac1n [i + \sum_{j=1}^i \epsilon(i) ]$,
and we obtain $\sum_{i=1}^n \epsilon(i) = n\,x(n) - n = 0$.

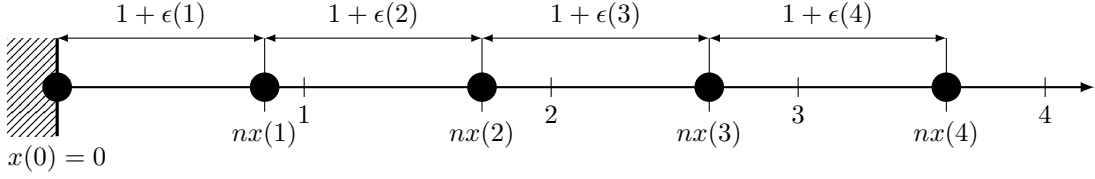
\begin{figure}[t]
\centering
\begin{tikzpicture}[scale=0.65, >= latex]
	\draw (0,-1) node [below] {$x(0) = 0$}; 
	\draw[very thick] (0,-1) -- (0,1);
	\draw[->, thick] (0,0) -- (21,0); % node [right] {$y$};
	\fill[pattern = north east lines] (-1, -1) rectangle (0, 1);	
	\draw (5, 0.2) -- (5, -0.2) node [below] {$1$};
	\draw (10, 0.2) -- (10, -0.2) node [below] {$2$};
	\draw (15, 0.2) -- (15, -0.2) node [below] {$3$};
	\draw (20, 0.2) -- (20, -0.2) node [below] {$4$};
    \foreach \x in {0,4.2,8.6,13.2,18}
      {
      \fill (\x, 0) circle (0.3);
      }
    %\draw[->] (5, 0.8) -- (4.2, 0.8) node[midway, above] {$u_1$};
    \draw[<->] (0,1) -- (4.2, 1) node[midway, above] {$1 + \epsilon(1)$};
    \draw (4.2, -0.5) node [below] {$n x(1)$} -- (4.2, 1);
    %\draw[->] (10, 0.8) -- (8.6, 0.8) node[midway, above] {$u_2$};
    \draw[<->] (4.2,1) -- (8.6, 1) node[midway, above] {$1 + \epsilon(2)$};
    \draw (8.6, -0.5) node [below] {$n x(2)$} -- (8.6, 1);
    %\draw[->] (15, 0.8) -- (13.2, 0.8) node[midway, above] {$u_3$};
    \draw[<->] (8.6,1) -- (13.2, 1) node[midway, above] {$1 + \epsilon(3)$};
    \draw (13.2, -0.5) node [below] {$n x(3)$} -- (13.2, 1);
    %\draw[->] (20, 0.8) -- (18, 0.8) node[midway, above] {$u_4$};
    \draw[<->] (13.2,1) -- (18, 1) node[midway, above] {$1 + \epsilon(4)$};
    \draw (18, -0.5) node [below] {$n x(4)$} -- (18, 1);
\end{tikzpicture}
\caption{Zoom-in of the particle system from Figure \ref{fig:PS:bdd} at the boundary layer. Both choices of variables given by $n x(i)$ and $\epsilon(i)$ are illustrated.}\label{fig:Tom:coos}
\end{figure}

%\begin{equation}\label{for:energy_diff}
%  E_n(x) - E_n(\bar{x}_n)
%  = \sum_{k=1}^n \sum_{j=0}^{n-k} V\big(n[x(j+k)-x(j)]\big)
%  - \sum_{k=1}^n\sum_{j=0}^{n-k} V(k).
%\end{equation}

Expressing the energy difference $E_n^1$ in \eqref{intro:eq:ene:expansion} in $\epsilon$, we obtain
\begin{gather}\label{for:defn:En1}
  E^1_n(\epsilon):=\sum_{k=1}^n \sum_{j=0}^{n-k} 
  \bigg[V\bigg(k+\sum_{l=j+1}^{j+k} \epsilon(l)\bigg)-V(k)\bigg],
  \quad\Dom E_n^1 
  = \bigaccv{ \epsilon \in [-1, n]^n }{ \epsilon \cdot \mathbf 1 = 0 }.
\end{gather}
We note that the double sum over $V(k)$ equals $E_n( \bar x )$. We make three basic observations:
\begin{enumerate}
  \item Since the change of variable given above is a bijection
    from $\Dn$ to $\Dom E_n^1$ and $E_n$ has a unique minimiser in the
    interior of $\Dn$, it follows that $E_n^1$ has a unique 
    minimiser in the interior of $\Dom E_n^1$.
  \item Viewing $\epsilon$ as the perturbation to the 
    distances between particles away from unit spacing, we expect 
    $\epsilon(i)\approx 0$ for $i\approx\frac{n}{2}$, which is 
    equivalent to the fact that far from
    the boundary, the distances between particles are close to 
    $1$.
  \item By the symmetry in the geometry of the double pile-up, the 
    minimiser of $E_n^1$ has \emph{reversal symmetry}, i.e.~$\epsilon(i) = \epsilon(n+1-i)$. The reversal symmetry of the minimiser is easily proved from the strict convexity of $E_n$. We introduce the following notation for `reversing' a sequence:
      \begin{equation} \label{for:defn:reversal:epsilon}
      \text{for } \epsilon \in \Dom E_n^1, \text{ let } \cev \epsilon(i) :=
        \epsilon(n+1-i).
    \end{equation}
    It is easy to check that $\cev \epsilon \in \Dom E_n^1$, and that $E_n^1 (\cev \epsilon) = E_n^1 (\epsilon) \geq E_n^1 ( (\epsilon + \cev \epsilon)/2 )$.
\end{enumerate}

%\TOM{[I'm not sure we need to state the following as a separate lemma, as it's not explicitly referenced elsewhere in the document.]}
%
%\begin{lem}[Reversal symmetry] \label{lem:rev:symy}
%Let $\epsilon \in \Dom E_n^1$. Then $\epsilon' := \frac12 (\epsilon + \cev \epsilon) \in \Dom E_n^1$ has reversal symmetry, and satisfies $E^1_n(\epsilon')\leq E^1_n(\epsilon)$.
%\end{lem}
%
%\begin{proof}
%  It is easy to see that $\epsilon' \in \Dom E_n^1$ has reversal symmetry. By convexity of $E_n^1$, we obtain 
%\begin{equation*}
%  E_n^1(\epsilon) 
%  = \frac12 \bighaa{ E_n^1(\epsilon) + E_n^1(\cev \epsilon) }
%  \geq E_n^1 (\epsilon').\qedhere
%\end{equation*}
%\end{proof}

\subsection{Structure of $E_n^1$ and key estimates}
\label{ssec:key:ests}

In order to prove a $\Gamma$--convergence result, we extend
the definition of $E^1_n$ so that these functionals are
defined over the same topological space. Here, the right space
turns out to be $\ell^2(\N{})$.

To do so, we define the embedding $\iota_n:\Dom 
E^1_n\to\ell^2(\N{})$, where
\begin{equation*}
  \iota_n\epsilon(i) := \begin{cases}
    \epsilon(i) 
    &\text{if } i = 1,\ldots,n, \\
    0
    &\text{otherwise}.
  \end{cases}
\end{equation*}
This permits us to extend $E^1_n$ over $\ell^2(\N{})$ in the 
following manner:
\begin{equation*}
  E^1_n(\epsilon) = \begin{cases}
    E^1_n(\epsilon) &\text{if }\epsilon \in\iota_n(\Dom E^1_n),\\
    +\infty &\text{otherwise.}
  \end{cases}
\end{equation*}

To expose the locally quadratic structure of $E^1_n(\epsilon)$ as
defined in \eqref{for:defn:En1}, we rewrite it by subtracting and adding a term which is linear in $\epsilon$. For any $\epsilon\in\ell^2$ with finite energy, we obtain
\begin{multline} \label{eq:energy:splitting}
 E^1_n(\epsilon)
 = \sum_{k=1}^n
   \sum_{j=0}^{n-k}\Bigg(V\bigg(k+ 
   \sum_{l=j+1}^{k+j}\epsilon(l)\bigg)-V(k)
   -V'(k) \sum_{l=j+1}^{k+j} \epsilon(l)\Bigg) 
   + \sum_{k=1}^n V'(k)\sum_{j=0}^{n-k}
   \sum_{l=j+1}^{k+j}\epsilon(l) \\
 = \underbrace{\sum_{k=1}^n
   \sum_{j=0}^{n-k} \phi_k \bigg(
   \sum_{l=j+1}^{k+j} \epsilon(l) \bigg)}_{=: \Qn (\epsilon)}
   + \bighaa{ \sigma^n, \epsilon + \cev \epsilon }_{\ell^2 (\N{})},
\end{multline}
where 
\begin{gather} \label{for:defn:phik}
  \phi_k (y) 
  := V(k+y)-V(k)-V'(k)y, \\\label{for:defn:sigman}
  \sigma^n(i)
  := \left\{ \begin{aligned}
  		&\sum_{k = i+1}^{n-i} \bigbhaa{ (k-i) \wedge (n-i+1-k) } \bigabs{ V' (k) },
  		&&\text{if } i \leq \lfloor n/2 \rfloor. \\
        &0, 
        &&\text{otherwise}.
  \end{aligned} \right.
\end{gather}
The second equality in \eqref{eq:energy:splitting} follows from changing the order of summations and using the fact that
$(\epsilon, \mathbf 1)_{\ell^2(\N{})} = 0$; the details of this
computation are provided in Appendix \ref{app:sigman:computs}. The function $\phi_k$ is the error of the first order Taylor expansion of $V(x)$ around $x = k$, expressed in the shifted variable $y := x - k$. We interpret $\sigma^n$ as a stress which arises due to 
the constraint that the particles are confined to lie in a finite
interval.

Lemma \ref{lem:lb_single_term} states precisely what we mean by $\Qn$ being `locally quadratic'; it provides a quadratic lower and upper bound for $\phi_k$. Both bounds are essential in the proof of Theorem  \ref{thm:Gamma:conv}. Figure \ref{fig:phik} illustrates $\phi_k$ together with the lower and upper bound. The proof of Lemma \ref{lem:lb_single_term} is a direct consequence of {\bf(Cvx)}, i.e.~the strict convexity of $V$.

\begin{figure}[ht]
\centering
\begin{tikzpicture}[scale=1.8, >= latex]
\draw[dotted] (-2,3.2) -- (-2,0) node[below] {$-k$};           
\draw[->] (-3,0) -- (4,0) node[right] {$y$};
\draw[->] (0,0) node[below] {$0$} -- (0,3.2);
\draw[thick, domain=-1.6:4, smooth] plot (\x,{2/(\x + 2) - 1 + \x/2});
\draw[domain=-1:1, smooth] plot (\x,{3.2*\x*\x});
\draw[domain=1:4, smooth] plot (\x,{4*(\x - 0.5)/27});
\draw[domain=-2:1, smooth] plot (\x,{2*\x*\x / 27});
\draw[dotted] (4, 1.25) -- (1.5, 0);
\draw (3, 0.75) -- (3.5, 0.75) -- (3.5, 1) node[midway, anchor = west] {$V'(k)$};
\draw (4, 4/3 ) node[anchor = west] {$\phi_k$};
\draw (4, 14/27 ) node[anchor = west] {$\Phi_k$};
\draw (1,3.2) node[anchor = north west] {$\dfrac{C_\delta}{k^{a+2}} y^2$};
\end{tikzpicture} \\
\caption{The function $\phi_k$ as defined in Lemma \ref{lem:lb_single_term}.}
\label{fig:phik}
\end{figure}
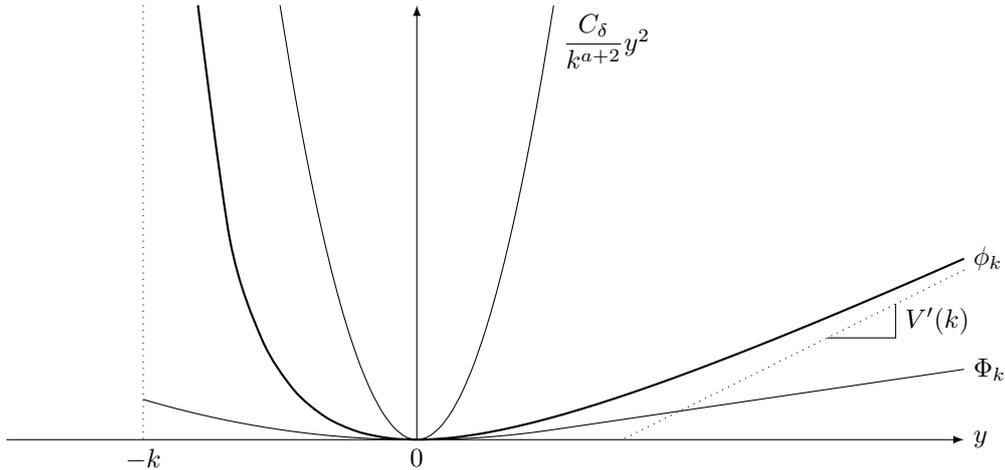

\begin{lem}[Lower bound on $V$]
\label{lem:lb_single_term}
For any $k\in[1,+\infty)$, it holds for all
$y\in(-k,+\infty)$ that
\begin{align*}
  \phi_k (y) 
  &\geq \Phi_{k}(y) := \begin{cases}
    \textstyle \frac12\lambda(k+1)\,y^2 
    & y\in(-k,1],\\
    \lambda(k+1) \, \bighaa{ y - \frac12 }  
    & y\in[1,+\infty).
  \end{cases}
\end{align*}
Moreover, if $y \geq k (\delta - 1)$ for some $\delta > 0$, then there exists a $C_\delta > 0$ such that $\phi_k (y) \leq C_\delta k^{-a-2} y^2$.
\end{lem}

\begin{proof}
The proof relies on the following observation. For any $f, g \in C^2 (\R{})$ satisfying $f(0) = g(0)$, $f' (0) = g' (0)$, and $f'' \geq g''$ on some interval $I \ni 0$, then $f \geq g$ on $I$. This is easily proven from the fact that $f - g$ is convex with $0 = (f - g)(0) = (f - g)'(0)$.

Since $V(x)$ is $\lambda(x)$-convex in the sense of {\bf(Cvx)}, it
follows that $\phi_k(y)$ is $\lambda(y + k)$-convex. Since $\lambda$ is decreasing, the lower bound of $\phi_k$ in Lemma \ref{lem:lb_single_term} follows.

Since $V \in C^2 (0, \infty)$ and $V''(x) \lesssim 1/x^{a + 2}$, it holds that 
$$
\sup_{y > k (\delta - 1)} \phi_k'' (y) 
= \sup_{x > k \delta} V'' (x) 
\leq \frac{2 C}{(\delta k)^{a+2}} 
$$ 
is finite for any fixed $\delta > 0$. The upper bound for $\phi_k$ follows.
\end{proof}

The linear term in \eqref{eq:energy:splitting} is fully characterized by $\sigma^n$. Lemma~\ref{lem:sigman} states its key properties. Its proof relies on the decay property $|V'(x)| \lesssim x^{-a-1}$ in \eqref{for:ests:V:Vprime} with $a > 3/2$;
in fact, this is the key point at which the assumption
{\bf($a$-Dec)} is necessary for our continuing analysis.

\begin{lem}[Properties of $\sigma^n$]
\label{lem:sigman}
$\sigma^n \in \ell^2 (\N{})$ as defined in \eqref{for:defn:sigman} satisfies
\begin{equation} \label{for:defn:sigma:infty}
  0 
  \leq \sigma^n(i)
  \leq \sigma^\infty(i)
  := \sum_{k = i+1}^\infty (k-i) \bigabs{ V' (k) },
  \quad \text{for all } i \geq 1.
\end{equation}
Moreover, $\sigma^n \to \sigma^\infty$ in $\ell^2 (\N{})$ as $n \to \infty$.
\end{lem}

\begin{proof}
\eqref{for:defn:sigma:infty} follows from the definitions viz.
\begin{equation} \label{for:lb:T4}
  \sigma^n(i)
  \leq \sum_{k = i+1}^{n-i} (k-i) \bigabs{ V' (k) }
  \leq \sum_{k = i+1}^\infty (k-i) \bigabs{ V' (k) }
  = \sigma^\infty(i).
\end{equation}
For the convergence in $\ell^2 (\N{})$, we set $R_n := \lfloor n/2 \rfloor$ and note that the summands in \eqref{for:defn:sigman} and \eqref{for:defn:sigma:infty} are equal for $k = i+1, \ldots, R_n$. Since $a > 3/2$ and the decay property $|V'(x)| \lesssim x^{-a-1}$ in \eqref{for:ests:V:Vprime}, we obtain the convergence from
\begin{align*} 
  \big\| \sigma^\infty - \sigma^n \big\|_{\ell^2(\N{})}^2
  &\leq \sum_{i=1}^{R_n - 1} \bigg(\sum_{k = R_n+1}^\infty (k-i) \bigabs{ V'(k) } \bigg)^2 
     + \sum_{i = R_n}^\infty \bigg(\sum_{k = i + 1}^\infty (k-i) \bigabs{ V' (k) } \bigg)^2\\ 
  &\lesssim \sum_{i=1}^{R_n-1}\bigg(\int_{R_n}^\infty \frac{x-i}{x^{a+1}}
   \mathrm{d}x\bigg)^2
     + \sum_{i = R_n}^\infty \bigg(\sum_{k = i + 1}^\infty k \frac1{k^{a+1}} \bigg)^2 \\ %\label{for:sigman:s:ell2}
  &\lesssim \sum_{i-1}^{R_n-1} \frac{(R_n - i)^2}{R_n^{2a}} 
     + \sum_{i = R_n}^\infty \frac1{i^{2(a-1)}}
  \lesssim \frac1{n^{2a - 3}} \xrightarrow{n \to \infty} 0.
  \qedhere
\end{align*}
\end{proof}

\subsection{Main result: $\Gamma$--convergence}
\label{ssec:Gconv}

We prove $\Gamma$--convergence in the weak topology of $\ell^2 (\N{})$. To accommodate for a splitting of $\epsilon^n \in \Dom E_n^1$ to account for the boundary layer at the left and right barrier separately, we introduce the following notation:
\begin{equation*}
  \epsilon^{n, 1/2}(i)
  := \begin{cases}
    \epsilon^n(i)
    &\text{if } i = 1,\ldots, \lceil n/2 \rceil, \\
    0
    &\text{otherwise,}
  \end{cases}
  \quad \text{and} \quad
  \cev \epsilon^{n, 1/2}(i)
  := \begin{cases}
    \cev \epsilon^n(i)
    &\text{if } i = 1,\ldots, \lceil n/2 \rceil, \\
    0
    &\text{otherwise,}
  \end{cases}
\end{equation*}
We will also write
\begin{align*}
    \epsilon^n \xrightarrow 2 (\epsilon, \cev \epsilon)
    \, \Leftrightarrow \,
    \lracc{ \begin{array}{c}
    \epsilon^{n, 1/2} \to \epsilon \\
    \cev \epsilon^{n, 1/2} \to \cev \epsilon
    \end{array} }
    \text{ in } \ell^2 (\N{}),
    \quad\text{and}\quad
    \epsilon^n \xweakto 2 (\epsilon, \cev \epsilon)
    \, \Leftrightarrow \,
    \lracc{ \begin{array}{c}
    \epsilon^{n, 1/2} \weakto \epsilon \\
    \cev \epsilon^{n, 1/2} \weakto \cev \epsilon
    \end{array} }
    \text{ in } \ell^2 (\N{}).
\end{align*}
We remark that the reversal of a sequence \eqref{for:defn:reversal:epsilon} is only well-defined for sequences that are equivalent to finite dimensional vectors (i.e.~sequences which have finite support). Therefore, in the definition above, there need not be any relation between $\epsilon$ and $\cev \epsilon$, while $\cev \epsilon^n(i) = \epsilon^n(n+1-i)$.

Theorem \ref{thm:Gamma:conv} states that the $\Gamma$--limit of $E_n^1$ is given by 
\begin{gather} \notag %\label{for:defn:Einfty:Dom}
  \Dom E^1_\infty 
  := \bigaccv{ (\epsilon, \cev \epsilon) \in \ell^2 (\N{}) \times \ell^2 (\N{}) }{ \epsilon(i), \cev \epsilon(i) \geq -1 \: \: \forall \, i \geq 1 }, \\\label{for:defn:Einfty}
  E^1_\infty(\epsilon, \cev \epsilon) 
  := \underbrace{ \sum_{k=1}^\infty\sum_{j=0}^\infty \phi_k \bigg(\sum_{l = j+1}^{k+j} \epsilon(l)\bigg) }_{ =: Q_\infty } 
  + \underbrace{ \sum_{k=1}^\infty\sum_{j=0}^\infty \phi_k \bigg(\sum_{l=j+1}^{k+j} \cev \epsilon(l)\bigg)  }_{ =: \cev Q_\infty } 
  + \bighaa{ \sigma^\infty, \epsilon + \cev \epsilon }_{\ell^2(\N{})}, 
\end{gather}
where $\phi_k$ as in \eqref{for:defn:phik} and $\sigma^\infty$ as in \eqref{for:defn:sigma:infty}.

\begin{thm}[$\Gamma$--convergence of $E_n^1$] \label{thm:Gamma:conv}
If $E^1_n(\epsilon^n)$ is uniformly bounded in $n$ for some sequence $(\epsilon^n) \subset \ell^2 (\N{})$, then $\|\epsilon^n\|_{\ell^2}$ is uniformly bounded. Moreover, for any $(\epsilon, \cev \epsilon) \in \Dom E_\infty^1$, it holds that
\begin{subequations}
\label{for:thm:Gamma:conv}
\begin{alignat}2
\label{for:thm:Gamma:conv:liminf}
&\text{for all } \epsilon^n \xweakto 2 (\epsilon, \cev \epsilon) \text{ with }  \epsilon^n \in \Dom E_n^1 , 
&\liminf_{n\to\infty} E^1_n(\epsilon^n) &\geq E^1_\infty (\epsilon, \cev \epsilon), \\
&\text{there exists }\epsilon^n \xto 2 (\epsilon, \cev \epsilon) \text{ with }  \epsilon^n \in \Dom E_n^1 \text{ such that}\quad 
& \limsup_{n\to\infty} E^1_n(\epsilon^n) &\leq E^1_\infty (\epsilon, \cev \epsilon).
\label{for:thm:Gamma:conv:limsup}
\end{alignat}
\end{subequations}
\end{thm}

%The proof of this result relies upon Lemma~\ref{lem:lb_single_term} and a series of careful error estimates.

%\PATRICK{REMARK: the proof below also holds in the case when we consider $V (n \ell_n [x_i - x_j])$ with $\ell_n \to \ell$. First, we define $V_\ell (r) = V(\ell r)$ s.t. $V (n \ell_n [x_i - x_j]) = V_\ell (n \alpha_n [x_i - x_j])$ with $\alpha_n \to 1$. Then we define $\tilde \epsilon(i) = \epsilon(i) / \alpha_n$ with $\epsilon$ as in \eqref{eq:defn:eps}. Clearly $\sum
%_{i=1}^n \tilde \epsilon(i) = 0$, and the lower bound changes to $\tilde \epsilon(i) \geq -1/\alpha_n$. Since we never use the sharpness of this lower bound (the energy is infinite there anyway), the proof holds just as well for $\tilde \epsilon$. }

\begin{proof}[Proof of compactness in Theorem \ref{thm:Gamma:conv}]
First we obtain a sufficient lower bound on $\Qn$ in \eqref{eq:energy:splitting}. we use Lemma~\ref{lem:lb_single_term} on the nearest neighbour interactions to estimate
\begin{equation} \label{for:lb:T1}
  \Qn (\epsilon^n)
  = \sum_{k=1}^n
   \sum_{j=0}^{n-k} \phi_k \bigg(
   \sum_{l=j+1}^{k+j} \epsilon(l) \bigg)
  \geq \sum_{j=1}^{n}\Phi_{1}\big(\epsilon(j)\big).
\end{equation}
To obtain a sufficient lower bound of $E_n^1 (\epsilon^n)$ from \eqref{for:lb:T1}, \eqref{for:lb:T4} and $\sigma^\infty \in \ell^2 (\N{})$, we split $\epsilon^n$ into a positive and negative part viz.
\begin{equation*}
  \epsilon^n_+(i) := \epsilon^n(i) \vee 0 \geq 0,
  \qquad \epsilon^n_-(i) := -\bighaa{ \epsilon^n(i) \wedge 0 } \geq 0.
\end{equation*}
Then, we use Lemma \ref{lem:lb_single_term} and $\Phi_1(0) = 0$ to estimate
\begin{align} \nonumber
 E^1_n(\epsilon^n)
 &\geq \sum_{j=1}^n \phi_1 \bighaa{ \epsilon^n(j) } + \bighaa{ \sigma^n, \epsilon^n_+ + \cev \epsilon^n_+ }_{\ell^2 (\N{})} - \bighaa{ \sigma^n, \epsilon^n_- + \cev \epsilon^n_- }_{\ell^2 (\N{})} \\\nonumber
 &\geq \sum_{j=1}^n \Bigbhaa{ \Phi_1 \bighaa{ \epsilon^n_+(j) } 
   + \Phi_1 \bighaa{ -\epsilon^n_-(j) } } - \sum_{j=1}^n \sigma^\infty(j) \bighaa{ \epsilon^n_-(j) + \cev \epsilon^n_-(j) } \\\nonumber
 &\geq \sum_{j=1}^n \Phi_1 \bighaa{ \epsilon^n_+(j) } 
   + \tfrac{\lambda(2)}4 \sum_{j=1}^n \Bigbhaa{ \epsilon^n_-(j)^2 + \bighaa{ \cev \epsilon^n_-(j) }^2 } - \sum_{j=1}^n \sigma^\infty (j) \bighaa{ \epsilon^n_-(j) + \cev \epsilon^n_-(j) } \\\nonumber
 &= \sum_{j=1}^n \Phi_1 \bighaa{ \epsilon^n_+(j) } 
   + \tfrac{\lambda(2)}4 \sum_{j=1}^n \Bigbhaa{ \bighaa{ \epsilon^n_-(j) - \tfrac{2\sigma^\infty(j)}{\lambda(2)}  }^2 + \bighaa{ \cev \epsilon^n_-(j) - \tfrac{2\sigma^\infty(j)}{\lambda(2)} }^2 } - \tfrac2{\lambda(2)} \sum_{j=1}^n  \sigma^\infty(j)^2 \notag\\
 &\geq \sum_{j=1}^n \Phi_1 \bighaa{ \epsilon^n_+(j) } 
       + c \, \bignorm{ \, \epsilon^n_- - \tfrac2{\lambda(2)} \sigma^\infty }{\ell^2(\N{})}^2
       + c \, \bignorm{ \, \cev \epsilon^n_- - \tfrac2{\lambda(2)} \sigma^\infty }{\ell^2(\N{})}^2 
       - C. \label{for:pf:thm:cpness:1}
\end{align}
Hence, since $E_n^1 (\epsilon^n) \leq C$ by hypothesis, we obtain from $\sigma^\infty \in \ell^2 (\N{})$ that $(\epsilon^n_-)$ and $(\cev \epsilon^n_-)$ are uniformly bounded in $\ell^2 (\N{})$.

It remains to show that  $(\epsilon^n_+)$ is uniformly bounded in $\ell^2 (\N{})$. We obtain this from the uniform boundedness of $\sum_{j=1}^n \Phi_1 \bighaa{ \epsilon^n_+(j) }$ by \eqref{for:pf:thm:cpness:1}. Indeed, it follows from the linear growth of $\Phi_1$ that for some positive constant $C'$ we have
\begin{equation*}
  \epsilon^n_+(i) \leq C' \qquad\text{for all }i=1,\ldots,n
  \text{ and all }n\in\N{}.
\end{equation*}
Hence, there exists a constant $c$ (which depends on $C'$) such that
\begin{equation*}
  \Phi_1(x)\geq c x^2\qquad\text{for all }x\in[0,C'],
\end{equation*}
and so
\begin{equation*}
  E^1_n(\epsilon^n)\geq c
  \bigg(\sum_{j=1}^n \bighaa{ \epsilon^n_+(j) }^2\bigg)-C
    =c\|\epsilon^n_+\|_{\ell^2}^2-C,
\end{equation*}
thus implying that $(\epsilon^n_+)$ is uniformly bounded in $\ell^2 (\N{})$.
\end{proof}

\begin{proof}[Proof of \eqref{for:thm:Gamma:conv:liminf}]
Let $\epsilon^n \xweakto 2 (\epsilon, \cev \epsilon)$ such that $E_n^1 (\epsilon^n)$ is bounded. For the second term of $E_n^1 (\epsilon^n)$ in \eqref{eq:energy:splitting}, we use the strong convergence of $\sigma^n$ (see Lemma \ref{lem:sigman}) to obtain
\begin{equation} \label{for:conv:lin:term}
  \bighaa{ \sigma^n, \epsilon^n + \cev \epsilon^n }_{\ell^2(\N{})}
  \xto{n \to \infty} \bighaa{ \sigma^\infty, \epsilon + \cev \epsilon }_{\ell^2(\N{})}.
\end{equation}

We bound $\Qn (\epsilon_n)$ in \eqref{eq:energy:splitting} from below by dropping some terms in the summation. We set $R_n := \lfloor n/2 \rfloor$, and estimate
\begin{equation*}
  \Qn (\epsilon_n) 
  = \sum_{k=1}^n \sum_{j=0}^{n-k} \phi_k \bigg(\sum_{l = j+1}^{k+j} \epsilon^n(l)\bigg)
  \geq \sum_{k=1}^{R_n} \sum_{j=0}^{R_n - k} 
  \bigg[ \phi_k \bigg(\sum_{l = j+1}^{k+j} \epsilon^n(l) \bigg) + \phi_k \bigg(\sum_{l=j+1}^{k+j} \cev \epsilon^n(l)\bigg) \bigg].
\end{equation*}
To pass to the liminf as $n \to \infty$, we use Fatou's Lemma, by which we interpret the double sum as an integral over the lattice $\Np 2$. We focus on the first term in the summand, because the second term involving $\cev \epsilon^{n}$ can be estimated analogously. For the pointwise lower bound (as $n \to \infty$ with $k$ and $j$ fixed) of the summand, we interpret $\sum_{l = j+1}^{k+j} \epsilon^{n}_{l}$ as an inner product of $\epsilon^{n}$ with an $n$-independent sequence consisting of $1$'s and $0$'s. Then the fact that $\epsilon^{n, 1/2} \weakto \epsilon$ implies 
\begin{equation*} %\label{for:conv:of:eps:ip:1vec}
  \sum_{l = j+1}^{k+j} \epsilon^{n, 1/2}(l)
  \xrightarrow{n \to \infty} \sum_{l = j+1}^{k+j}\epsilon(l)
  \qquad \text{for all } k, j \geq 1.
\end{equation*}
Since Lemma~\ref{lem:lb_single_term} implies that $\phi_k$ is positive and lower semicontinuous, it follows that
\begin{equation*}
  \liminf_{n\to\infty} \phi_k \bigg(\sum_{l = j+1}^{k+j}\epsilon^{n, 1/2}_l\bigg)
  \geq \phi_k \bigg(\sum_{l = j+1}^{k+j}\epsilon(l)\bigg)
  \qquad \text{for all } k, j \geq 1.
\end{equation*}
This shows that the hypotheses of Fatou's Lemma are satisfied, and thus
\begin{equation*} %\label{for:limf:T1}
  \liminf_{n\to\infty} \sum_{k=1}^n \sum_{j=0}^{n-k} \phi_k \bigg(\sum_{l = j+1}^{k+j} \epsilon^n(l)\bigg) 
  \geq \sum_{k=1}^\infty\sum_{j=0}^\infty 
  \bigg[ \phi_k \bigg(\sum_{l = j+1}^{k+j} \epsilon(l)\bigg) + \phi_k \bigg(\sum_{l=j+1}^{k+j} \cev \epsilon(l)\bigg) \bigg].
  \qedhere
\end{equation*}
\end{proof}

\begin{proof}[Proof of \eqref{for:thm:Gamma:conv:limsup}]
Let $(\epsilon, \cev \epsilon) \in \Dom E_\infty^1$ such that $E_\infty^1 (\epsilon, \cev \epsilon) =: C < \infty$. Then $\epsilon, \cev \epsilon \in \ell^2 (\N{}) \subset \ell^\infty (\N{})$, and 
\begin{equation*}
  E^1_\infty(\epsilon, \cev \epsilon) 
  \geq \phi_1 \bighaa{ \epsilon(i) } + \phi_1 \bighaa{ \cev \epsilon(j)}  
  - \norm{ \sigma^\infty }{\ell^2(\N{})} \norm{ \epsilon + \cev \epsilon }{\ell^2(\N{})}
  \quad \text{for all } i,j \geq 1. 
\end{equation*}
Hence, there exists a $\delta > 0$ such that
\begin{equation*}
    \epsilon, \cev \epsilon     
    \in X_\delta
    := \Bigaccv{ \epsilon\in\ell^2(\N{}) }{ \delta - 1 \leq \epsilon(i) \leq \frac{1}{\delta}
,\text{ for all }i\in\N{} }.
\end{equation*}

Next we construct a recovery sequence. As in \cite{Hudson2013}, we note
that the constraint that $\sum_{i=1}^\infty \epsilon^n(i) = 0$
need not be preserved in the limit as $n\to\infty$. We take this into account by introducing $1 \ll S_n \ll n$ as the index where we match the boundary layer with the bulk. We note that as $n\to\infty$, it holds that $S_n\to\infty$ and $S_n/n\to0$. We further set
\begin{equation} \label{for:defn:un}
  u_n := \sum_{i=1}^{S_n} \epsilon(i),
  \quad \text{and} \quad
  \cev u_n := \sum_{i=1}^{S_n} \cev \epsilon(i).
\end{equation}
%For $\epsilon \in X_\delta$, we have $|u_n| \leq S_n/\delta$, and, similarly, $|\cev u_n| \leq S_n/\delta$. 
We now define the recovery sequence
\begin{equation} \label{for:defn:rec:seq}
  \epsilon^n(i) := \begin{cases}
    {\epsilon}(i) & i\in\{1,\ldots,S_n\},\\
    -\dfrac{u_n + \cev u_n}{n-2S_n} &i\in\{S_n+1,\ldots,n-S_n\},\\
    \cev{\epsilon}(n+1-i) & i\in\{n-S_n+1,\ldots,n\}.
  \end{cases}
\end{equation}
It is easily checked that $\sum_{i=1}^n \epsilon^n(i) = 0$ and  $\epsilon^n(i) \geq -1$ for $n$ large enough, hence $\epsilon^n \in \Dom E_n^1$. 

To show that $\epsilon^n \xto 2 (\epsilon, \cev \epsilon)$, we prove that $\epsilon^{n,1/2} \to \epsilon$ in $\ell^2 (\N{})$, and conclude by an analogous argument that $\cev \epsilon^{n,1/2} \to \cev \epsilon$ in $\ell^2 (\N{})$. To this end, we estimate
\begin{equation} \label{for:rec:seq:conv}
  \norm{ \epsilon - \epsilon^{n,1/2} }{\ell^2 (\N{})}^2
  = \sum_{i=S_n+1}^{ \lceil n/2 \rceil } \hspace{-1mm}\bighaa{ \epsilon(i) + \tfrac{u_n + \cev u_n}{n-2S_n} }^2 
    + \hspace{-2mm}\sum_{i=\lceil n/2 \rceil + 1}^\infty \hspace{-2mm}\epsilon(i)^2
  \leq 2 \sum_{i=S_n+1}^{ \lceil n/2 \rceil } \bighaa{ \tfrac{u_n + \cev u_n}{n-2S_n} }^2 
    + 2 \hspace{-1mm}\sum_{i=S_n+1}^\infty \hspace{-1mm}\epsilon(i)^2.
\end{equation}
The second term in the right-hand side of \eqref{for:rec:seq:conv} converges to $0$ as $n \to \infty$ because $\epsilon \in \ell^2 (\N{})$. To show that the first term in the right-hand side is also small for large $n$, we interpret $u_n$ in \eqref{for:defn:un} as the inner product of $\epsilon$ with a sequence consisting of $1$'s and $0$'s. Applying the Cauchy-Schwartz inequality on this inner product yields
\begin{equation*}
    \sum_{i=S_n+1}^{ \lceil n/2 \rceil } \bigghaa{ \frac{u_n + \cev u_n}{n-2S_n} }^2 
    \leq \frac1{(n-2S_n)^2} \sum_{i=S_n+1}^{ \lceil n/2 \rceil } \bighaa{ 2 \sqrt{S_n} \norm\epsilon{\ell^2 (\N{})} }^2
    \lesssim \frac1{n^2} \frac n2 4 S_n \norm\epsilon{\ell^2 (\N{})}^2
    \xto{ n \to \infty } 0,
\end{equation*}
where we recall that $S_n\ll n$ as $n\to\infty$.
This completes the proof of $\epsilon^n \xto 2 (\epsilon, \cev \epsilon)$.

To establish the limsup inequality \eqref{for:thm:Gamma:conv:limsup}, we observe from the argument leading to \eqref{for:conv:lin:term} that it is enough to focus on $\Qn$ in \eqref{eq:energy:splitting}, since the convergence of
terms involving $\sigma^n$ is implied by the fact that $\epsilon^n \xto 2 (\epsilon, \cev \epsilon)$, as just shown. For convenience, we choose $S_n$ to be even. We split the summation in $\Qn$ into four parts:
\begin{multline} \label{for:ene:splitting:sumterm:limp}
  \Qn (\epsilon^n)
%  = \sum_{k=1}^n \sum_{j=0}^{n-k} \phi_k \bigg(\sum_{l = j+1}^{k+j} \epsilon^n(l)\bigg)
  = \sum_{k=1}^{S_n/2} \Bigg[ \sum_{j=0}^{S_n-k} \phi_k \bigg(\sum_{l = j+1}^{k+j} \epsilon^n(l)\bigg)
       + \sum_{j=S_n-k + 1}^{n - S_n - 1} \phi_k \bigg(\sum_{l = j+1}^{k+j} \epsilon^n(l) \bigg) \\
       + \sum_{j = n-S_n}^{n-k} \phi_k \bigg(\sum_{l = j+1}^{k+j} \epsilon^n(l) \bigg) \Bigg] 
       + \sum_{k = 1 + S_n/2}^n \sum_{j=0}^{n-k} \phi_k \bigg(\sum_{l = j+1}^{k+j} \epsilon^n(l)\bigg).
\end{multline}
The first and third term are constructed to contain only those elements of $\epsilon^n(l)$ which equal either $\epsilon(l)$ or $\cev \epsilon(l)$. Using this observation, we estimate these terms by
\begin{multline*}
    \sum_{k=1}^{S_n/2} \Bigg[ \sum_{j=0}^{S_n-k} \phi_k \bigg(\sum_{l = j+1}^{k+j} \epsilon^n(l)\bigg)
       + \sum_{j = n-S_n}^{n-k} \phi_k \bigg(\sum_{l = j+1}^{k+j} \epsilon^n(l) \bigg) \Bigg] \\
  %= \sum_{k=1}^{S_n/2} \sum_{j=0}^{S_n-k} \Bigg[ \phi_k \bigg(\sum_{l = j+1}^{k+j} \epsilon(l)\bigg) + \phi_k \bigg(\sum_{l = j+1}^{k+j} \cev \epsilon(l) \bigg) \Bigg] \\
  \leq \sum_{k=1}^\infty\sum_{j=0}^\infty 
  \bigg[ \phi_k \bigg(\sum_{l = j+1}^{k+j} \epsilon(l)\bigg) + \phi_k \bigg(\sum_{l=j+1}^{k+j} \cev \epsilon(l)\bigg) \bigg].
\end{multline*}
Since the right-hand side equals the first two terms of $E_\infty^1$ given by $Q_\infty$ and $\cev Q_\infty$, it remains to show that the second and fourth term in \eqref{for:ene:splitting:sumterm:limp} converge to $0$ as $n \to \infty$. 

We start by proving that the second term is small for large $n$. We observe that it solely contains those elements of $\epsilon^n(i)$ which equal either entries of the tails of $\epsilon$ and $\cev \epsilon$, or equal the (small) constant term in \eqref{for:defn:rec:seq}. For this reason, it turns out to be  enough to bound the second term by employing the quadratic upper bound of $\phi_k$ given by Lemma \ref{lem:lb_single_term} (it applies because of $\epsilon^n \in X_\delta$), and then applying Jensen's inequality. In more detail
\begin{multline*}
   \sum_{k=1}^{S_n/2} \sum_{j=S_n-k + 1}^{n - S_n - 1} \phi_k \bigg(\sum_{l = j+1}^{k+j} \epsilon^n(l) \bigg) 
  \leq \sum_{k=1}^{S_n/2} \sum_{j=S_n-k + 1}^{n - S_n - 1} 
  \frac{C_\delta}{k^{a+2}} k^2 \bigg(\frac{1}{k} \sum_{l = j+1}^{k+j} \epsilon^n(l) \bigg)^2
  \\\leq \sum_{k=1}^{S_n/2} \frac{C_\delta}{k^{a+1}} \sum_{j=S_n-k + 1}^{n - S_n - 1} \sum_{l = j+1}^{k+j} \bigabs{\epsilon^n(l)}^2 
  \leq \sum_{k=1}^{S_n/2} \frac{C_\delta}{k^{a+1}} \sum_{i= S_n/2 + 2}^{n - S_n/2 - 1} k \bigabs{\epsilon^n(i)}^2
  \\\leq \bigg(\sum_{k=1}^\infty\frac{C_\delta}{k^a}\bigg) \Biggbhaa{
    \sum_{i=S_n/2}^\infty \bighaa{ |\epsilon(i)|^2 + |\cev \epsilon(i)|^2 } +\sum_{i=S_n+1}^{n - S_n} \bigghaa{ \frac{u_n + \cev u_n}{n-2S_n} }^2
    |\epsilon^n(i)|^2 },
\end{multline*}
in which the right-hand side converges to $0$ as $n\to\infty$ by the same argument that we use for showing the convergence of the right-hand side in \eqref{for:rec:seq:conv}.

Finally, we show that the fourth term in \eqref{for:ene:splitting:sumterm:limp} converges to $0$. Since $k$ (the distance between particles in terms of their index) is large, we use similar arguments in the following estimate
\begin{multline*}
  \sum_{k = 1 + S_n/2}^n \sum_{j=0}^{n-k} \phi_k \bigg(\sum_{l = j+1}^{k+j} \epsilon^n(l)\bigg)
  \leq \sum_{k = 1 + S_n/2}^n \sum_{j=0}^{n-k} \frac{C_\delta}{k^{a+2}} k^2 \bigg(\frac{1}{k} \sum_{l = j+1}^{k+j} \epsilon^n(l) \bigg)^2 \\
  \leq \sum_{k = 1 + S_n/2}^n \frac{C_\delta}{k^{a+1}} \sum_{j=0}^{n-k}\sum_{l = j+1}^{k+j} \bigabs{\epsilon^n_l}^2 
  \leq \sum_{k = 1 + S_n/2}^n \frac{C_\delta}{k^a} \sum_{i=1}^n \bigabs{\epsilon^n(i)}^2
  = C_\delta \norm{ \epsilon^n }{\ell^2 (\N{})}^2 \sum_{k = 1 + S_n/2}^n \frac1{k^a},
\end{multline*}
which converges to $0$ as $n \to \infty$ since $a > 3/2 > 1$.
\end{proof}

\begin{rem}
A careful study of the above proof (in 
particular the proof of the limsup inequality) shows that the 
weaker condition given by $V''(x) < c_\delta |x|^{-a - 3/2}$ for 
any $x \in (\delta, \infty)$ would be enough as long as
\eqref{for:ests:V:Vprime} holds. Since we do not know of an
interesting example of an interaction potential which 
satisfies this weakened version of {\bf($a$-Dec)}, we
have assumed {\bf($a$-Dec)} for convenience. % PATRICK: I checked this carefully in my notes, labeled a.6.48.mid

Moreover, it may be possible to allow for interaction potentials 
$V$ whose tail decreases asymptotically slower than $x^{-a}$ for 
any $a > 3/2$, but asymptotically faster than $x^{-3/2}$. Once 
more, due to a lack of interesting examples of such potentials, we
have not studied this generalization.
\end{rem}

\subsection{Properties of the limit energy and the Euler--Lagrange equation}
\label{ssec:minz:plus:EL:eqn}

The fact that $E_\infty^1$ can be written as
\begin{equation*}
  E_\infty^1 (\epsilon, \cev \epsilon)
  = E_\infty^l (\epsilon) + E_\infty^r (\cev \epsilon)
\end{equation*}
shows that the interaction between the two boundary layers completely decouples as $n\to\infty$. For this reason we focus on $E_\infty^l$, whose domain is given by
\begin{equation*}
  \Dom E^l_\infty 
  := \bigaccv{ \epsilon \in \ell^2 (\N{}) }{ \epsilon(i) \geq -1 \: \: \forall \, i \geq 1 },
\end{equation*}

\begin{lem} \label{lem:ex:and:uniq:Elinf}
$E^l_\infty$ has a unique minimiser, denoted $\bar{\epsilon}$, on $\Dom E_\infty^l$. Moreover, $\bar \epsilon \in \operatorname{int} \bighaa{ \Dom E_\infty^l }$.
\end{lem}

\begin{proof}
We note that $E^l_\infty$ is bounded from below by 
\eqref{for:pf:thm:cpness:1} and \eqref{for:thm:Gamma:conv:limsup}. It is not identical to $\infty$, because $E^l_\infty(0)= 0$.

By the standard properties of $\Gamma$--convergence, Theorem \ref{thm:Gamma:conv} implies that the unique minimisers of $E_n^1$ (see Section \ref{ssec:BL:ene}) converge to a minimiser of $E^1_\infty$ in $\Dom E_\infty^1$. Since $E^1_\infty$ is strictly convex (by the argument that implies the strict convexity of $E_n^1$ in Section \ref{ssec:BL:ene}), the minimiser is unique. By the argument at the beginning of the proof of \eqref{for:thm:Gamma:conv:limsup}, it follows that  $\bar \epsilon \in \operatorname{int} \bighaa{ \Dom E_\infty^l }$.
\end{proof}

To compare the Euler--Lagrange equation with the equation for the boundary layer in \eqref{intro:eqn:BL:ell2}, we change variables once more. Let $u_i$ be the blown-up perturbation of the particles with respect to the equidistant configuration, i.e.
\begin{equation*}
  u(i) := \sum_{j=1}^i \epsilon(j),
  \quad u(0) := 0.
\end{equation*}
This transformation defines a bijection between $\Dom E_\infty^1$ and the set of sequences given by
\begin{equation*}
  \mathcal U := \bigaccv{ \big(0,u(1), u(2), \ldots\big) }{ Du \in \ell^2 (\N{}) \text{ and } Du(i) \geq -1 \text{ for all } i \geq 1 },
\end{equation*}
where $Du$ denotes the finite difference
\begin{equation*}
  Du(i) = u(i) - u(i-1). 
\end{equation*}
$\mathcal U$ is a subset of the Hilbert space
\begin{equation*}
  \mathcal{W}:=\bigaccv{ \big(0,u(1), u(2), \ldots\big) }{
  Du \in \ell^2 (\N{})}\quad\text{where}\quad
  (u,v)_{\mathcal{W}} = (Du,Dv)_{\ell^2(\N{})}.
\end{equation*}

We note that $\mathcal U$ is a convex subset of $\mathcal W$, and  since $\bar \epsilon$ is in the interior of $\Dom E_\infty^l$, it holds that $\bar u$, the image of $\bar{\epsilon}$ under this change of variable, is in the interior of $\mathcal U$. Therefore, the Euler--Lagrange equation is given by 
\begin{equation} \label{for:EL:eqn:u:impl}
  0 = \frac {\partial E_\infty^1 (u)}{ \partial u(i) },
  \quad \text{for all } i \geq 1.
\end{equation}
By taking the derivative of $E_\infty^1$, a straightforward calculation shows that the constant term in $\phi_k$ vanishes, and that the linear term of $\phi_k$ cancels with the linear term involving $\sigma^\infty$. The explicit form of \eqref{for:EL:eqn:u:impl} is given by
\begin{equation} \label{for:EL:eqn:u:expl}
  \left\{ \begin{aligned}
    &0 = \sum_{\substack{ j = 0 \\ j \neq i }}^\infty V' \big(u(i) - u(j) + i-j\big),
    \quad \text{for all } i \geq 1, \\
    &0 = u(0), \\
    &Du \in \ell^2 (\N{}).
  \end{aligned} \right.
\end{equation}
It is easy to see that \eqref{for:EL:eqn:u:expl} is equivalent to \eqref{intro:eqn:BL:ell2}.

\subsection{The case $a < 3/2$}
\label{ssec:a:smaller:3over2}

Here we motivate why the case $a < 3/2$ is significantly different from Theorem \ref{thm:Gamma:conv}. We do this by separating two scenarios; $a \leq 1$ and $1 < a < 3/2$.

Many steps in the proof of Theorem \ref{thm:Gamma:conv} require $a > 1$. The main reason for this requirement is that $a \leq 1$ does not guarantee integrability of the tail of $V$. Moreover, our choice for the variables given by $\epsilon^n$ relies heavily upon the fact that the particles in the bulk should be equispaced.
As remarked above, this was expected due to the results of 
\cite{VanMeursMunteanPeletier14}: these break down when the tail of $V$ fails to be
integrable. 

The case $1 < a < 3/2$ is more delicate. To illustrate that Theorem \ref{thm:Gamma:conv} does not hold, we show that the functional $E_\infty^1$ is not bounded from below in this case. As a
consequence, it would seem that there is a term which has an energy scaling which is neglected in this case, and lies between the bulk energy due to the equispaced configuration and the boundary layer energy. We expect that this term describes a correction to the bulk profile, which appears because the boundary layers begin to interact with each other as the interactions become more nonlocal.

To show that $E_\infty^1$ is not bounded from below when $1 < a < 3/2$, we assume that $V''(x) \simeq x^{-a-2}$ for large $x$, which implies (by a similar argument leading to \eqref{for:ests:V:Vprime}) that $|V'(x)| \simeq x^{-a-1}$ and $V(x) \simeq x^{-a}$ for large $x$. We claim that $\sigma^\infty_i \simeq i^{1-a}$, which implies
\begin{equation*}
  \sigma^\infty \in \ell^p (\N{})
  \: \Longleftrightarrow \:
  p > \frac1{a - 1}.
\end{equation*}
To prove this claim, we use $V'(x) \simeq x^{-a-1}$ for large $x$ to obtain
\begin{equation*}
  \sigma^\infty_i 
  = \sum_{k = i+1}^\infty (k-i) \bigabs{ V' (k) }
  \simeq \sum_{j = 1}^\infty \frac j{(j + i)^{a+1}}
  \quad \text{for large } i.
\end{equation*}
Then, we estimate from above and below to obtain the desired result:
\begin{align*}
  \sigma^\infty_i 
  &\lesssim \sum_{j = 1}^\infty \frac1{(j + i)^{a}} 
  \simeq i^{1 - a}, \\
  \sigma^\infty_i 
  &\geq \sum_{j = 1}^i \frac j{(2 i)^{a+1}} + \sum_{j = i+1}^\infty \frac{1/2}{(j + i)^{a}}
  \simeq i^{1 - a}.
\end{align*}

Next we estimate $Q_\infty (\epsilon)$ (i.e.~the first term of $E_\infty^1$ as in \eqref{for:defn:Einfty}) from above for $(\epsilon, \epsilon) \in \Dom E_\infty^1$ satisfying $-1/2 \leq \epsilon(i) \leq 0$ for all $i \geq 1$. By first using Lemma \ref{lem:lb_single_term} and then applying Jensen's inequality, we obtain
\begin{align*}
  Q_\infty (\epsilon)
  = \sum_{k=1}^\infty\sum_{j=0}^\infty \phi_k \bigg(\sum_{l = j+1}^{k+j} \epsilon(l)\bigg)
  \lesssim \sum_{k=1}^\infty\sum_{j=0}^\infty k^{-a-2} k^2 \bigg( \frac1k \sum_{l = j+1}^{k+j} \epsilon(l)\bigg)^2
  \leq \sum_{k=1}^\infty\sum_{j=0}^\infty k^{-a-1} \sum_{l = j+1}^{k+j} \epsilon(l)^2.
\end{align*}
By the computations in Appendix \ref{app:sigman:computs} it then follows that
\begin{align*}
  Q_\infty (\epsilon)
  \lesssim \sum_{j = 1}^{\infty} \bigghaa{ \sum_{k=1}^\infty k^{-a} } \epsilon_j^2
  \simeq \norm{ \epsilon }{\ell^2 (\N{})}^2.
\end{align*}

It is now easy to find $(\epsilon, \cev \epsilon) \in \Dom E_\infty^1$ for which $E_\infty^1 (\epsilon, \cev \epsilon) = -\infty$. We set $\epsilon(i) = - i^{-b}/2$ with $1/2 < b < 2 - a$ and $\cev \epsilon = \epsilon$. By these choices, we have
\begin{gather*}
    (\sigma^\infty, \epsilon)_{\ell^2 (\N{})} 
    \simeq - \sum_{i=1}^\infty i^{1-a} i^{-b}
    \leq - \sum_{i=1}^\infty \frac1i
    = -\infty, \\
    Q_\infty (\epsilon) 
    \lesssim \sum_{i=1}^\infty \frac1{i^{2b}}
    = C,
\end{gather*}
and analogous estimates for $\cev \epsilon$, and thus $E_\infty^1 (\epsilon, \cev \epsilon) = -\infty$.

\subsection{Computation of $\sigma^n$}
\label{app:sigman:computs}

In this section we fix $n$, and show that 
\begin{equation} \label{for:linr:term}
   \sum_{k=1}^n \sum_{j=0}^{n-k}
   \sum_{l=j+1}^{k+j} V'(k) \epsilon(l) \\
 = \sigma \cdot (\epsilon + \cev \epsilon),
\end{equation}
where $\sigma$ is given by \eqref{for:defn:sigman}, i.e.
\begin{equation*} 
  \sigma_i 
  = \sum_{k = i+1}^{n-i} \bigbhaa{ (k-i) \wedge (n-i+1-k) } \bigabs{ V' (k) }
  \quad \text{for } i = \acc{1, \ldots, \lfloor n/2 \rfloor}.
\end{equation*}

We start by changing the order of summation in %. By considering Figure~\ref{fig:change:order:summ:1} the following manipulation is seen to hold:
\begin{multline*} %\label{for:change:of:summ}
   \sum_{k=1}^n V'(k) \sum_{j=0}^{n-k}
   \sum_{l=j+1}^{k+j} \epsilon(l)
 = \sum_{k=1}^n \sum_{l=1}^n \epsilon(l) V'(k) \sum_{j = 0 \vee (l - k) }^{ (n-k) \wedge (l-1) } 1 \\
 = \sum_{k=1}^n \sum_{l=1}^n V'(k) \bigbhaa{ k \wedge l \wedge (n-k+1) \wedge (n-l+1) } \epsilon(l) =: v \cdot \tilde A \epsilon,
\end{multline*}
where the vector $v \in \R n$ is defined by $v_k := V'(k) < 0$, and the matrix $\tilde A \in \R{n \times n}$ is illustrated in Figure \ref{fig:index:matrices}. 

%\begin{figure}[ht]
%\centering
%\begin{tikzpicture}[scale=0.8, >= latex]
%    \def \r {0.1}
%    
%	\draw[->] (0,0) -- (0,5.5) node [above] {$j$};
%	\draw[->] (0,0) -- (8.5,0) node [right] {$l$}; 
%	\draw (1,0) -- (6,5) -- (8,5) -- (3,0);
%	\draw[dotted] (0,5) -- (6,5);
%	 
%    \foreach \i in {1,...,5}
%      {
%      \fill (\i, 0) circle (\r);
%      \fill (0, \i) circle (\r);
%      \fill (\i + 1, \i) circle (\r);
%      \fill (\i + 3, \i) circle (\r);
%      }
%    \foreach \i in {0,6,7,8}
%      {
%      \fill (\i, 0) circle (\r);
%      }
%    \fill (7, 5) circle (\r);
%      
%    \draw (0-\r,0) node[left] {$0$};
%    \draw (0,0-\r) node[below] {$0$};
%    \draw (0-\r,3) node[left] {$k$};
%    \draw (0-\r,5) node[left] {$n-k$};
%    \draw (1,0-\r) node[below] {$1$};
%    \draw (3,0-\r) node[below] {$k$};
%    \draw (6,0-\r) node[below] {$n-k$};
%    \draw (8,0-\r) node[anchor = north] {$n$};
%    \draw (5.5,2.5) node[anchor = north west] {$j = l - k$};
%    \draw (4.5,3.5) node[anchor = south east] {$j = l - 1$};
%\end{tikzpicture}
%\caption{Schematic picture for the change of summation in \eqref{for:change:of:summ} for some fixed $k \in \acc{1, \ldots, n}$.}\label{fig:change:order:summ:1}
%\end{figure}

Since $\epsilon \cdot \mathbf 1 = 0$, it holds that $\tilde A \epsilon = (\tilde A - B) \epsilon$ for any matrix $B$ whose rows are multiples of $\mathbf 1$. We take $B$ such that the entries in its $i$-th row equal $i \wedge (n-i+1)$, and set $A := -(\tilde A - B)^T$ (see Figure \ref{fig:index:matrices} for its structure). Then
\begin{equation*}
  v \cdot \tilde A \epsilon
  = -v \cdot A^T \epsilon
  = -\epsilon \cdot A v
  = \sum_{l = 1}^{\lfloor n/2 \rfloor} \bigbhaa{ 
      (-A v)(l) \epsilon(l) + (-A v)(n-l+1) \cev \epsilon(l) }
\end{equation*}
From Figure \ref{fig:index:matrices} it is easy to see that the vector $-A v$ has reversal symmetry, and \eqref{for:linr:term} follows.

\begin{figure}[ht]
\centering
\begin{tikzpicture}[scale=0.35, >= latex]
    \def \r {0.15}
    
\begin{scope}[shift={(0,0)},scale=1]
  \foreach \i in {1,...,12}{
    \foreach \j in {1,...,12}{
      \fill[gray] (\i, \j) circle (\r);
    }
  }
  
  \foreach \i in {1,2,5,6}{
    \draw (\i, \i) rectangle (13 - \i, 13 - \i);
  }
  \foreach \i in {3,4}{
    \draw[dotted] (\i, \i) rectangle (13 - \i, 13 - \i);
  }
  
  \draw (1-\r,12) node[left] {$1$};
  \draw (1-\r,1) node[left] {$n$};
  \draw (1,12+\r) node[above] {$1$};
  \draw (12,12+\r) node[above] {$n$};
  \draw[<-] (12,11.5) -- (13,11.5) node[right] {$1$};
  \draw[<-] (11,9.5) -- (13,9.5) node[right] {$2$};
  \draw[<-] (8,7.5) -- (13,7.5) node[right] {$\dfrac n2 - 1$};
  \draw[<-] (7,6.5) -- (13,5) node[right] {$\dfrac n2$}; 
  \draw (0.5,6.5) node[left] {\LARGE $\tilde A =$};
\end{scope}

\begin{scope}[shift={(22,0)},scale=1]
  \foreach \i in {1,...,12}{
    \foreach \j in {1,...,12}{
      \fill[gray] (\i, \j) circle (\r);
    }
  }
  
  \foreach \i in {1,2,5,6}{
    \draw (\i, 1) -- (6, 7 - \i) -- (7, 7 - \i) -- (13 - \i, 1);
    \draw (\i, 12) -- (6, 6 + \i) -- (7, 6 + \i) -- (13 - \i, 12);
  }
  \foreach \i in {3,4}{
    \draw[dotted] (\i, 1) -- (6, 7 - \i) -- (7, 7 - \i) -- (13 - \i, 1);
    \draw[dotted] (\i, 12) -- (6, 6 + \i) -- (7, 6 + \i) -- (13 - \i, 12);
  }
  \draw (6,6) -- (6,7);
  \draw (7,6) -- (7,7);
  
  \draw (1-\r,12) node[left] {$1$};
  \draw (1-\r,1) node[left] {$n$};
  \draw (1,12+\r) node[above] {$1$};
  \draw (12,12+\r) node[above] {$n$};
  \draw[<-] (11.5,11.5) -- (13,11.5) node[right] {$0$};
  \draw[<-] (9.5,10.5) -- (13,9.5) node[right] {$1$};
  \draw[<-] (7.5,11.5) -- (13,7.5) node[right] {$\dfrac n2 - 2$};
  \draw[<-] (6.5,1) -- (13,5) node[right] {$\dfrac n2 - 1$}; 
  \draw (0.5,6.5) node[left] {\LARGE $A =$};
  \draw (10.5, 6.5) node {\huge $0$};
  \draw (2.5, 6.5) node {\huge $0$};
\end{scope}
\end{tikzpicture}
\caption{Schematic picture of the two index matrices $\tilde A$ and $A = \tilde A - B$. The lines are level sets which connect entries with the same value. We have taken $n$ to be even for convenience.}\label{fig:index:matrices}
\end{figure}
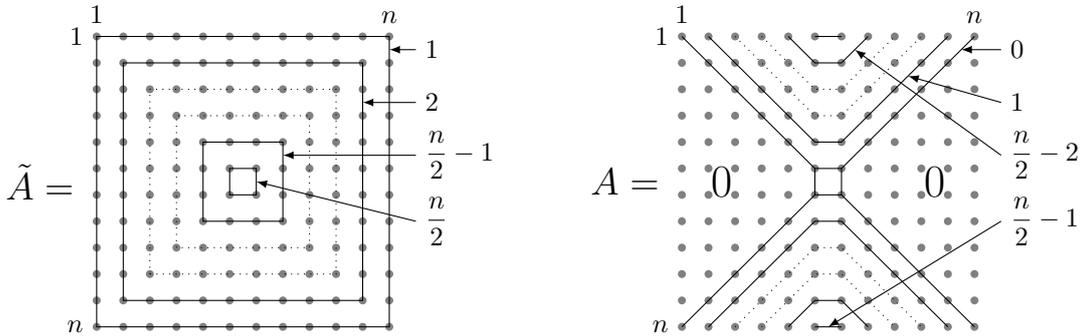

\section{Numerics}
\label{sec:numerics}
In this section, we investigate by means of numerical computations to what extent the solution to the equation for the boundary layer in \eqref{intro:eqn:BL:ell2} matches with the solution to the force balance in \eqref{eq:equilibrium} for several values of $n$. The computations are performed for two physically--motivated choices of the interaction potential $V$.

\subsection{Numerical method for solving \eqref{intro:eqn:BL:ell2}.}

To approximate the infinite sum in \eqref{intro:eqn:BL:ell2} by a finite sum which depends on a finite set of unknowns, we assume the particles to be equispaced after a fixed index $I$, i.e.~$y_j = y_I + (j - I)$ for all $j \geq I$. This choice is equivalent to finding a minimiser of $E_\infty^l$ in $\bigaccv{ \epsilon \in \Dom E_\infty^l }{ \epsilon(j) = 0 \text{ for all } j > I }$, and the related force balance can be written as
\begin{equation}
  0 = \sum_{ \substack{ j = 0 \\ j \neq i } }^I - V' ( y_i - y_j )
      - \sum_{ j = I + 1 }^\infty V' ( y_i - y_I + I - j ),
      \qquad \forall \ i = 1, \ldots, I.
      \label{num:sums:temp}
\end{equation}

In order to deal with the infinite sum in \eqref{num:sums:temp} numerically, we introduce a second approximation. We use the Euler--Maclaurin summation formula (see, for example, \cite{CKP}) to approximate the tail of the sum, given by all indices $j$ for which are larger than some fixed index $J > I$, with an integral. To this end, we use oddness of $V'$ to rewrite the infinite sum in \eqref{num:sums:temp} as
\begin{equation*}
   \sum_{ j = I + 1 }^\infty V' ( y_i - y_I + I - j )
   = \sum_{ j = I + 1 }^{J-1} V' ( y_i - y_I + I - j )
      - \sum_{ k = 0 }^{\infty} V' ( y_I - y_i - I + J + k )
\end{equation*} 
The Euler--Maclaurin summation formula gives us the result that
\begin{equation*}
 \sum_{k = 0}^{\infty} V' ( y_I - y_i + J-I + k )
 = \int_0^{\infty} V'( y_I - y_i + J - I + r) \ \mathrm{d} r
 - \tfrac{1}{2} V'(y_I - y_i + J - I)
 + R,
\end{equation*}
where the remainder term $R$ is given by (given that $V \in C^3 ([J-I, \infty))$)
\begin{equation} \label{for:defn:R:E-McL}
  R := \int_0^\infty \frac{B_2(\{r\})}{2} V'''(y_I - y_i + J - I + r) \ \mathrm{d} r,
\end{equation}
where $B_2(\{\cdot\})$ is the 1-periodic extension of the second Bernoulli polynomial. If we further assume that $V''' \leq 0$, we can use $\| B_2(\{\cdot\}) \|_\infty = \tfrac16$ to estimate
\[
  |R| 
  \leq \frac1{12} \int_0^\infty \big| V'''(y_I - y_i + J - I + r) \big| \ \mathrm{d} r
  = \frac1{12} V''(y_I - y_i + J - I)
  \lesssim |J-I|^{-a-2}.
\]
Depending on the regularity and monotonicity properties of a given $V$, we can obtain stronger estimates on $R$ by using integration by parts in \eqref{for:defn:R:E-McL}.
Neglecting $R$, we obtain the following non-linear system of equations:
\begin{equation} \label{force:bal:BL:numl}
\left\{ \begin{aligned}
  0 &= - \frac 12 V' ( y_i - y_J ) - V ( y_i - y_J ) - \sum_{ \substack{ j = 0 \\ j \neq i } }^{J-1} V' ( y_i - y_j ),
      &&\forall \ i = 1, \ldots, I, \\
  y_j &= y_I + j - I, &&\forall \ j = I+1, \ldots, J, \\
  y_0 &= 0.
\end{aligned} \right.
\end{equation}
%\TOM{[State numerical method for solving this in practice here: I
%presume it's gradient descent?]} \PATRICK{[Ideally we should mention the solver, but you will not like the answer; I solve it by computing the gradient flow with a variable time-step time integrator (ode15s). Reason: Newton did not work; the trivial initial guess fails. 
%
%I see two ways out: not mention the solving method (it is just a system of non-linear equations; there is no numerical challenge here), or redoing all the numerics for non-linear conjugate gradients, and after a few iterations using Newton for fast convergence. At this state I vote for the first option, and if the referee complains, I take my losses and redo the numerics.
%
%As a sidenote; steepest descent does not work; our energy landscape is convex, but the spectrum of the Hessian ranges over orders of magnitude, which makes steepest descent oscillate like crazy.]}
%
%\CAMERON{[I'm happy to leave this out of the paper and see what referees say. I find it interesting that the numerics are a bit tricky for interesting reasons, but agree that it's probably best to keep this out for now!]}
%\TOM{[Me too.]}

\subsection{Computations for $V (x) = x^{-2}$}
\label{sec:Numerics:a=2}

This potential is considered in \cite{Hall2010}, and models the interaction between dislocation dipoles. It satisfies all basic assumptions on $V$ described in \S\ref{sec:Setting}, as well as both strengthened hypotheses {\bf(Reg+)} and {\bf(Dec+)} given in \S\ref{sec:Results}. Figure \ref{fig:eps:a2} depicts the solution $\epsilon^l$ to \eqref{force:bal:BL:numl} (with $I = 10^3$ and $J = 10^2$), together with the minimiser of $E_n^1$. 

\begin{figure}[ht]
\centering
\begin{tikzpicture}[scale=1.5]
\node (label) at (0,0){\includegraphics[width=3in]{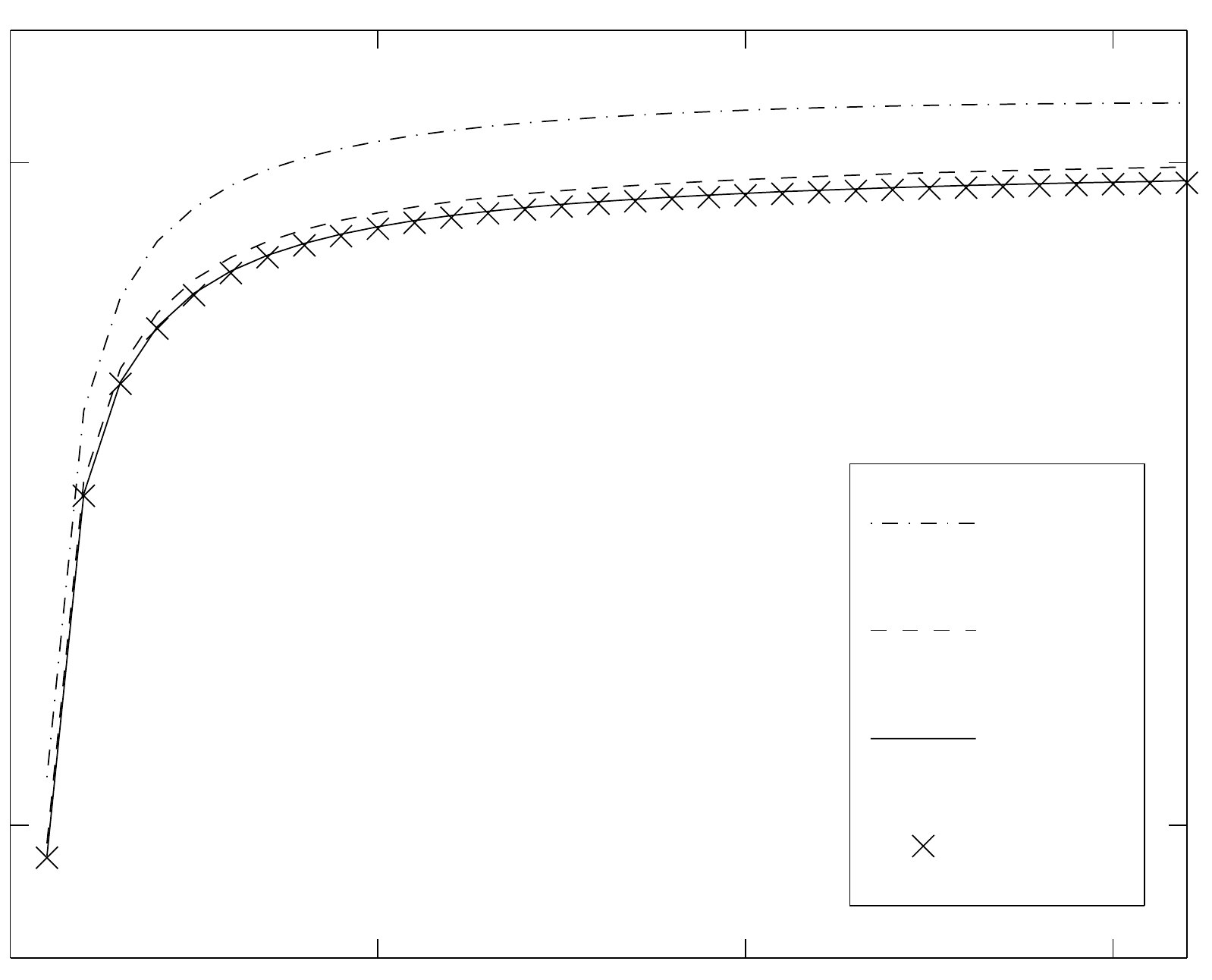}};
\draw (0,-2) node[below] {\LARGE $i$};
\draw (-2.5,0) node[left] {\LARGE $\epsilon^n (i)$};
\draw (-0.94,-2) node[below] {$10$};
\draw (-2.49,-2) node[below] {$0$};
\draw (0.6,-2) node[below] {$20$};
\draw (2.12,-2) node[below] {$30$};
\draw (-2.5,-1.4) node[left] {$-0.1$};
\draw (-2.5, 1.4) node[left] {$0$};
\draw (1.6,-1.5) node[right] {\scriptsize $\epsilon^l (i)$};
\draw (1.6,-1.02) node[right] {\tiny $n = 2^{12}$};
\draw (1.6,-0.58) node[right] {\tiny $n = 2^9$};
\draw (1.6,-0.12) node[right] {\tiny $n = 2^6$};
\end{tikzpicture}
\caption{The three line plots of the minimisers $\epsilon^n (i)$ for $n = 2^6, 2^9, 2^{12}$ illustrate the convergence to the solution $\epsilon^l$ ($\times$) of \eqref{force:bal:BL:numl}.}
\label{fig:eps:a2}
\end{figure}
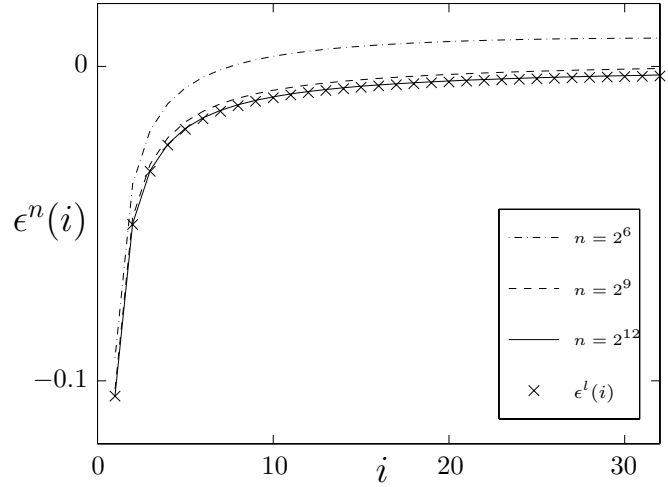

Next we test the rate of convergence at which $\epsilon^n (i)$ converges to $\epsilon^l (i)$ as $n \to \infty$ for fixed $i$. Since we consider several values of $n$ that are larger than $I$, the accuracy of $\epsilon^n (i)$ may be greater than the accuracy of our solution method for finding $\epsilon^l (i)$. We therefore consider the incremental error given by
\begin{equation} \label{for:defn:incr:error:eps:a2}
  d^n(i) := |\epsilon^{n, 1/2}(i) - \epsilon^{2n, 1/2}(i)| \sim \frac 1 {n^p} \quad \text{for \emph{any} fixed } i \leq n/2.   
\end{equation}
Figure \ref{fig:dn:a2} illustrates the decay rate of $d^n(i)$. This rate is fairly similar to $n^{-1} \log n$, which is expected from \eqref{eq:Formal:a2:logterm}.

\begin{figure}[b]
\centering
\begin{tikzpicture}[scale=1.5]
\node (label) at (0,0){\includegraphics[width=3in, trim={0 6.9cm 0 6.9cm}, clip]{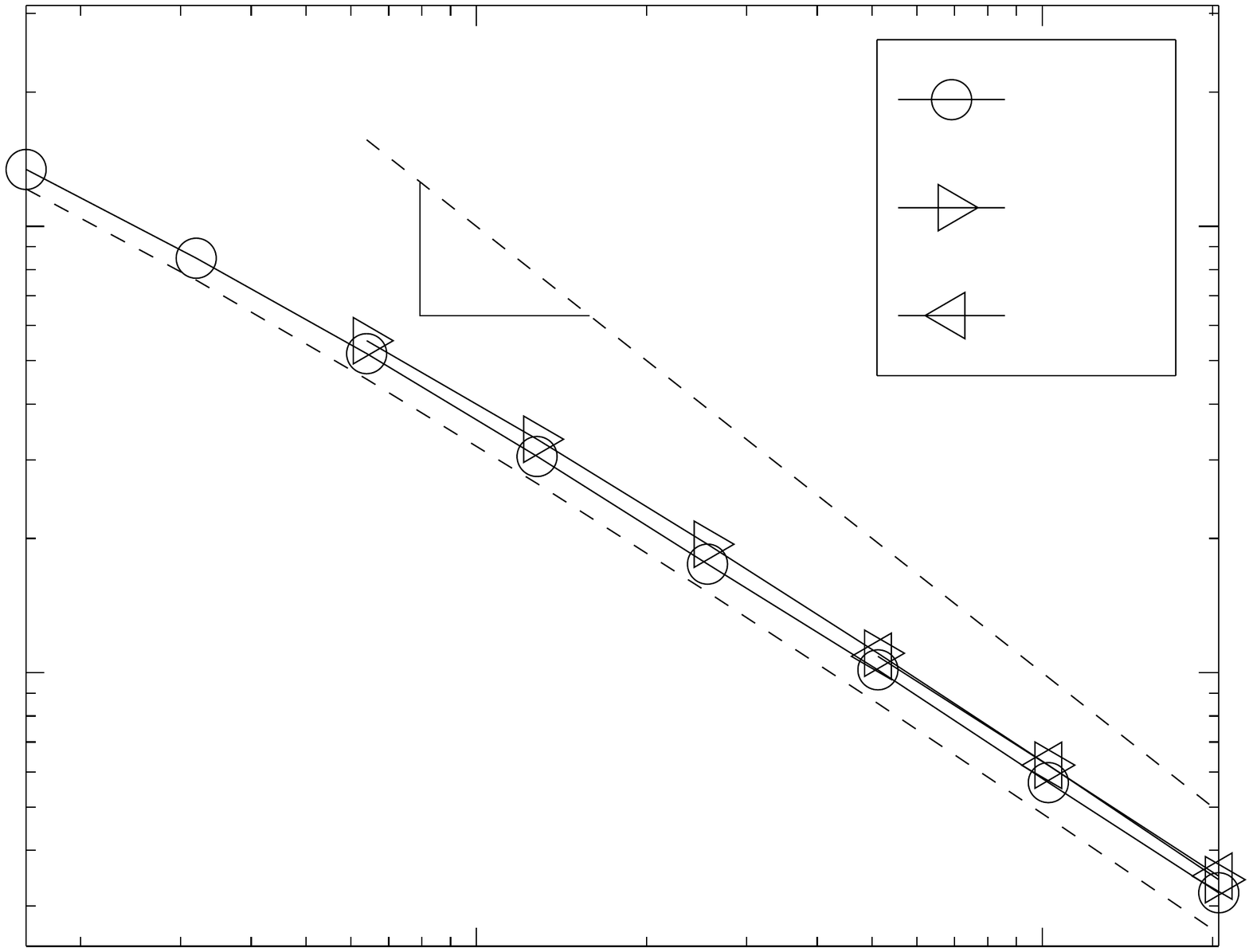}};
\draw (0,-2) node[below] {\huge $n$};
\draw (-2.5,0) node[left] {\LARGE $d^n(i)$};
\draw (-0.56,-2) node[below] {$10^2$};
\draw (1.74,-2) node[below] {$10^3$};
\draw (-2.5, 1) node[left] {$10^{-2}$};
\draw (-2.5,-0.8) node[left] {$10^{-3}$};
\draw (1.55,1.5) node[right] {\footnotesize $d^n(1)$};
\draw (1.55,1.06) node[right] {\footnotesize $d^n(9)$};
\draw (1.55,0.62) node[right] {\footnotesize $d^n(81)$};
\draw (-0.45,-0.45) node {$C \dfrac{\log n}n$};
\draw (-1,0.9) node {$-1$};
\end{tikzpicture}
\caption{The decay rate of the incremental error $d^n(i)$ seems to be independent of $i$. The graphs of $n \mapsto d^n(i)$ for $i = 3, 27$ are visibly indistinguishable from those for $i = 9, 81$, and are therefore omitted. The decay rate of $d^n(i)$ seems similar to $O(n^{-1} \log n)$.}
\label{fig:dn:a2}
\end{figure}

With $\epsilon^l$ we can improve the equispaced zeroth-order estimate of the rescaled minimiser of $E_n$ by
\begin{equation*}
  \hat y(i) := i + \sum_{j = 1}^i \epsilon^l(j) =: n \hat x^n (i),
  \quad \text{for } i = 1,\ldots,\frac n2.
\end{equation*}
Figure \ref{fig:x:y:a2} depicts the related density plots, both in the original variable $x$ (a), and in the rescaled variable $y$ (b). Given $x \in \mathcal D_n$, we define the `discrete density' by
\begin{equation} \label{for:defn:rhoi}
  \rho(i) 
  := \frac{2}{ n \big( x(i+1) - x(i-1) \big) }
  %= \frac{2}{ y(i+1) - y(i-1) }
  = \frac{2}{ \eps(i+1) + \eps(i) + 2 },
  \quad \text{for all } i =1, \ldots, n-1,
\end{equation}
which measures locally how close neighbouring particles are to each other. The expression in terms of $\epsilon$ shows that, in Figure \ref{fig:x:y:a2}.(b), the $n$-dependent offset of the predictor $\hat y$ is in line with the observations from Figures \ref{fig:eps:a2} and \ref{fig:dn:a2}.

\begin{figure}[h]
\centering
\begin{tikzpicture}[scale=1.25]
\node (label) at (0,0){\includegraphics[width=2.5in]{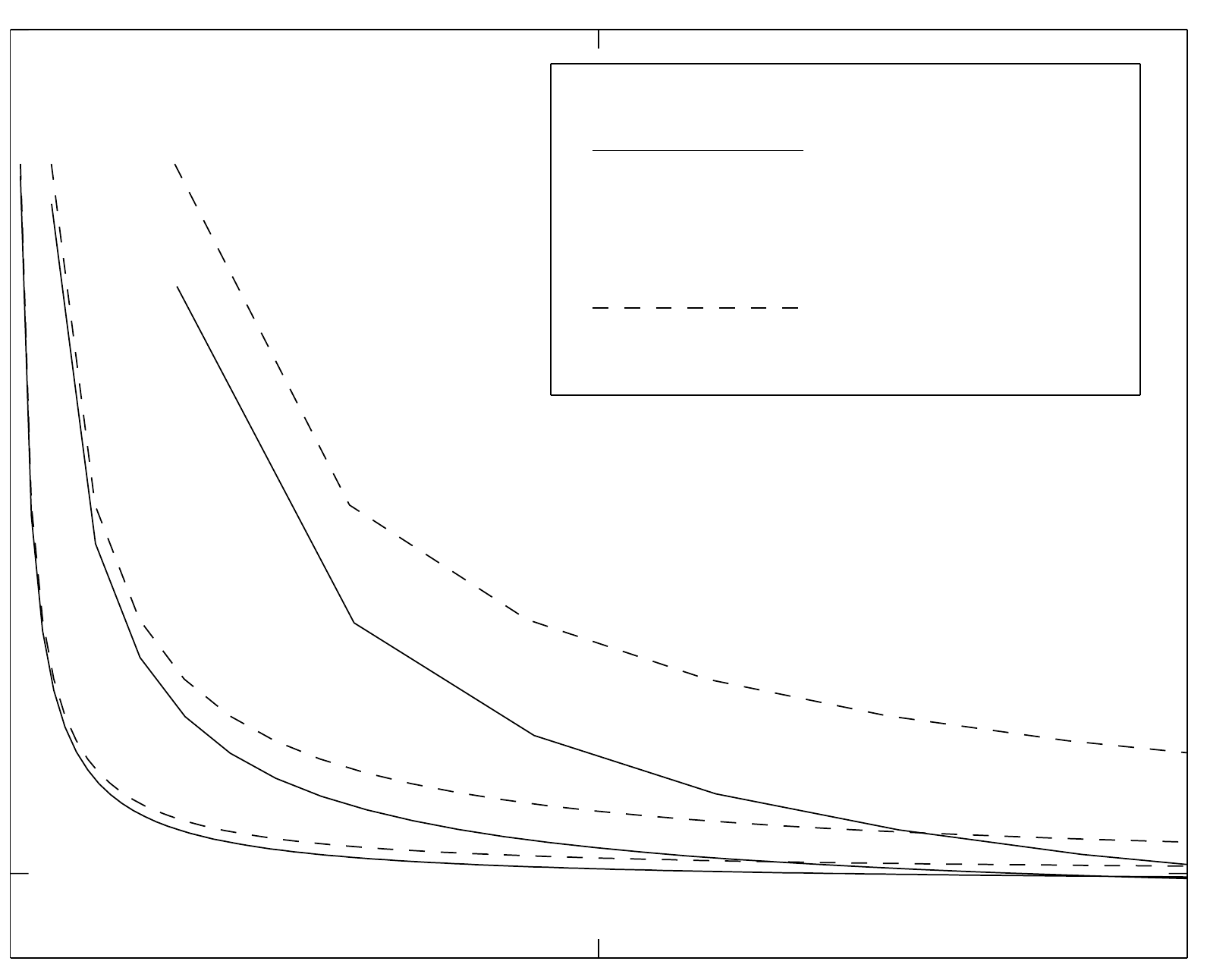}};
\draw (0,2) node[above] {(a)};
\draw (1.25,-2) node[below] {\LARGE $x$};
\draw (-2.5,0) node[left] {\LARGE $\rho$};
\draw (-2.49,-2) node[below] {$0$};
\draw (0,-2) node[below] {$0.05$};
\draw (2.49,-2) node[below] {$0.1$};
\draw (-2.5, 1.95) node[left] {$1.1$};
\draw (-2.5,-1.6) node[left] {$1$};
\draw (0.82, 1.43) node[right] {\tiny $(x^n(i), \rho^n(i))_i$};
\draw (0.9, 0.73) node[right] {\tiny $(\hat x^n(i), \hat \rho(i))_i$};

\begin{scope}[shift={(6,0)}]
\node (label) at (0,0){\includegraphics[width=2.5in]{ya2.pdf}};
\draw (0,2) node[above] {(b)};
\draw (0,-2) node[below] {\LARGE $y$};
\draw (-2.5,0) node[left] {\LARGE $\rho$};
\draw (-2.5, 1.95) node[left] {$1.1$};
\draw (-2.5,-1.6) node[left] {$1$};
\draw (-0.94,-2) node[below] {$10$};
\draw (-2.49,-2) node[below] {$0$};
\draw (0.6,-2) node[below] {$20$};
\draw (2.12,-2) node[below] {$30$};
\draw (1.1,1.5) node[right] {$n = 2^6$};
\draw (1.1,0.92) node[right] {$n = 2^8$};
\draw (1.1,0.35) node[right] {$n = 2^{10}$};
\draw (0.9,-0.25) node[right] {\footnotesize $(\hat y(i), \hat \rho(i))_i$};
\end{scope}
\end{tikzpicture}
\caption{Comparison between: (a) the minimiser $\rho^n(i)$ and the predictor $\hat \rho(i)$, and (b) their horizontally rescaled counterparts by a factor $n$. More precisely, the line graphs are the linear interpolation between the $i$-indexed sets of coordinates $(x^n(i), \rho^n(i))_i$ and $(\hat x^n(i), \hat \rho(i))_i$ in graph (a), and $(y^n(i), \rho^n(i))_i$ and $(\hat y(i), \hat \rho(i))_i$ in graph (b), where $\rho^n$ and $\hat \rho$ are defined in \eqref{for:defn:rhoi} with respect to $x^n$ and $\hat x^n$. We use $n = 2^6, 2^8, 2^{10}$.}
\label{fig:x:y:a2}
\end{figure}

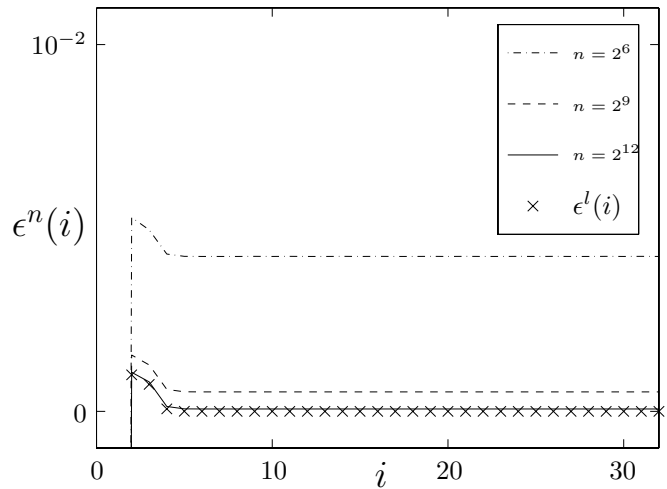
\begin{figure}[b]
\centering
\begin{tikzpicture}[scale=1.5]
\node (label) at (0,0){\includegraphics[width=3in]{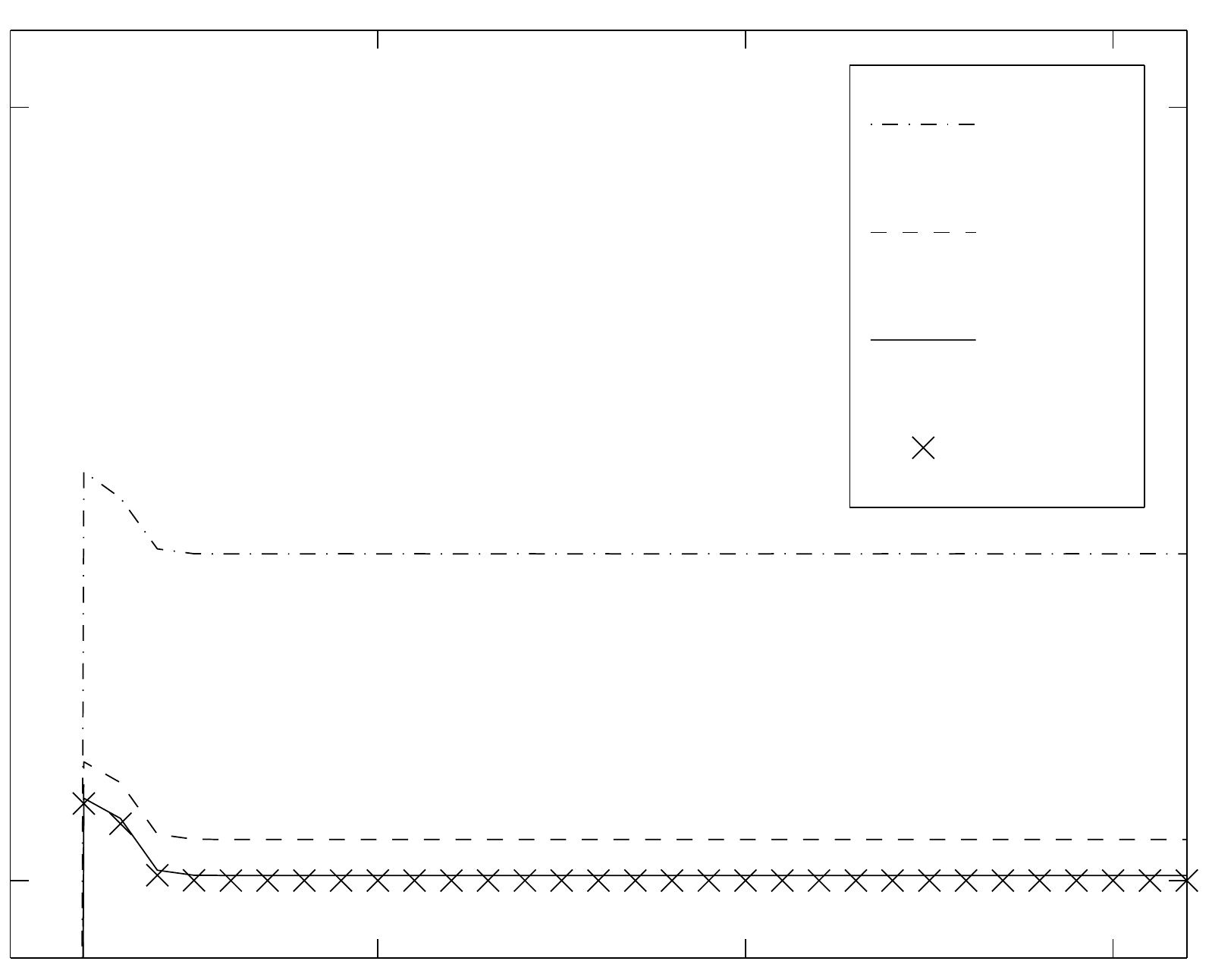}};
\draw (0,-2) node[below] {\LARGE $i$};
\draw (-2.5,0) node[left] {\LARGE $\epsilon^n (i)$};
\draw (-0.94,-2) node[below] {$10$};
\draw (-2.49,-2) node[below] {$0$};
\draw (0.6,-2) node[below] {$20$};
\draw (2.12,-2) node[below] {$30$};
\draw (-2.5,-1.66) node[left] {$0$};
\draw (-2.5, 1.62) node[left] {$10^{-2}$};
\draw (1.6,0.18) node[right] {$\epsilon^l (i)$};
\draw (1.6,0.65) node[right] {\tiny $n = 2^{12}$};
\draw (1.6,1.08) node[right] {\tiny $n = 2^9$};
\draw (1.6,1.54) node[right] {\tiny $n = 2^6$};
\end{tikzpicture}
\caption{The three line plots of the minimisers $\epsilon^n (i)$ for $n = 2^6, 2^9, 2^{12}$ illustrate the convergence to the solution $\epsilon^l$ ($\times$) of \eqref{force:bal:BL:numl}. The data for $\epsilon^n (1)$ are $-O (10^{-1})$, which are relatively far away from the range of the vertical axis.}
\label{fig:eps:Vwall}
\end{figure}

\subsection{Computations for $V (x) = x \coth x - \log | 2 \sinh x |$}

This potential describes the interaction of dislocation walls \cite{Geers2013}. It satisfies all assumptions on $V$ for any $a \in \R{}$. As in Figure~\ref{fig:eps:a2}, Figure~\ref{fig:eps:Vwall} shows the solution $\epsilon^l$ to \eqref{force:bal:BL:numl} (with $I = 200$ and $J = 20$), together with the minimisers $\epsilon^{n, 1/2}$ of $E_n^1$. An intriguing difference with Figure \ref{fig:eps:a2} is that the profiles of $\epsilon^{n, 1/2}$ and $\epsilon^l$ are not monotone: this shows that different potentials satisfying our imposed assumptions can result in qualitatively different boundary layer profiles. Moreover, the values of $\epsilon^{n, 1/2} (i)$ and $\epsilon^l (i)$ are an order of magnitude smaller than those for the $-2$-homogeneous potential, and they decay faster to $0$ as $i$ increases.

Figure \ref{fig:PS:bdd:Vwall} suggests that the incremental errors $d^n(i)$ defined in \eqref{for:defn:incr:error:eps:a2} decay as $n^{-1}$. Again, we observe that this decay is independent of $i$. The decay is faster in comparison to the $-2$-homogeneous interaction potential.

\begin{figure}[ht]
\centering
\begin{tikzpicture}[scale=1.5]
\node (label) at (0,0){\includegraphics[width=3.0375in]{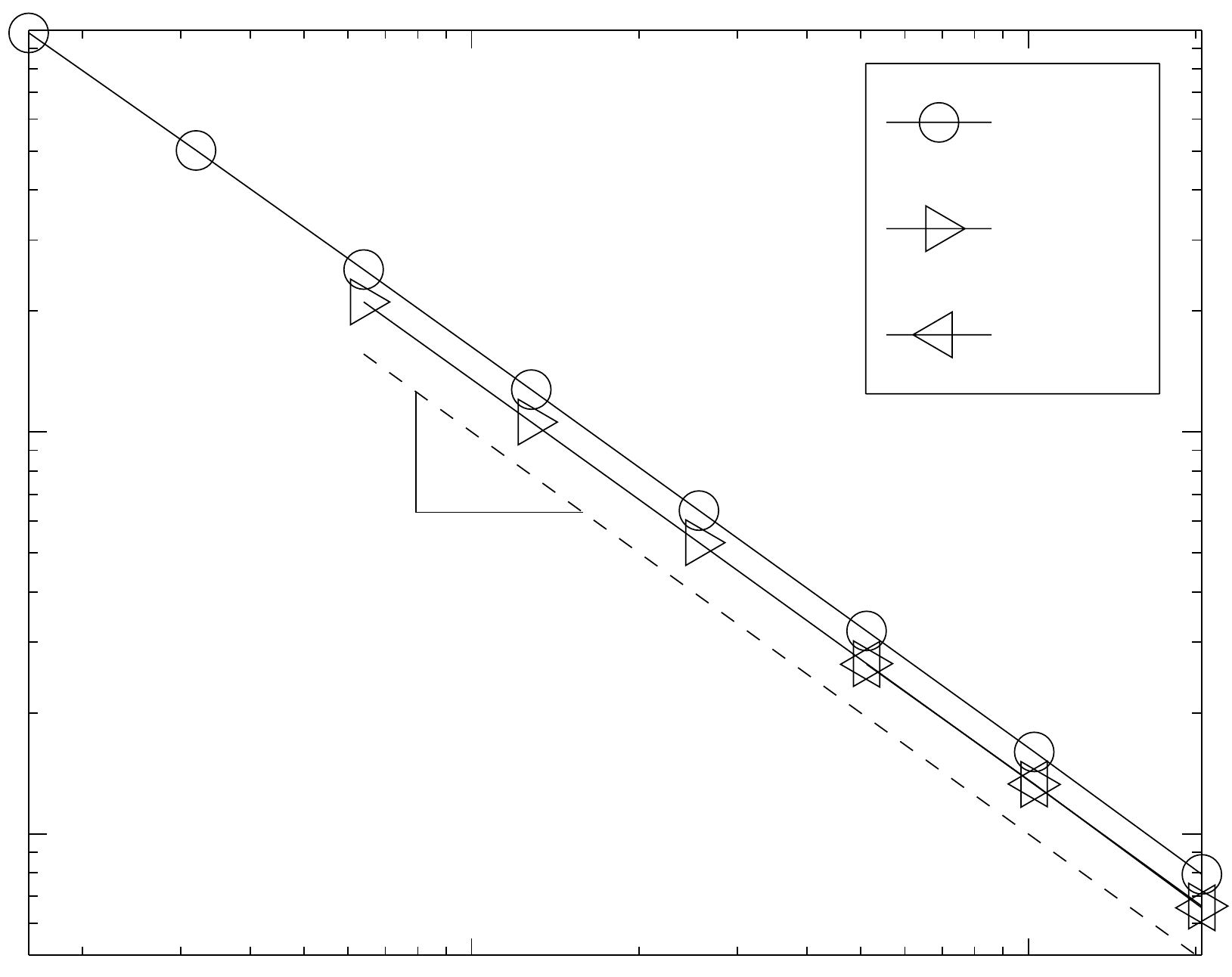}};
\draw (0,-2) node[below] {\huge $n$};
\draw (-2.5,-0.2) node[left] {\LARGE $d^n(i)$};
\draw (-0.56,-2) node[below] {$10^2$};
\draw (1.74,-2) node[below] {$10^3$};
\draw (-2.5,1.9) node[left] {$10^{-2}$};
\draw (-2.5,0.25) node[left] {$10^{-3}$};
\draw (-2.5,-1.43) node[left] {$10^{-4}$};
\draw (1.55,1.5) node[right] {\footnotesize $d^n(1)$};
\draw (1.55,1.06) node[right] {\footnotesize $d^n(9)$};
\draw (1.55,0.62) node[right] {\footnotesize $d^n(81)$};
\draw (-1,0.17) node {$-1$};
\end{tikzpicture}
\caption{The incremental error $d^n(i)$ decays as $O(n^{-1})$, independent of $i$. The graphs of $n \mapsto d^n(i)$ for $i = 3, 27$ are visibly indistinguishable from those for $i = 9, 81$, and are therefore omitted.}
\label{fig:PS:bdd:Vwall}
\end{figure}
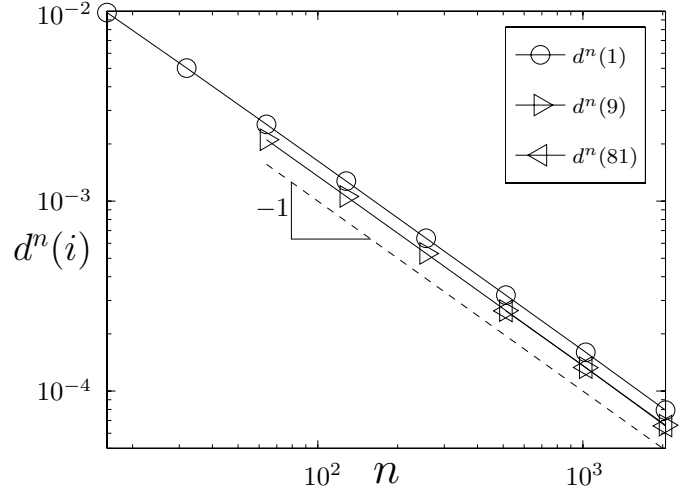

Figure \ref{fig:x:y:Vwall} is the counterpart of Figure \ref{fig:x:y:a2}. Compared to Figure \ref{fig:x:y:a2}, the boundary-layer profile is different, and the offset between the minimiser and the predictor is smaller.

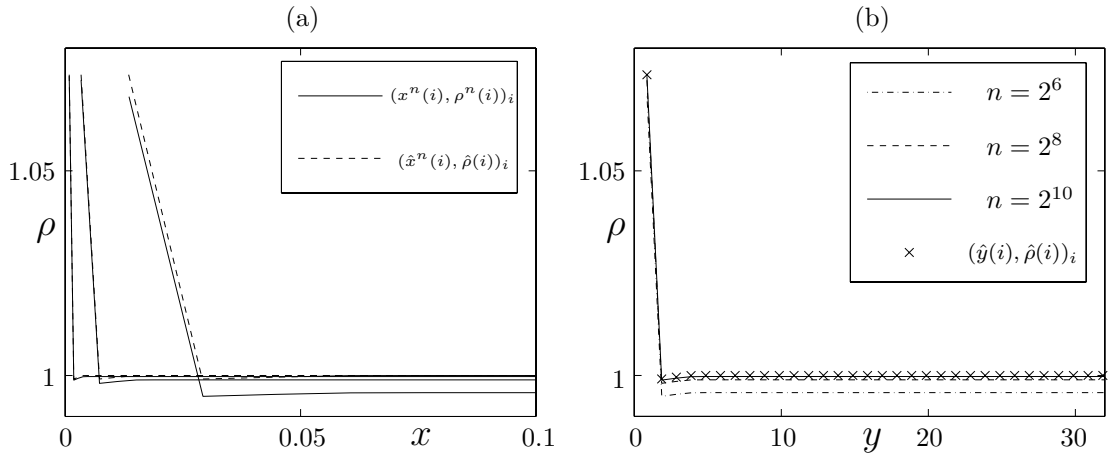
\begin{figure}[h]
\centering
\begin{tikzpicture}[scale=1.25]
\node (label) at (0,0){\includegraphics[width=2.5in]{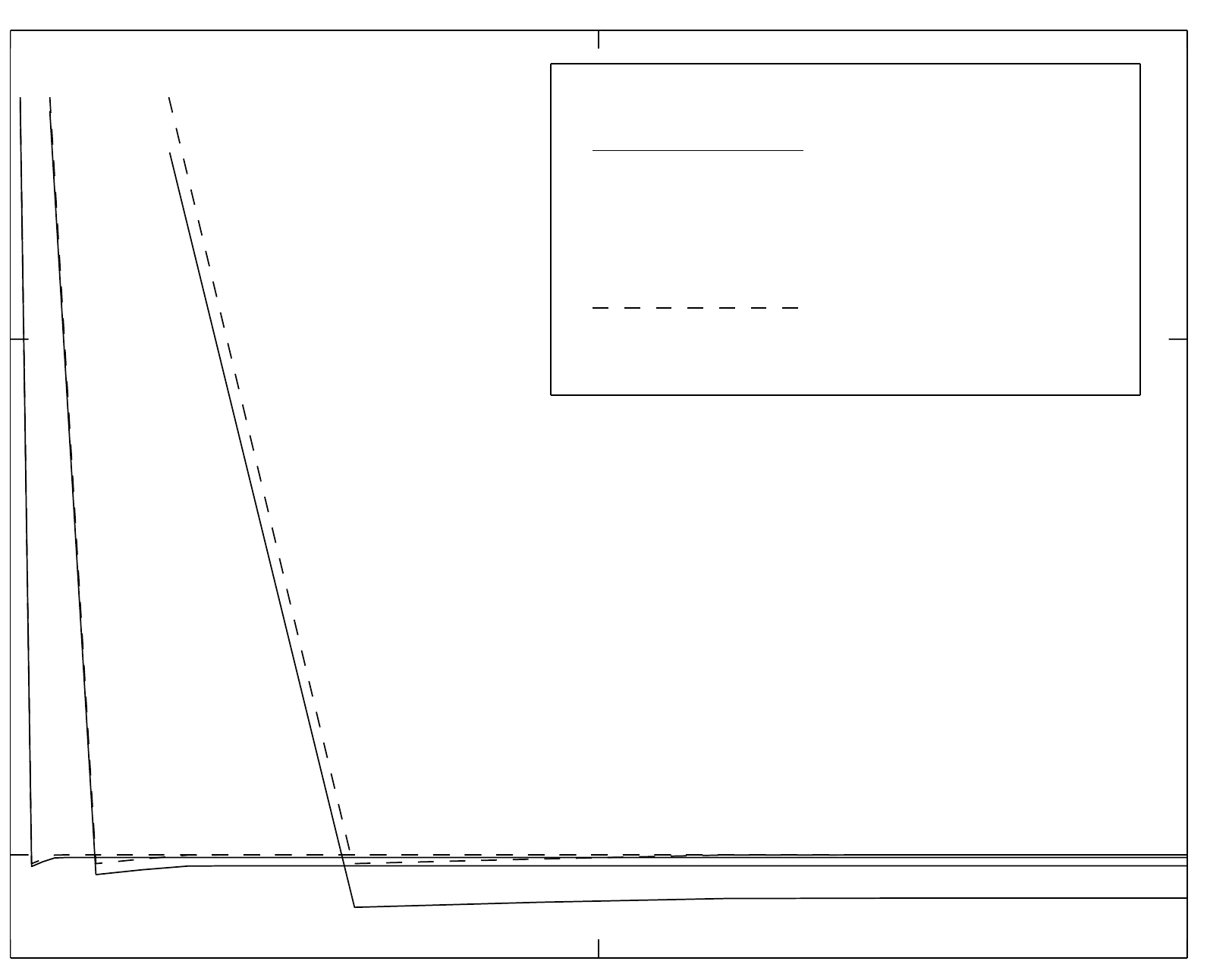}};
\draw (0,2) node[above] {(a)};
\draw (1.25,-2) node[below] {\LARGE $x$};
\draw (-2.5,0) node[left] {\LARGE $\rho$};
\draw (-2.49,-2) node[below] {$0$};
\draw (0,-2) node[below] {$0.05$};
\draw (2.49,-2) node[below] {$0.1$};
\draw (-2.5,0.63) node[left] {$1.05$};
\draw (-2.5,-1.56) node[left] {$1$};
\draw (0.82, 1.43) node[right] {\tiny $(x^n(i), \rho^n(i))_i$};
\draw (0.9, 0.73) node[right] {\tiny $(\hat x^n(i), \hat \rho(i))_i$};

\begin{scope}[shift={(6,0)}]
\node (label) at (0,0){\includegraphics[width=2.51in]{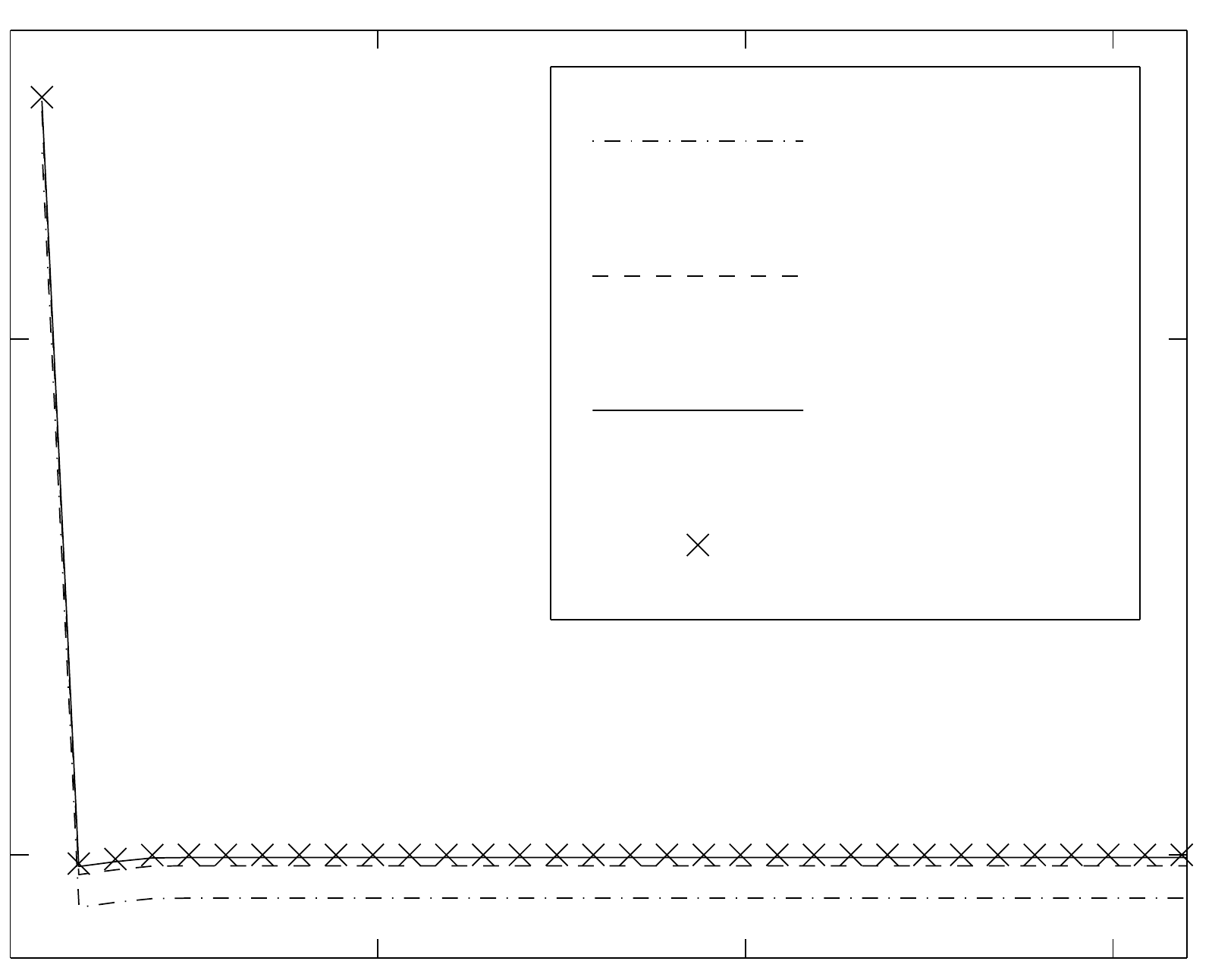}};
\draw (0,2) node[above] {(b)};
\draw (0,-2) node[below] {\LARGE $y$};
\draw (-2.5,0) node[left] {\LARGE $\rho$};
\draw (-2.5,0.63) node[left] {$1.05$};
\draw (-2.5,-1.58) node[left] {$1$};
\draw (-0.94,-2) node[below] {$10$};
\draw (-2.49,-2) node[below] {$0$};
\draw (0.6,-2) node[below] {$20$};
\draw (2.12,-2) node[below] {$30$};
\draw (1.1,1.5) node[right] {$n = 2^6$};
\draw (1.1,0.92) node[right] {$n = 2^8$};
\draw (1.1,0.35) node[right] {$n = 2^{10}$};
\draw (0.9,-0.25) node[right] {\footnotesize $(\hat y(i), \hat \rho(i))_i$};
\end{scope}
\end{tikzpicture}
\caption{Comparison between: (a) the minimiser $\rho^n(i)$ and the predictor $\hat \rho(i)$, and (b) their horizontally rescaled counterparts by a factor $n$. Figure \ref{fig:x:y:a2} provides a more precisely descriptions of the line graphs. We use $n = 2^6, 2^8, 2^{10}$.}
\label{fig:x:y:Vwall}
\end{figure}

\section*{Acknowledgements}
\noindent
  \emph{Thanks:} The authors would like to thank Mark Peletier for 
  valuable discussions, and TU Eindhoven for providing funds to 
  cover research visits by TH and CH.
  TH and PvM would also like to thank the Hausdorff Research 
  Institute for Mathematics in Bonn for hosting them during the
  junior workshop `Analytic approaches to scaling limits for random
  systems' during which work on this project was carried out.
  \medskip

  \noindent
  \emph{Funding:} The work of TH is funded by a public grant overseen by the 
  French National Research Agency (ANR) as part of the 
  ``Investissements d'Avenir'' program (reference:
  ANR-10-LABX-0098).
  
  The work of PvM is partially funded by NWO Complexity grant 645.000.012, and partially by the International Research Fellowship of the Japanese Society for the Promotion of Science, together with the JSPS KAKENHI grant 15F15019.
\medskip

\noindent
\emph{Conflict of interest:} The authors declare that there is no conflict 
of interest regarding this work.

%\bibliographystyle{alpha}
%\bibliography{DiscBLs}

\begin{thebibliography}{GvMPS15}

\bibitem[BBH12]{BethuelBrezisHelein}
Fabrice Bethuel, Ha{\"\i}m Brezis, and Fr{\'e}d{\'e}ric H{\'e}lein.
\newblock {\em Ginzburg-Landau Vortices}, volume~13.
\newblock Springer Science \& Business Media, 2012.

\bibitem[BC07]{BraidesCicalese2007}
Andrea Braides and Marco Cicalese.
\newblock Surface energies in nonconvex discrete systems.
\newblock {\em Math. Models Methods Appl. Sci.}, 17(7):985--1037, 2007.

\bibitem[BDMG99]{BraidesDalMasoGarroni1999}
Andrea Braides, Gianni Dal~Maso, and Adriana Garroni.
\newblock Variational formulation of softening phenomena in fracture mechanics:
  the one-dimensional case.
\newblock {\em Arch. Ration. Mech. Anal.}, 146(1):23--58, 1999.

\bibitem[BLBL07]{BLBL07}
Xavier Blanc, Claude Le~Bris, and Pierre-Louis Lions.
\newblock Atomistic to continuum limits for computational materials science.
\newblock {\em M2AN Math. Model. Numer. Anal.}, 41(2):391--426, 2007.

\bibitem[Bra02]{Braides02}
Andrea Braides.
\newblock {\em {$\Gamma$}-convergence for beginners}, volume~22 of {\em Oxford
  Lecture Series in Mathematics and its Applications}.
\newblock Oxford University Press, Oxford, 2002.

\bibitem[CKP05]{CKP}
G.~F. Carrier, M.~Krook, and C.~E. Pearson.
\newblock {\em Functions of a Complex Variable -- Theory and Technique}.
\newblock SIAM, 2005.

\bibitem[Com74]{ComtetComb}
L.~Comtet.
\newblock {\em Advanced Combinatorics: The art of finite and infinite
  expansions}.
\newblock D. Reidel Publishing COmpany, 1974.

\bibitem[DM93]{DalMaso93}
Gianni Dal~Maso.
\newblock {\em An introduction to {$\Gamma$}-convergence}.
\newblock Progress in Nonlinear Differential Equations and their Applications,
  8. Birkh\"auser Boston, Inc., Boston, MA, 1993.

\bibitem[Dys62]{Dyson1962}
Freeman~J Dyson.
\newblock Statistical theory of the energy levels of complex systems. i.
\newblock {\em Journal of Mathematical Physics}, 3(1):140--156, 1962.

\bibitem[Gla80]{G80}
Graham~ML Gladwell.
\newblock {\em Contact problems in the classical theory of elasticity}.
\newblock Springer Science \& Business Media, 1980.

\bibitem[GPPS13]{Geers2013}
M~G~D Geers, R~H~J Peerlings, M~A Peletier, and L~Scardia.
\newblock Asymptotic behaviour of a pile-up of infinite walls of edge
  dislocations.
\newblock {\em Archive for Rational Mechanics and Analysis}, 209:495--539,
  2013.

\bibitem[GvMPS15]{GarronivanMeursPeletierScardia14ArXiv}
A.~Garroni, P.~{v}an Meurs, M.~A. Peletier, and L.~Scardia.
\newblock Boundary-layer analysis for a pile-up of walls of edge dislocations
  at a lock.
\newblock {\em ArXiv: 1502.05805}, 2015.

\bibitem[Hal10]{Hall2010a}
C.~L. Hall.
\newblock Asymptotic expressions for the nearest and furthest dislocations in a
  pile-up against a grain boundary.
\newblock {\em Philosophical Magazine}, 90(29):3879--3890, 2010.

\bibitem[Hal11]{Hall2011}
C.~L. Hall.
\newblock Asymptotic analysis of a pile-up of regular edge dislocation walls.
\newblock {\em Materials Sceince and Engineering A}, 530:144--148, 2011.

\bibitem[HCO10]{Hall2010}
C.~L. Hall, S.~J. Chapman, and J.~R. Ockendon.
\newblock Asymptotic analysis of a system of algebraic equations arising in
  dislocation theory.
\newblock {\em SIAM Journal on Applied Mathematics}, 70(7):2729--2749, 2010.

\bibitem[Hin91]{HinchPert}
E.~J. Hinch.
\newblock {\em Perturbation Methods}.
\newblock Cambridge University Press, 1991.

\bibitem[HL82]{HirthLothe}
John~Price Hirth and Jens Lothe.
\newblock {\em Theory of Dislocations}.
\newblock John Wiley \& Sons, 2nd edition, 1982.

\bibitem[Hud13]{Hudson2013}
Thomas Hudson.
\newblock Gamma-expansion for a 1{D} confined {L}ennard-{J}ones model with
  point defect.
\newblock {\em Netw. Heterog. Media}, 8(2):501--527, 2013.

\bibitem[IRM{\etalchar{+}}13]{Ivanov2013}
V.~A. Ivanov, A.~S. Rodionova, J.~A. Martemyanova, M.~R. Stukan, M.~M{\"u}ller,
  W.~Paul, and K.~Binder.
\newblock Wall-induced orientational order in athermal semidilute solutions of
  semiflexible polymers: {M}onte {C}arlo simulations of a lattice model.
\newblock {\em J Chem Phys}, 138:234903, 2013.

\bibitem[Jon24]{LennardJones24}
J.~E. Jones.
\newblock On the determination of molecular fields. ii. from the equation of
  state of a gas.
\newblock {\em Proceedings of the Royal Society of London A: Mathematical,
  Physical and Engineering Sciences}, 106(738):463--477, 1924.

\bibitem[Lyn93]{Lyness1993}
J.~N. Lyness.
\newblock {\em Approximation and Computation: A Festschrift in Honor of Walter
  Gautschi}, volume 119, chapter Finite-part integrals and the Euler--Maclaurin
  expansion, pages 297--407.
\newblock Birkh{\"a}user Boston, 1993.

\bibitem[ML98]{Monegato1998}
G.~Monegato and J.~N. Lyness.
\newblock The euler--maclaurin expansion and finite-part integrals.
\newblock {\em Numerische Mathematik}, 81:273--291, 1998.

\bibitem[MRR{\etalchar{+}}53]{MH53}
Nicholas Metropolis, Arianna~W Rosenbluth, Marshall~N Rosenbluth, Augusta~H
  Teller, and Edward Teller.
\newblock Equation of state calculations by fast computing machines.
\newblock {\em The {J}ournal of {C}hemical {P}hysics}, 21(6):1087--1092, 1953.

\bibitem[NvdB15]{Nussinov2015}
Z.~Nussinov and J.~van~den Brink.
\newblock Compass models: {T}heory and physical motivations.
\newblock {\em Reviews of Modern Physics}, 87:1--59, 2015.

\bibitem[PS14]{PetracheSerfaty2014}
Mircea Petrache and Sylvia Serfaty.
\newblock Next order asymptotics and renormalized energy for riesz
  interactions.
\newblock {\em Journal of the Institute of Mathematics of Jussieu}, pages
  1--69, 2014.

\bibitem[RS06]{RS06}
Giovanni Russo and Peter Smereka.
\newblock Computation of strained epitaxial growth in three dimensions by
  kinetic monte carlo.
\newblock {\em Journal of Computational Physics}, 214(2):809--828, 2006.

\bibitem[Sid12]{Sidi2012}
A.~Sidi.
\newblock {E}uler--{M}aclaurin expansions for integrals with arbitrary
  algebraic-logarithmic endpoint singularities.
\newblock {\em Constructive Approximation}, 36:331--352, 2012.

\bibitem[SSZ11]{ScardiaSchloemerkemperZanini2011}
Lucia Scardia, Anja Schl{\"o}merkemper, and Chiara Zanini.
\newblock Boundary layer energies for nonconvex discrete systems.
\newblock {\em Math. Models Methods Appl. Sci.}, 21(4):777--817, 2011.

\bibitem[TT14]{TT14}
John~Meurig Thomas and W~John Thomas.
\newblock {\em Principles and practice of heterogeneous catalysis}.
\newblock John Wiley \& Sons, 2014.

\bibitem[VCMO09]{Voskoboinikov2009}
R.~E. Voskoboinikov, S.~J. Chapman, J.~B. Mcleod, and J.~R. Ockendon.
\newblock Asymptotics of edge dislocation pile-up against a bimetallic
  interface.
\newblock {\em Mathematics and Mechanics of Solids}, 14(1-2):284 -- 295, 2009.

\bibitem[vMMP14]{VanMeursMunteanPeletier14}
P.~{v}an Meurs, A.~Muntean, and M.~A. Peletier.
\newblock Upscaling of dislocation walls in finite domains.
\newblock {\em European Journal of Applied Mathematics}, 25(6):749--781, 9
  2014.

\bibitem[Wig55]{Wigner1955}
Eugene~P Wigner.
\newblock Characteristic vectors of bordered matrices with infinite dimensions.
\newblock {\em Annals of Mathematics}, pages 548--564, 1955.

\bibitem[WMHL13]{Wennberg2013}
C.~L. Wennberg, T.~Murtola, B.~Hess, and E.~Lindahl.
\newblock {L}ennard--{J}ones lattice summation in bilayer simulations has
  critical effects on surface tension and lipid properties.
\newblock {\em J Chem Theory Comput}, 9(8):3527--3537, 2013.

\bibitem[WW09]{Wang2009}
W.~Wang and T.~Wang.
\newblock General identities on {B}ell polynomials.
\newblock {\em Computers and Mathematics with Applications}, 2009.

\bibitem[ZV15]{Zschocke2015}
F.~Zschocke and M.~Vojta.
\newblock Physical states and finite-size effects in {K}itaev's honeycomb
  model: {B}ond disorder, spin excitations, and {NMR} line shape.
\newblock {\em Phys. Rev. B}, 92:014403, 2015.

\end{thebibliography}

\newcommand{\etalchar}[1]{$^{#1}$}

\end{document}